\documentclass[10pt]{article}
\usepackage[british]{babel}
\usepackage{amsmath,amssymb,stmaryrd,enumerate,bbm,xcolor,latexsym,theorem}
\usepackage[a4paper]{geometry}

\usepackage{bm}  

\usepackage[pdftex,bookmarks,colorlinks,breaklinks]{hyperref}  
\definecolor{dullmagenta}{rgb}{0.4,0,0.4}   
\definecolor{darkblue}{rgb}{0,0,0.4}
\hypersetup{linkcolor=red,citecolor=blue,filecolor=dullmagenta,urlcolor=darkblue} 

\usepackage{memhfixc}  
\usepackage{pdfsync}  

\catcode`\>=\active \def>{
\fontencoding{T1}\selectfont\symbol{62}\fontencoding{\encodingdefault}}
\newcommand{\assign}{:=}
\newcommand{\backassign}{=:}
\newcommand{\comma}{{,}}
\newcommand{\mathd}{\mathrm{d}}
\newcommand{\nobracket}{}
\newcommand{\tmaffiliation}[1]{\\ #1}
\newcommand{\tmcolor}[2]{{\color{#1}{#2}}}
\newcommand{\tmem}[1]{{\em #1\/}}
\newcommand{\tmemail}[1]{\\ \textit{Email:} \texttt{#1}}
\newcommand{\tmop}[1]{\ensuremath{\operatorname{#1}}}
\newcommand{\tmscript}[1]{\text{\scriptsize{$#1$}}}
\newcommand{\tmstrong}[1]{\textbf{#1}}
\newcommand{\tmtextbf}[1]{{\bfseries{#1}}}
\newcommand{\tmtextit}[1]{{\itshape{#1}}}
\newenvironment{enumeratealpha}{\begin{enumerate}[a{\textup{)}}] }{\end{enumerate}}
\newenvironment{enumeratenumeric}{\begin{enumerate}[1.] }{\end{enumerate}}
\newenvironment{enumerateroman}{\begin{enumerate}[i.] }{\end{enumerate}}
\newenvironment{proof}{\noindent\textbf{Proof\ }}{\hspace*{\fill}$\Box$\medskip}
\newenvironment{proof*}[1]{\noindent\textbf{#1\ }}{\hspace*{\fill}$\Box$\medskip}
\newenvironment{quoteenv}{\begin{quote} }{\end{quote}}
\newtheorem{corollary}{Corollary}
\newtheorem{definition}{Definition}
\newtheorem{lemma}{Lemma}
\newtheorem{proposition}{Proposition}
{\theorembodyfont{\rmfamily}\newtheorem{remark}{Remark}}
\newtheorem{theorem}{Theorem}
\newcommand{\tmkeywords}{\textbf{Keywords:} }

\begin{document}

\title{Noiseless regularisation by noise}

\author{
  L.~Galeati \& M.~Gubinelli
  \tmaffiliation{Institute of Applied Mathematics \&\\
  Hausdorff Center for Mathematics\\
  University of Bonn\\
  Germany}
  \tmemail{\{lucio.galeati,gubinelli\}@iam.uni-bonn.de}
}

\date{}

\maketitle

\begin{abstract}
  We analyse the effect of a \tmtextit{generic} continuous additive
  perturbation to the well-posedness of ordinary differential equations.
  Genericity here is understood in the sense of prevalence. This allows us to
  discuss these problems in a setting where we do not have to commit ourselves
  to any restrictive assumption on the statistical properties of the
  perturbation. The main result is that a generic continuous perturbation
  renders the Cauchy problem well-posed for arbitrarily irregular vector
  fields. Therefore we establish regularisation by noise ``without
  probability''.
  
  \
  
  {\noindent}\tmtextbf{MSC(2020):} Primary: 60H50. Secondary: 37C20.
\end{abstract}

\tmkeywords{Regularisation by noise, Prevalence.}

{\tableofcontents}

\section{Introduction}\label{sec1}

From the modelling point of view, the presence of external perturbations to
otherwise autonomous evolutions is a very natural assumption. Let $d \in
\mathbb{N}$ and consider the ODE in $\mathbb{R}^d$
\begin{equation}
  \left\{\begin{array}{l}
    \dot{x} (t) = b (t, x (t)) + \dot{w} (t)\\
    x (0) = x_0 \in \mathbb{R}^d
  \end{array}\right., \qquad t \geqslant 0, \label{eq:pert-ode}
\end{equation}
where $w \in C (\mathbb{R}_+, \mathbb{R}^d)$ is a fixed perturbation, the dot
denotes differentiation with respect to time and $b : \mathbb{R}_+ \times
\mathbb{R}^d \rightarrow \mathbb{R}^d$ is a time-dependent vector field.
Provided eq.~{\eqref{eq:pert-ode}} is understood as an integral equation and
thanks to the additive nature of the perturbation, there are no particular
regularity requirements, apart from continuity, which have to be imposed on
the function $w$. A natural question is then for which classes of vector
fields \ $b$ eq.~{\eqref{eq:pert-ode}} is well-posed and if, for certain sets
of perturbations $w$, one can obtain well-posedness results in classes which
are known to lead to an ill-posed problem when $w = 0$.

\

One possible approach to this problem is to consider $w$ a sample path of a
stochastic process $W$ defined on a probability space $(\Omega, \mathcal{F},
\mathbb{P})$. Indeed, in recent years there has been a lot of activity in
understanding the possible role of random perturbations to improve the
well-posedness of ordinary (or partial) differential equations (ODE/PDE)
(see~{\cite{flandoli}} for a recent review). This approach has, however,
certain limitations:
\begin{enumeratealpha}
  \item It requires to make very specific assumptions on the kind of
  randomness which is allowed in any specific problem.
  
  \item It introduces into the picture considerations which are not quite
  germane to the initial formulation. For example measurability (or
  adaptedness) wrt. $\Omega$ of solutions as soon as we need to look at them
  in the sense of stochastic processes (i.e. seen as random variables) and
  weaker notions of uniqueness which are not easy to compare to the
  deterministic setting.
\end{enumeratealpha}
With respect to point a) one can use other assumptions to justify specific
choices. Within the class of time-dependent continuous random processes, for
example, \ Brownian motion has suitable features of universality and
Markovianity, making it a natural choice. Furthermore, a large set of
theoretical tools is available to analyse the effect of Brownian perturbations
to deterministic evolutions and this topic has a long and extensive
literature~{\cite{BFGM,dapratoflandoli,DFRV,FGP,krylov,veretennikov,zvonkin}}.
Other classes of random perturbations, like fractional Brownian motion (fBm)
have been more recently analysed, or even more exotic variants (e.g.
$\alpha$-stable and log regular
processes)~{\cite{banosproske,banos2019,le,nualartouknine,priola2012,priola2020}}.
Let us finally mention the remarkable results from~{\cite{butkovsky2019}}
concerning rates of convergence of numerical schemes
for~{\eqref{eq:pert-ode}}.

\

As for the technical limitations in point b), a possible solution is to
modify the probabilistic setting in order to derive path-wise statements:
\begin{enumerateroman}
  \item Davie and Flandoli introduced a stronger concept of uniqueness called
  \tmtextit{path-by-path uniqueness}~{\cite{davie,flandoli,flandoli2011}} in
  the Brownian setting; see also~{\cite{shaposhnikov2020}} regarding its
  difference from pathwise uniqueness.
  
  \item Catellier and one of the authors~{\cite{catelliergubinelli}} studied
  almost sure regularisation properties of fractional Brownian motion (fBm)
  and applied them to show strong well-posedness results
  for~{\eqref{eq:pert-ode}} when $w$ is a sample path of an fBm.
\end{enumerateroman}

In this work we take a conceptually different approach and consider the
regularisation by noise problem from the point of view of \tmtextit{generic}
perturbations, in particular without reference to any (specific) probabilistic
setting.

\

We will say that a property holds for \tmtextit{almost every} path $w$ if it
holds for a \tmtextit{prevalent} set of paths. Prevalence~{\cite{ott}} is a
notion of ``Lebesgue measure zero sets'' in infinite dimensional complete
metric vector spaces. Such sets cannot be naively defined due to the fact that
there cannot exist $\sigma$-additive, translation invariant measures in
infinite dimensional spaces. It was first introduced by Christensen
in~{\cite{christensen}} in the context of abelian Polish groups and later
rediscovered independently by Hunt, Sauer and Yorke in~{\cite{huntsauer}} for
complete metric vector spaces.

\

Prevalence has been used in different contexts in order to study the
properties of generic functions belonging to spaces of suitable regularity.
For instance, it was proved in~{\cite{hunt}} that almost every continuous
function is nowhere differentiable, while in~{\cite{fraysse2010,fraysse2006}}
the multi-fractal nature of generic Sobolev functions was shown. Recently,
prevalence has also attracted a lot of attention in the study of dimension of
graphs and images of continuous functions, see among
others~{\cite{fraserhyde,bayart}}.

\

A key advantage of prevalence, with respect to other notions of genericity,
is that it allows the use of probabilistic methods in the proof. However the
statements are fully non-probabilistic and the kind of problems one encounters
in formulating prevalence results are quite distinct from those of a purely
probabilistic setting, extensively investigated in the probabilistic
literature.

\

Armed with this ``noiseless'' notion of ``almost every path'', we can already
state informally one of the results of the paper as follows:

\begin{quoteenv}
  Let $b \in C ([0, T] ; H^{- s} (\mathbb{R}^d))$ be fixed, $s > 0$
  arbitrarily large. Then \tmtextit{almost every} perturbation $w \in C ([0,
  T] ; \mathbb{R}^d)$ has \tmtextit{infinite} regularisation effect on the ODE
  associated to $b$, namely it renders the ODE~{\eqref{eq:pert-ode}}
  well-posed and with a smooth flow.
\end{quoteenv}

\

In order to proceed and precise the above claims we will need a suitable
notion of solution to~{\eqref{eq:pert-ode}} which makes sense for
distributional fields $b$. The key observation in this direction comes from
the work~{\cite{catelliergubinelli}}, which started the study of analytic
properties of paths which affects the regularisation of ODEs.

\

In particular, the work~{\cite{catelliergubinelli}} introduces the
\tmtextit{averaging operator} $T$ as a tool to study the regularisation
properties of a path $w$. It is the operator acting on time-dependent vector
fields $b$ and paths $w$ as as
\[ (w, b) \mapsto (T^w b) (t, x) = \int_0^t b (s, x + w (s)) \mathd s, \qquad
   x \in \mathbb{R}^d, t \geqslant 0. \]
It is a linear operator in $b$, so that one can fix $w$ and consider the
operator $T^w : b \mapsto T^w b$ as above; in this case we say that $T^w$ is
the averaging operator associated to $w$. Alternatively, one can fix $b$ and
vary $w$, $w \mapsto T^w b$; to stress the latter case, we say that $T^w b$ is
an \tmtextit{averaged field}.

\

Averaging is connected to an alternative formulation of the ODE via the
theory of non-linear Young integration. Assume for the moment $b : [0, T]
\times \mathbb{R}^d \rightarrow \mathbb{R}^d$ smooth and consider the
ODE~{\eqref{eq:pert-ode}} in integral form
\begin{equation}
  x_t = x_0 + \int_0^t b (s, x_s) \mathd s + w_t \label{eq:integral-cauchy},
  \qquad t \in [0, T]
\end{equation}
with $w \in C ([0, T] ; \mathbb{R}^d)$. Then, this equation admits a unique
solution of the form $x = w + C^1$, in the sense that the difference $x - w$
is a $C^1$ path, regardless the regularity of $w$. Applying the change of
variables $\theta \assign x - w$ we get the new integral equation
\begin{equation}
  \theta_t = \theta_0 + \int_0^t b (s, \theta_s + w_s) \mathd s, \qquad t \in
  [0, T] . \label{eq:integral-cauchy-theta}
\end{equation}
Since both $b$ and $\theta$ are continuous, the last integral can be
approximated via Riemann--Stieltjes type sums as follows
\begin{equation}
  \int_0^t b (s, \theta_s + w_s) \mathd s = \lim_{| \Pi | \rightarrow 0}
  \sum_i \int_{t_i}^{t_{i + 1}} b (s, \theta_{t_i} + w_s) \mathd s = \lim_{|
  \Pi | \rightarrow 0} \sum_i T^w b_{t_i, t_{i + 1}} (\theta_{t_i})
  \label{sec4.1 eq1}
\end{equation}
where the limit is taken over all possible partitions $\Pi = \{ t_0, \ldots,
t_n \}$ with $0 = t_0 < t_1 < \cdots < t_n = t$ with mesh $| \Pi | = \sup_i |
t_{i + 1} - t_i |$ converging to $0$ and where for a function $A = A (t, x)$
we adopt the compact notation $A_{s, t} (x) \assign A (t, x) - A (s, x)$. The
r.h.s. of equation~{\eqref{sec4.1 eq1}} depends now on the averaged field $T^w
b$. The key observation of~{\cite{catelliergubinelli}} is that, under suitable
space-time regularity conditions on $T^w b$, it is possible to show
convergence of the above Riemann--Stieltjes type sums to a unique limit even
when $b$ is not continuous anymore, thus allowing to \tmtextit{define} the
integral on the l.h.s. of~{\eqref{sec4.1 eq1}} as their limit. This limit is
called in~{\cite{catelliergubinelli}} a \tmtextit{non-linear Young integral}
and denoted as
\[ \int_0^t T^w b (\mathd s, \theta_s) . \]
Eq.~{\eqref{eq:integral-cauchy-theta}} takes then the form of an integral
equation involving non-linear Young integrals:
\[ \theta_t = \theta_0 + \int_0^t T^w b (\mathd s, \theta_s) . \]
The analysis of such equations (existence, uniqueness, regularity of the flow)
for irregular $b$ depends essentially on the regularity properties of the
averaged field $T^w b$ and a substantial part of the present paper will be
dedicated to analyse them in detail. For example we will prove that:

\begin{quoteenv}
  Let $b \in C ([0, T] ; H^{- s} (\mathbb{R}^d))$ be fixed, $s > 0$
  arbitrarily large. Then \tmtextit{almost every} perturbation $w \in C ([0,
  T] ; \mathbb{R}^d)$ has \tmtextit{infinite} regularisation effect on $b$,
  namely $T^w b \in C ([0, T] ; C^{\infty})$.
\end{quoteenv}

A quantitative version of the statements above, which collects some of the
main results of this paper, is the following one.

\begin{theorem}
  \label{intro_main_thm}Let $b \in B^{\alpha}_{\infty, \infty}$ be a compactly
  supported distribution, $\alpha \in (- \infty, 1)$, $\delta \in (0, 1)$.
  \begin{enumerateroman}
    \item If $\delta < (2 - 2 \alpha)^{- 1}$, then for a.e. $w \in
    C_t^{\delta}$ it holds $T^w b \in C^{\gamma}_t C^1_x$ and
    ODE~{\eqref{eq:integral-cauchy}} has a meaningful interpretation; moreover
    for any initial $x_0 \in \mathbb{R}^d$ there exists a solution to the ODE.
    
    \item If $\delta < (2 - 2 \alpha)^{- 1}$ and we fix $x_0 \in
    \mathbb{R}^d$, then for a.e. $w \in C_t^{\delta}$ there exists a unique
    solution to the ODE with initial condition $x_0$.
    
    \item If $\delta < (4 - 2 \alpha)^{- 1}$ then for a.e. $w \in
    C_t^{\delta}$ the ODE is well posed and it admits a locally $C^1$ flow.
    
    \item If $\delta < (2 n - 2 \alpha)^{- 1}$, then for a.e. $w \in
    C_t^{\delta}$ the flow is locally $C^{n - 1}$.
    
    \item Finally, for a.e. $w \in C^0$ the ODE admits a smooth flow.
  \end{enumerateroman}
\end{theorem}

\begin{remark}
  In this theorem we could allow time dependent $b \in L^{\infty}_t
  B^{\alpha}_{\infty, \infty}$ provided $\delta < 1 / 2$. This is due to some
  technical limitations in the proof technique.
\end{remark}

Let us point out that this results is the first general statement which
supports the heuristics ``the rougher the noise, the better the
regularisation'' observed in the probabilistic literature since
e.g.~{\cite{catelliergubinelli}} but so far never discussed abstracting from a
particular probabilistic model of the perturbation.

\

We conclude this introduction by discussing possible extensions are relations
with related work. The averaging operator $T^w$ is, in many respect, a key
tool introduced in~{\cite{catelliergubinelli}} to study analytically the
regularisation properties of perturbations in dynamical problems. In this
paper we refrain to investigate more thoroughly this operator from the point
of view of prevalence since this will be the main objective of the companion
paper~{\cite{galeatigubinelli}}. There we continue the study of the prevalent
properties of path which are associated to the regularisation by noise
phenomenon by concentrating on the notion of $\rho$\mbox{-}irregularity of a
path, as introduced in~{\cite{catelliergubinelli}}, and the related notion of
occupation measure, obtaining as a by-product information on the prevalent
properties of $T^w$.

The setting we propose in this paper opens up a completely new research
subject with many natural problems, one prominent among them is to investigate
the \tmtextit{zero noise limit}, that is the limit as $\varepsilon \rightarrow
0$ for solutions to the equation $\dot{x}_{\varepsilon} = b (x_{\varepsilon})
+ \varepsilon w$. Already in the probabilistic setting this limit is not well
understood, especially from the path-wise perspective and the dependence of
the limit on the law assumed for $w$ is not clear.

On a more technical level several improvement of our results could be
possible. For example it would be interesting to obtain estimates for the
averaging in $L^p$\mbox{-}based spaces with $p \in [1, 2)$, see
Remark~\ref{sec3.3 remark limitation} below and the related discussion in
Appendix~\ref{appendixA3}. In particular let us note that the natural
Conjecture~1.2 from~{\cite{catelliergubinelli}} is still partially open; after
the first draft of this work appeared, Nicolas Perkowski presented us a proof
that answers negatively the conjecture in the case $H < 1 / d$ for general $d
\geqslant 3$ and $H \leqslant 1 / 2$ for $d = 2$.

\

While we were finalizing the present paper, two related preprints appeared.
Harang and Perkowski~{\cite{perkowski}} \ study the flow of the
ODE~{\eqref{eq:pert-ode}} perturbed with a Gaussian process very similar to
that considered in~{\cite{banosproske}} but from the pathwise point of view
of~{\cite{catelliergubinelli}}. Along the way they give proofs of some results
on the flow of Young differential equations alternative to those we give
below. In~{\cite{amine2020}} Amine, Mansouri and Proske study with techniques
very different from ours, the path-by-path uniqueness for transport equations
driven by fBm with Hurst index $H < 1 / 2$ and with bounded vector-fields. It
is to be noted that while both works obtain interesting results, they still
consider very specific probabilistic models. Therefore they are both far from
the novel point of view we propose here and in the companion
paper~{\cite{galeatigubinelli}} and from the specific results it generates.

Let us finally mention the very recent work~{\cite{gerencser2020}} in which
Gerencser provides instances of regularisation by noise for $w \in C^{\delta}$
with $\delta > 1$.

\

\tmtextbf{Structure of the paper.} We start by introducing the concept
prevalence and its basic properties. Section~\ref{sec3} is devoted to the
study of prevalence statements for averaged fields. Fractional Brownian motion
(fBm) enters into the picture as a suitable transverse measure for prevalence.
Thanks to a functional Ito--Tanaka type formula, we deduce regularity
estimates for distributions averaged by fBm, which are strong enough to lead
to prevalence statements. Section~\ref{sec4} is devoted to the application of
the results from the previous section to perturbed ODEs via the theory of
nonlinear Young integrals. After recalling and expanding the results
from~{\cite{catelliergubinelli}}, we provide conditions (in terms of the
regularity of $T^w b$) under the ODE admits a flow with prescribed regularity.
Combined with Section~\ref{sec3}, this allows to prove
Theorem~\ref{intro_main_thm}. Finally, we consider the case of perturbed
transport type PDEs, for which it is again possible to establish
well-posedness under suitable regularity conditions on $T^w b$. We choose to
put in the Appendix reminders of standard facts and certain technical results.

\

\tmtextbf{Acknowledgments.} We thank Mark Veraar and Simone Floreani for a
very useful discussion on integration in UMD Banach spaces.

\

{\tmstrong{Notation.}} We will use the notation $a \lesssim b$ to mean that
there exists a positive constant $c$ such that $a \leqslant c b$; we use the
index $a \lesssim_x b$ to highlight the dependence $c = c (x)$. $a \sim b$ if
and only if $a \lesssim b$ and $b \lesssim a$, similarly for $a \sim_x b$.

We will always work on a finite time interval $[0, T]$ unless stated
otherwise. Whenever useful we adopt the convention that $f_t$ stands for $f
(t)$ for a function $f$ indexed on $t \in [0, T]$, but depending on the
context we will use both notations; similarly for the increments of $f_{s, t}
= f_t - f_s$.

For $x \in \mathbb{R}^d$, $| x |$ denotes the Euclidean norm, $x \cdot y$ the
scalar product. For any $R > 0$, $B_R$ stands for $B (0, R) = \{ x \in
\mathbb{R}^d : | x | \leqslant R \}$.

We denote by $\mathcal{S} (\mathbb{R}^d ; \mathbb{R}^m)$ and $\mathcal{S}'
(\mathbb{R}^d ; \mathbb{R}^m)$ respectively the spaces of vector\mbox{-}valued
Schwarz functions and tempered distributions on $\mathbb{R}^d$; similarly
$C^{\infty}_c (\mathbb{R}^d ; \mathbb{R}^m)$ is the set of
vector\mbox{-}valued smooth compactly supported functions.

Given a separable Banach space $E$, we denote by $L^q (0, T ; E) = L^q_t E$
the Bochner--Lebesgue space of $E$--valued measurable functions $f : [0, T]
\rightarrow E$ such that
\[ \| f \|_{L^q (0, T ; E)} = \left( \int_0^T \| f_t \|_E^q \mathd t
   \right)^{1 / q} < \infty, \]
with the essential supremum in the limit case $q = \infty$. $C^{\alpha} ([0,
T] ; E) = C^{\alpha}_t E$ is the space of $E$--valued $\alpha$-H\"{o}lder
continuous functions, for $\alpha \in (0, 1)$, i.e. $f : [0, T] \rightarrow E$
such that
\[ \| f \|_{C^{\alpha} E} \assign \| f \|_{C^0 E} + \llbracket f
   \rrbracket_{C^{\alpha} E} = \sup_{t \in [0, T]} \| f_t \|_E +
   \sup_{\tmscript{\begin{array}{c}
     s \neq t \in [0, T]
   \end{array}}} \frac{\| f_{s, t} \|_E}{| t - s |^{\alpha}} < \infty . \]
A similar definition holds for $\tmop{Lip} ([0, T] ; E) = \tmop{Lip}_t E$.
More generally, for a given modulus of continuity $\omega$ (possibly defined
only in a neighbourhood of $0$), we denote by $C^{\omega} ([0, T] ; E) =
C^{\omega} E$ the set of all $E$-valued continuous functions with modulus of
continuity $\omega$, $\| f \|_{C^{\omega} E}$ and $\llbracket f
\rrbracket_{C^{\omega} E}$ defined as above.

Whenever $E =\mathbb{R}^d$, we will refer to $w \in C^{\alpha}_t = C^{\alpha}
([0, T] ; \mathbb{R}^d)$ as a {\tmem{path}} and in this case we allow $\alpha
\in [0, \infty)$ with the convention that $w \in C_t^{\alpha}$ it is has
continuous derivatives up to order~$\lfloor \alpha \rfloor$ and $D^{\lfloor
\alpha \rfloor} \varphi$ is $\{ \alpha \}$--H\"{o}lder continuous, where
$\lfloor \alpha \rfloor$ and $\{ \alpha \}$ denote respectively integer and
fractional part.

$B^s_{p, q} (\mathbb{R}^d ; \mathbb{R}^m)$, $L^{s, p} (\mathbb{R}^d ;
\mathbb{R}^m)$ and $F^s_{p, q} (\mathbb{R}^d ; \mathbb{R}^m)$ will denote
respectively vector\mbox{-}valued Besov, Bessel potential and
Triebel--Lizorkin spaces (see Appendix~\ref{appendixA2}), $L^p (\mathbb{R}^d ;
\mathbb{R}^m)$ standard Lebesgue spaces. Whenever it doesn't create confusion,
we will just write $B^s_{p, q}$, $L^{s, p}$, $F^s_{p, q}$ and $L^p$ for short.
For $\alpha \in \mathbb{R} \setminus \mathbb{N}_0$, $C^{\alpha} (\mathbb{R}^d
; \mathbb{R}^m) = C^{\alpha}_x = B^{\alpha}_{\infty, \infty}$; instead for
$\alpha \in \mathbb{N}_0$, $C^n (\mathbb{R}^d ; \mathbb{R}^m) = C^n_x$ denotes
the Banach space of all continuous functions with continuous derivatives up to
order $n$, endowed with the norm
\[ \| f \|_{C^{\alpha}} = \sup_{x \in \mathbb{R}^d} | f (x) | + \sum_{\beta
   \in \mathbb{N}_0^n : | \beta | = n} \sup_{x \in \mathbb{R}^d} | D^{\beta} f
   (x) | . \]
Let us stress in particular that by saying that $f \in C^n_x$, we are implying
that we have a uniform bound on the whole $\mathbb{R}^d$ for its derivatives.
If instead we want to say that $f$ has continuous derivatives up to order $n$,
we will write $f \in C^n_{\tmop{loc}}$. We will adopt short-hand notations of
the form $L^q_t L^p_x = L^q (0, T ; L^p (\mathbb{R}^d ; \mathbb{R}^m))$,
$C^{\alpha}_t C^{\beta}_x = C^{\alpha} ([0, T] ; C^{\beta} (\mathbb{R}^d ;
\mathbb{R}^m))$.

Whenever a stochastic process $X = (X_t)_{t \geqslant 0}$ is considered, even
when it is not specified we imply the existence of an abstract underlying
filtered probability space $(\Omega, \mathcal{F}, \{ \mathcal{F}_t \}_{t
\geqslant 0}, \mathbb{P})$ such that $\mathcal{F}$ and $\mathcal{F}_t$ satisfy
the usual assumptions and $X_t$ is adapted to $\mathcal{F}_t$. If
$\mathcal{F}_t$ is said to be the natural filtration generated by $X$, then it
is tacitly implied that it is actually its right continuous, normal
augmentation. We denote by $\mathbb{E}$ integration (equiv. expectation)
w.r.t. the probability $\mathbb{P}$.

\section{Prevalence}\label{sec2}

Here we follow the setting and the terminology given in~{\cite{huntsauer}}
even if, for our purposes, we will be interested only in the case of a Banach
space $E$.

\begin{definition}
  \label{definition prevalence}Let $E$ be a complete metric vector space. A
  Borel set $A \subset E$ is said to be \tmtextit{shy} if there exists a
  measure $\mu$ such that:
  \begin{enumerateroman}
    \item There exists a compact set $K \subset E$ such that $0 < \mu (K) <
    \infty$.
    
    \item For every $v \in E$, $\mu (v + A) = 0$.
  \end{enumerateroman}
  In this case, the measure $\mu$ is said to be transverse to $A$. More
  generally, a subset of $E$ is shy if it is contained in a shy Borel set. The
  complement of a shy set is called a prevalent set.
\end{definition}

Sometimes it is said more informally that the measure $\mu$ ``witnesses'' the
prevalence of $A^c$.

It follows immediately from part \tmtextit{i.} of the definition that, if
needed, one can assume $\mu$ to be a compactly supported probability measure
on $E$. If $E$ is separable, then any probability measure on $E$ is tight and
therefore \tmtextit{i.} is automatically satisfied.

\

The following properties hold for prevalence (all proofs can be found
in~{\cite{huntsauer}}):
\begin{enumeratenumeric}
  \item If $E$ is finite dimensional, then a set $A$ is shy if and only if it
  has zero Lebesgue measure.
  
  \item If $A$ is shy, then so is $v + A$ for any $v \in E$.
  
  \item Prevalent sets are dense.
  
  \item If $\dim (E) = + \infty$, then compact subsets of $E$ are shy.
  
  \item Countable union of shy sets is shy; conversely, countable intersection
  of prevalent sets is prevalent. 
\end{enumeratenumeric}
From now, whenever we say that a statement holds for a.e. $v \in E$, we mean
that the set of elements of $E$ for which the statement holds is a prevalent
set. Property~1. states that this convention is consistent with the finite
dimensional case.

\

In the context of a function space $E$, it is natural to consider as
probability measure the law induced by an $E$\mbox{-}valued stochastic
process. Namely, given a stochastic process $W$ defined on a probability space
$(\Omega, \mathcal{F}, \mathbb{P})$, taking values in a separable Banach space
$E$, in order to show that a property $\mathcal{P}$ holds for a.e. $f \in E$,
it suffices to show that
\[ \mathbb{P} \left( \text{$f + W$ satisfies property $\mathcal{P}$} \right) =
   1, \qquad \forall \, f \in E. \]
Clearly, we are assuming that the set $A = \left\{ w \in E : \text{$w$
satisfies property $\mathcal{P}$} \right\}$ is Borel measurable and if $E$ is
not separable, then we need to require in addition that the law of $W$ is
tight, so as to satisfy point~\tmtextit{i.} of Definition~\ref{definition
prevalence}.

\

As a consequence of properties~4. and~5., the set of all possible
realisations of a probability measure on a separable Banach space is a shy
set, as it is contained in a countable union of compact sets (this is true
more in general for any tight measure on a Banach space). This highlights the
difference between a statement of the form
\[ \text{``Property $\mathcal{P}$ holds for a.e. $f$''} \]
and, for instance,
\[ \text{ ``Property $\mathcal{P}$ holds for all Brownian trajectories''}, \]
where this last statement corresponds to $\mu \left( \text{Property
$\mathcal{P}$ holds} \right) = 1$, $\mu$ being the Wiener measure on $C ([0,
1])$. Indeed, the second statement doesn't provide any information regarding
whether the property might be prevalent or not. Intuitively, the elements
satisfying a prevalence statement are ``many more'' than just the realisations
of the Wiener measure.

\section{Averaging operators}\label{sec3}

We introduce in detail the averaging operator $(w, b) \mapsto T^w b$ and
analyse its prevalent properties in various functional spaces. Fractional
Brownian motion is used as a convenient tranverse measure to detect prevalent
regularisation properties of paths.

\subsection{Definition of averaging operator and basic
properties}\label{sec3.1}

In this section we provide the definition of the averaging operator $T^w$ for
measurable $w : [0, T] \rightarrow \mathbb{R}^d$, together with some basic
properties which will be fundamental for later sections and our first main
prevalence result. Our definition is rather abstract and works for a general
class of Banach spaces $E$, but keep in mind that for our purposes $E$ will
always be either a Bessel space $L^{s, p}$ or a Besov space $B^s_{p, q}$ with
$p \in [2, \infty)$. Also, we consider for simplicity the scalar-valued case,
i.e. $E \subseteq \mathcal{S}' (\mathbb{R}^d)$. Everything generalises
immediately to the vector-valued case $\mathcal{S}' (\mathbb{R}^d ;
\mathbb{R}^m)$ reasoning component by component.

\

Let us assume that $E$ is a separable Banach space that continuously embeds
into $\mathcal{S}' (\mathbb{R}^d)$ (so that there is also a dual embedding
$\mathcal{S} (\mathbb{R}^d) \hookrightarrow E^{\ast}$) such that translation
$\tau^v : f \mapsto \tau^v f = f \left( \cdot \, + v \right)$ act continuously
on it and leave the norm invariant: $\| \tau^v f \|_E = \| f \|_E$ for all $v
\in \mathbb{R}^d$ and $f \in E$. Assume moreover that the map $v \mapsto
\tau^v$ is continuous in the sense that if $v_n \rightarrow v$, then
$\tau^{v_n} f \rightarrow \tau^v f$ for all $f \in E$.

\begin{definition}
  \label{sec3.1 definition averaging}Let $w : [0, T] \rightarrow \mathbb{R}^d$
  be a measurable function, $E$ as above. Then we define the averaging
  operator $T^w$ as the continuous linear map from $L^1 (0, T ; E)$ to $C^0
  ([0, T] ; E)$ given by
  \[ T^w_t b = T^w b (t) \assign \int_0^t \tau^{w_s} b_s \mathd s \quad
     \forall \, t \in [0, T] . \]
  We will refer to $T^w b$ as an averaged function to stress that $b$ is
  fixed, while $w$ might be varying.
\end{definition}

The definition is meaningful, since by the continuity properties of $v \mapsto
\tau^v$, the map $s \mapsto \tau^{w_s} b (s)$ is still measurable and by the
invariance of $\| \cdot \|_E$ under translations $\| b \|_{L^1 E} = \|
\tau^{w_{\cdot}} b \|_{L^1 E}$. Continuity of $T^w b$ and the bound $\| T^w b
\|_{C^0 E} \leqslant \| b \|_{L^1 E}$ follow from standard properties of
Bochner integral, as well as the linearity of the map $b \mapsto T^w b$.
Similarly, it is easy to see that, in the case $b$ enjoys higher
integrability, $T^w$ can also be defines as a linear bounded operator from
$L^q_t E$ to $C_t^{1 - 1 / q} E$. Furthermore, if $w$ and $\tilde{w}$ are such
that $w_t = \tilde{w}_t$ for Lebesgue-a.e. $t \in [0, T]$, then $T^w b$ and
$T^{\tilde{w}} b$ coincide for all $b$, so that $T^w$ can be defined for $w$
in an equivalence class.

\begin{lemma}
  \label{sec3.1 lemma 1}Let $w^n \rightarrow w$ in $L^1_t \mathbb{R}^d$ and $b
  \in L^q_t E$, then $T^{w^n} b \rightarrow T^w b$ in $C_t^{1 - 1 / q} E.$
\end{lemma}

\begin{proof}
  We can assume in addition that $w^n_t \rightarrow w_t$ for Lebesgue-a.e.
  $t$, the general case following from applying the reasoning to any possible
  subsequence that can be extracted from $\{ T^{w_n} b \}_n$. Since
  $\tau^{w^n_t} b_t \rightarrow \tau^{w_t} b_t$ for Lebesgue-a.e. $t$ and $\|
  \tau^{w^n_t} b_t - \tau^{w_t} b_t \|^q \lesssim \| b_t \|^q \in L^1$, it
  follows from dominated convergence that
  \[ \| T^{w^n} b - T^w b \|_{C^{1 - 1 / q} E} \lesssim \int_0^T \|
     \tau^{w^n_t} b_t - \tau^{w_t} b_t \|^q \rightarrow 0, \quad \text{as } n
     \rightarrow \infty, \]
  which gives the conclusion. 
\end{proof}

The advantage of the above definition of $T^w$ is that it is intrinsic and
does not depend on any approximation procedure by mollifiers. However, a
possibly more intuitive description of $T^w b$ can be given by duality. Recall
that in the sense of distributions $(\tau^v)^{\ast} = \tau^{- v}$, so that for
any $\varphi \in \mathcal{S} (\mathbb{R}^d) \hookrightarrow E^{\ast}$ it holds
\[ \langle T^w_t b, \varphi \rangle = \int_0^t \langle b_s, \varphi \left(
   \cdot \, - w_s \right) \rangle \mathd s \]
where the pairing is integrable since $\left| \langle b_s, \varphi \left(
\cdot \, - w_s \right) \rangle \right| \lesssim_{\varphi} \| b_s \|_E$. The
above relation holds for all $\varphi \in \mathcal{S} (\mathbb{R}^d)$ and
therefore uniquely identifies $T^w b (t)$ as an element of $\mathcal{S}'
(\mathbb{R}^d)$, for all $t \in [0, T]$. The advantage now is that the map
$(t, x) \mapsto \varphi (x - w_t)$ can be regarded as an element of
$L^{\infty} (0, T ; \mathcal{S} (\mathbb{R}^d))$, to which standard operations
on $\mathcal{S} (\mathbb{R}^d)$ such as differentiation and convolution can be
applied.

\begin{lemma}
  \label{sec3.1 lemma 2}Let $w$ and $b$ be as above. Then:
  \begin{enumerateroman}
    \item Averaging and spatial differentiation commute, i.e. for all $i = 1,
    \ldots, d$, $\partial_i T^w b = T^w \partial_i b$.
    
    \item Averaging and spatial convolution commute, i.e. for any $K \in
    C^{\infty}_c (\mathbb{R}^d)$ it holds
    \[ K \ast (T^w b) = T^w (K \ast b) = (T^w K) \ast b. \]
  \end{enumerateroman}
\end{lemma}

\begin{proof}
  Both statements follow easily from the duality formulation. For any $\varphi
  \in \mathcal{S} (\mathbb{R}^d)$ and $t \in [0, T]$ it holds
  \begin{align*}
    \langle \partial_i T^w b (t), \varphi \rangle & = - \langle T^w b (t),
    \partial_i \varphi \rangle = - \int_0^t \langle b_r, \partial_i \varphi
    \left( \cdot \, - w_r \right) \rangle \, \mathd r\\
    & = \int_0^t \langle \partial_i b_r, \varphi \left( \cdot \, - w_r
    \right) \rangle \, \mathd r = \langle (T^w \partial_i b) (t), \varphi
    \rangle .
  \end{align*}
  If $K \in C^{\infty}_c (\mathbb{R}^d)$, then denoting by $\tilde{K}$ its
  reflection, by duality it holds
  \begin{align*}
    \langle K \ast T^w b (t), \varphi \rangle & = \langle T^w b (t), \tilde{K}
    \ast \varphi \rangle = \int_0^t \langle b_r, \tau^{- w_r} (\tilde{K} \ast
    \varphi) \rangle \, \mathd r = \int_0^t \langle b_r, \tilde{K} \ast
    (\tau^{- w_r} \varphi) \rangle \, \mathd r\\
    & = \int_0^t \langle K \ast b (r), \tau^{- w_r} \varphi \rangle \, \mathd
    r = \langle T^w (K \ast b) (t), \varphi \rangle .
  \end{align*}
  A similar computation shows the other part of the identity.
\end{proof}

\begin{remark}
  \label{sec3.1 remark localization}Let us point out that if $w \in
  L^{\infty}$, then the averaging operator has finite speed of propagation and
  so behaves well under localisation. Indeed, if $b \in L^1_t E$ is such that
  $\tmop{supp} b_t \subset B_R$ for all $t \in [0, T]$, then $\tmop{supp} T^w
  b (t) \subset B_{R + \| w \|_{\infty}}$ for all $t \in [0, T]$ and similarly
  if $b$ and $\tilde{b}$ are such that their restrictions to $B_R$ coincide
  for all $t$, then $T^w b$ and $T^w \tilde{b}$ will still coincide on $B_{R -
  \| w \|_{\infty}}$.
\end{remark}

In view of the applications in Section~\ref{sec4}, our main goal is to
establish conditions under which $T^w b \in C^{\gamma}_t F$, where $\gamma > 1
/ 2$ and $F$ is another Banach space which enjoys better regularity properties
than the original space $E$: typically $F = C^{\beta}_x$ for suitable values
of $\beta$. For this reason, we are going to assume from now on that $b \in
L^q_t E$ for some $q > 2$. The idea behind this restriction is that sometimes
averaging allows to trade off time regularity for space regularity (think of
the analogy with parabolic regularity theory) and therefore in order to have
$T^w b \in C^{\gamma}_t F$, knowing a priori only that $T^w b \in C^{1 - 1 /
q}_t E$, we need to require at least
\[ 1 - \frac{1}{q} > \gamma > \frac{1}{2}  \enspace \Rightarrow \enspace q >
   2. \]
\begin{remark}
  \label{sec3.1 remark limitation}Despite our use of the terminology
  ``regularisation by averaging'', what we mean is really that we
  \tmtextit{fix} a drift $b$ and we want to establish that for a.e. path $w$
  the averaged function $T^w b$ has nice regularity properties. This is
  \tmtextit{different} from trying to establish that the averaging operator
  $T^w$ as a linear operator from $L^q_t E$ to $C^{\gamma}_t F$ is bounded,
  which is clearly false due to the time dependence of the drifts we consider.
  Indeed, given any $b \in E$, defining $\tilde{b}_t = \tau^{- w_t} b$, by
  definition of averaging we obtain $T_{s, t}^w \tilde{b} = (t - s) b$, which
  shows that for such choice of $\tilde{b}$, $T^w \tilde{b}$ cannot have
  better spatial regularity than $\tilde{b}$. The situation is more
  interesting if one defines $T^w$ for time independent drifts only.
  Prevalence statements for that case will be analysed in the companion
  paper~{\cite{galeatigubinelli}}.
\end{remark}

In order to show prevalence of regularisation by averaging, we first need to
show that such a property indeed defines Borel sets in suitable spaces of
paths. To this end, we require $F$ to be another Banach spaces which embeds
into $\mathcal{S}' (\mathbb{R}^d)$ which enjoys the following {\tmem{Fatou
type property}}: if $\{ x_n \}_n$ is a bounded sequence in $F$ such that $x_n$
converge to $x$ in the sense of distributions, then $x \in F$ and $\| x \|_F
\leqslant \liminf \| x_n \|_F$.

In the next lemma we allow any $\bar{\gamma} \in (0, 1)$, but our primary
focus will be $\bar{\gamma} = 1 / 2$.

\begin{lemma}
  \label{sec3.1 lemma 4}Let $F$ be as above, $b \in L^q_t E$ for some $q > 2$.
  Then for any $\bar{\gamma} \in (0, 1)$ the set
  \[ \mathcal{A}^{\bar{\gamma}} = \left\{ \, w : [0, T] \rightarrow
     \mathbb{R}^d \text{ such that } T^w b \in C^{\gamma}_t F \text{ for some
     } \gamma > \bar{\gamma} \right\} \]
  is Borel measurable w.r.t. the following topologies: $L^p$ with $p \in [1,
  \infty]$, $C^{\alpha}$ with $\alpha \geqslant 0$.
\end{lemma}

\tmcolor{blue}{\tmcolor{black}{\begin{proof}
  We can write $\mathcal{A}^{\bar{\gamma}}$ as a countable union of sets as
  follows:
  \[ \mathcal{A}^{\bar{\gamma}} = \bigcup_{m, n \in \mathbb{N}}
     \mathcal{A}_{m, n}^{\bar{\gamma}} \assign \bigcup_{m, n \in \mathbb{N}}
     \left\{ w : [0, T] \rightarrow \mathbb{R}^d \text{ such that } \| T^w b
     \|_{C^{\bar{\gamma} + 1 / m} F} \leqslant n \right\} ; \]
  in order to show the statement, it suffices to show that for every $m, n$
  the set $\mathcal{A}_{m, n}^{\bar{\gamma}}$ is closed in the above
  topologies. It suffices to show that it is closed in the $L^1$-topology,
  which is weaker than any of the others considered. Let $w^k$ be a sequence
  in $\mathcal{A}_{m, n}^{\bar{\gamma}}$ such that $w^k \rightarrow w$ in
  $L^1$, then by Lemma~\ref{sec3.1 lemma 1} we know that $T^{w^k} b
  \rightarrow T^w b$ in $C ([0, T] ; E)$ and so that for any $s < t$,
  $T^{w^k}_{s, t} b \rightarrow T^w_{s, t} b$ in $E$ and in $\mathcal{S}'
  (\mathbb{R}^d)$. On the other hand, by definition of $\mathcal{A}_{m,
  n}^{\bar{\gamma}}$ it holds
  \[ \sup_k  \frac{\| T^{w^k}_{s, t} b \|_F}{| t - s |^{\bar{\gamma} + 1 / m}}
     \leqslant n \]
  which implies by the Fatou property of $F$ that $T^w_{s, t} b \in F$ and
  \[ \frac{\| T^w_{s, t} b \|_F}{| t - s |^{\bar{\gamma} + 1 / m}} \leqslant
     n. \]
  As the reasoning holds for any $s < t$, it follows that $T^w b \in
  \mathcal{A}_{m, n}^{\bar{\gamma}}$ as well.
\end{proof}}}

\begin{remark}
  Any weakly\mbox{-}$\ast$ compact Banach space $F$ which embeds in
  $\mathcal{S}' (\mathbb{R}^d)$ satisfies the Fatou property. In the following
  we will always work with $L^p_x$\mbox{-}based function spaces with $p \in
  [2, \infty]$, so the property holds automatically. Let us also point out
  that the proof actually works more generally for conditions of the form $T^w
  b \in C^{\omega}_t F$, where $\omega$ is a prescribed modulus of continuity.
\end{remark}

We are now ready to provide a first prevalence statement.

\begin{theorem}
  \label{sec3.3 thm prevalence averaging}Let $b \in L^{\alpha}_t L^{s, p}_x$
  (resp. $b \in L^{\alpha}_t B^s_{p, q}$) for some $\alpha > 2$, $s \in
  \mathbb{R}$, $p, q \in [2, \infty)$. Let $\delta \in [0, 1)$ and $\beta \in
  \mathbb{R}$ satisfy
  \begin{equation}
    \beta < s + \frac{1}{\delta} \left( \frac{1}{2} - \frac{1}{\alpha} \right)
    - \frac{d}{p}, \label{eq:beta}
  \end{equation}
  where $d$ is the space dimension, i.e. $L^{s, p}_x = L^{s, p} (\mathbb{R}^d
  ; \mathbb{R}^m)$, and we adopt the convention that~{\eqref{eq:beta}} is
  satisfied for any $\beta$ if $\delta = 0$. Then for almost every $\varphi
  \in C^{\delta}_t$, $T^{\varphi} b \in C^{\gamma}_t C^{\beta}_x$ for some
  $\gamma > 1 / 2$.
\end{theorem}

\begin{proof*}{Proof of Theorem~\ref{sec3.3 thm prevalence averaging}}
  By Lemma~\ref{sec3.1 lemma 4}, the set
  \[ \mathcal{A} = \left\{ \, w \in C^{\delta}_t \, : T^w b \in C^{\gamma}_t
     C^{\beta}_x \text{ for some } \gamma > 1 / 2 \right\} \]
  is Borel in $C^{\delta}_t$. For simplicity we will adopt the notation $b \in
  L^{\alpha}_t E$, as the reasoning is the same for $E = L^{s, p}_x$ or $E =
  B^s_{p, q}$. In order to prove the statement, it remains to find a suitable
  tight probability distribution $\mu$ on $C^{\delta}_t$ such that for any
  $\varphi \in C^{\delta}_t$ it holds
  \begin{equation}
    \mu (\varphi + \mathcal{A}) = \mu \left( \text{$w \in C^{\delta}_t : \,
    T^{\varphi + w} b \in C^{\gamma}_t C^{\beta}_x$ for some $\gamma > 1 / 2$}
    \right) = 1. \label{sec3.1 eq1}
  \end{equation}
  Thanks to the translation invariance of $\| \cdot \|_E$, we can reduce the
  above problem to an easier one. Indeed, setting $\tilde{b}_t \assign
  \tau^{\varphi_t} b_t$ for all $t \in [0, T]$, $\tilde{b} \in L^{\alpha}_t E$
  and it holds $T^{\varphi + w} b = T^w \tilde{b}$. In particular in order to
  show that~{\eqref{sec3.1 eq1}} holds for fixed $b \in L^{\alpha}_t E$ and
  for all $\varphi \in C^{\delta}_t$, it actually suffices to find $\mu$ such
  that
  \begin{equation}
    \mu \left( \text{$w \in C^{\delta}_t : \, T^w \tilde{b} \in C^{\gamma}_t
    C^{\beta}_x$ for some $\gamma > 1 / 2$} \right) = 1 \text{ for all } \,
    \tilde{b} \in L^{\alpha}_t E. \label{sec3.1 eq2}
  \end{equation}
  Considering equation~{\eqref{sec3.1 eq2}} for the choice $E = L^{s, p}$
  (resp. $E = B^s_{p, q}$), it suffices to show that for all $\beta$
  satisfying~{\eqref{eq:beta}} there exists a tight measure $\mu_{\beta,
  \delta}$ on $C^{\delta}_t$ such that
  \begin{equation}
    \mu_{\beta, \delta} \left( w \in C^{\delta}_t : T^w b \in C^{\gamma}_t
    C^{\beta}_x \text{ for some } \gamma > 1 / 2 \right) = 1 \quad \text{for
    all } b \in L^{\alpha}_t E. \label{eq:prev-stat-averaging}
  \end{equation}
  The rest of the section will be devoted to the identification of such a
  measure. In particular, using Theorem~\ref{sec3.3 thm averaging bessel}
  (resp. Theorem~\ref{sec3.3 thm averaging besov}) combined with \
  Remark~\ref{sec3.3 remark embedding} below, we can choose $\mu_{\beta,
  \delta} = \mu^H$ to be the law of a fractional Brownian motion of parameter
  $H \in (0, 1)$ such that $H > \delta$ and
  \[ \beta < s + \frac{1}{H} \left( \frac{1}{2} - \frac{1}{\alpha} \right) -
     \frac{d}{p} . \]
\end{proof*}

We conclude this section with a lemma on approximation by mollifications which will
be very useful in Section~\ref{sec4}.

\begin{lemma}
  \label{sec3.1 lemma 3}Let $b \in L^q_t E$ such that $T^w b \in C^{\gamma}_t
  C^{\beta}_x$ for some $\gamma \in (0, 1]$, $\beta \in (0, \infty)$ and let
  $\{ \rho^{\varepsilon} \}_{\varepsilon > 0}$ be a family of standard spatial
  mollifiers; let $b^{\varepsilon} \assign \rho^{\varepsilon} \ast b$.Then for
  any $\delta > 0$ it holds $T^w b^{\varepsilon} \rightarrow T^w b$ locally in
  $C^{\gamma - \delta}_t C^{\beta - \delta}_x$, namely for any $R > 0$ $T^w
  b^{\varepsilon} |_{[0, T] \times B_R} \nobracket \rightarrow T^w b |_{[0, T]
  \times B_R} \nobracket$ in $C^{\gamma - \delta} ([0, T] ; C^{\beta - \delta}
  (B_R))$.
\end{lemma}

\begin{proof}
  It follows immediately from the property $(T^w b)^{\varepsilon} = T^w
  b^{\varepsilon}$ that
  \[ \| T^w b^{\varepsilon} \|_{C^{\gamma}_t C^{\beta}_x} \leqslant \| T^w
     b \|_{C^{\gamma}_t C^{\beta}_x} \quad \forall \,
     \varepsilon > 0 \]
  and moreover that $T^w b^{\varepsilon} (t) \rightarrow T^w b (t)$ in
  $\mathcal{S}' (\mathbb{R}^d)$ as $\varepsilon \rightarrow 0$. For any $R >
  0$ and $\delta > 0$, thanks to the above uniform bound, we can extract by
  Ascoli-Arzel{\`a} a (not relabelled) subsequence such that $T^w
  b^{\varepsilon} |_{[0, T] \times B_{_R}}$ converges in
  $C^{\gamma - \delta} ([0, T] ; C^{\beta - \delta} (B_R))$ to a suitable
  limit; by the above convergence in probability, the limit must necessarily
  coincide with $T^w b |_{[0, T] \times B_{_R}}$ and since the
  reasoning holds for any subsequence we can extract, the whole sequence must
  converge to $T^w b |_{[0, T] \times B_{_R}}$.
\end{proof}

\subsection{Fractional Brownian motion and It\^{o}--Tanaka
formula}\label{sec2.2}

In view of concluding the proof of Theorem~\ref{sec3.3 thm prevalence
averaging} we give here the essential details on the fractional Brownian
motion (fBm), whose law will be used as a transverse measure for prevalence.

In the literature, it is more common the use of \tmtextit{probes}, that is
finite dimensional transverse measures in order to establish prevalence
properties. The only other work we are aware of using general stochastic
processes in this context is~{\cite{bayart}}. However see also~{\cite{peres}}
and the references therein for the study of properties of fractional Brownian
motion with deterministic drift.

\

The material on fractional Brownian motion presented here is classical and
taken from~{\cite{nualart2006}} and~{\cite{picard}}. A one dimensional fBm
$(W^H_t)_{t \geqslant 0}$ of Hurst parameter $H \in (0, 1)$ is a mean zero
continuous Gaussian process with covariance
\[ \mathbb{E} [W^H_t W^H_s] = \frac{1}{2} (| t |^{2 H} + | s |^{2 H} - | t - s
   |^{2 H}) . \]
When $H = 1 / 2$, it coincides with standard Brownian motion and for $H \neq 1
/ 2$ it is not a semi-martingale nor a Markov process. However it shares many
properties of Brownian motion, such as stationarity, reflexivity and
self-similarity. The trajectories of fBm are $\mathbb{P}$-a.s.
$\delta$-H{\"o}lder continuous for any $\delta < H$ and nowhere
$\delta$\mbox{-}H{\"o}lder continuous for any $\delta \geqslant H$; it follows
from Ascoli--Arzel{\`a} that its law $\mu^H$ is tight on $C^{\delta}_t$ for
any $\delta < H$.

A $d$\mbox{-}dimensional fBm $W^H$ of Hurst parameter $H \in (0, 1)$ is an
$\mathbb{R}^d$\mbox{-}valued Gaussian process with components given by
independent one dimensional fBms; we state for simplicity in the rest of the
section all the results for $d = 1$ but they generalise immediately to higher
dimension reasoning component by component.

A very useful property of fBm is that it admits representations in terms of
stochastic integrals. Given a two-sided Brownian motion $\{ B_t \}_{t \in
\mathbb{R}}$, a fBm of parameter $H \neq 1 / 2$ can be constructed by
\begin{equation}
  W^H_t = c_H \int_{- \infty}^t [(t - r)^{H - 1 / 2}_+ - (- r)_+^{H - 1 / 2}]
  \, \mathd B_r \label{sec2.2 non canonical representation}
\end{equation}
where $c_H = \Gamma (H + 1 / 2)^{- 1}$ is a suitable renormalising constant.
Such a representation is usually called {\tmem{non canonical}} as the
filtration $\mathcal{F}_t = \sigma (B_s : s \leqslant t)$ is strictly larger
than the one generated by $W^H$; it is useful as it immediately shows that,
for any pair $0 \leqslant s < t$, the variable $W^H_t$ decomposes into the sum
of two mean zero Gaussian variables, $W^H_t = W^{1, H}_{s, t} + W^{2, H}_{s,
t}$, where
\[ W^{1, H}_{s, t} = c_H \int_s^t (t - r)^{H - 1 / 2} \, \mathd B_r, \quad
   W^{2, H}_{s, t} = c_H \int_{- \infty}^s [(t - r)^{H - 1 / 2}_+ - (- r)_+^{H
   - 1 / 2}] \, \mathd B_r \]
with $W^{2, H}_{s, t}$ being $\mathcal{F}_s$-measurable and $W^{1, H}_{s, t}$
being independent of $\mathcal{F}_s$ and with variance
\[ \tmop{Var} (W^{1, H}_{s, t}) = \tilde{c}_H | t - s |^{2 H} \]
where $\tilde{c}_H = c_H^2 / (2 H)$. In particular this implies that
\begin{equation}
  \tmop{Var} (W^H_t | \sigma (W^H_r, r \leqslant s \nobracket)) \geqslant
  \tmop{Var} (W^H_t | \mathcal{F}_s \nobracket) = \tmop{Var} (W^{1, H}_{s, t})
  = \tilde{c}_H | t - s |^{2 H} \label{sec2.2 lnd}
\end{equation}
which is a local nondeterminism property. Loosely speaking, it means that for
any $s < t$, the increment $W_t^H - W^H_s$ contains a part which is
independent of the the history of the path $W^H_{\cdot}$ up to time $s$ and
therefore makes the path $W^H_{\cdot}$ ``intrinsically chaotic''. The local
nondeterminism property was first formulated by Berman in~{\cite{berman73}} in
a different context; it plays a major role in the proofs of this section and
indeed the prevalence statement can be alternatively proved by using the laws
of other locally nondeterministic Gaussian processes, see Remark~\ref{sec3.3
remark lnd}.

\

We are going to prove an It\^o--Tanaka type formula for averaged functionals,
in the same spirit of the one considered in~{\cite{coutin}}. We first need to
recall the Clark--Ocone formula, see~{\cite{nualart2006}}. Given a
two\mbox{-}sided standard Brownian motion $B$ on a space $(\Omega,
\mathcal{F}, \mathbb{P})$, $\mathcal{F}_t = \sigma (B_s, s \leqslant t)$, and
given a Malliavin differentiable random variable $A$ with Malliavin derivative
$D_{\cdot} A$, the Clark--Ocone formula states that
\begin{equation}
  A =\mathbb{E} [A] + \int_{- \infty}^{+ \infty} \mathbb{E} [D_r A |
  \mathcal{F}_r \nobracket] \, \mathd B_r . \label{sec2.2 clark ocone}
\end{equation}
From~{\eqref{sec2.2 clark ocone}} it follows immediately that, for any $s \in
\mathbb{R}$, we have the more general identity
\[ A =\mathbb{E} [A | \mathcal{F}_s \nobracket] + \int_s^{+ \infty} \mathbb{E}
   [D_r A | \mathcal{F}_r \nobracket] \, \mathd B_r . \]
We do not provide here the general definition of Malliavin derivative of a
Brownian variable, which can be found in~{\cite{nualart2006}}; we only provide
it in the following specific case, which is the one of our interest: given a
smooth function $f$ and a variable $X = \int_{- \infty}^{+ \infty} K_s \,
\mathd B_s$, the Malliavin derivative of $A : = f (X)$ is given by
\begin{equation}
  D_t A = \nabla f (X) \cdot K_t . \label{sec2.2 malliavin derivative}
\end{equation}
In the next statement, $P_t$ denotes the heat kernel, i.e. $P_t f = p_t \ast
f$ where
\[
p_t (x) = (2 \pi t)^{- d / 2} e^{- \frac{| x |^2}{2 t}} .
\]
\begin{lemma}
  \label{sec3.2 lemma ito tanaka}Let $b : [0, T] \times \mathbb{R}^d
  \rightarrow \mathbb{R}$ be a smooth, compactly supported function, then for
  any fixed $0 \leqslant s \leqslant t \leqslant T$, $H \in (0, 1)$ and $x \in
  \mathbb{R}^d$, the following identity holds with probability $1$:
  \begin{eqnarray}
    \int_s^t b (r, x + W^H_r) \, \mathd r & = & \tmcolor{black}{\int_s^t
    P_{\tilde{c}_H | r - s |^{2 H}} \, b (r, x + W^{2, H}_{s, r}) \, \mathd r}
    \label{sec3.2 preliminary ito tanaka} \\
    &  & + \int_s^t \int_u^t P_{\tilde{c}_H | r - u |^{2 H}} \nabla b (r, x +
    W^{2, H}_{u, r}) c_H | r - u |^{H - 1 / 2} \, \mathd r \, \cdot \mathd B_u
    . \nonumber
  \end{eqnarray}
\end{lemma}

\begin{proof}
  For $H = 1 / 2$ the above formula is well known and coincides with a
  standard application of the It\^o--Tanaka trick together with a representation
  formula for solution of the heat equation, see for instance the discussion
  in~{\cite{coutin}}; so we can assume $H \neq 1 / 2$. Let us fix $x \in
  \mathbb{R}^d$. Since $b$ is smooth, for fixed $r$ we can apply Clark--Ocone
  formula to $b (r, x + W^H_r)$ to obtain
  \begin{align*}
    b (r, x + W^H_r) &= \mathbb{E} [b (r, x + W^H_r) | \mathcal{F}_s
    \nobracket] + \int_s^r \mathbb{E} [\nabla b (r, x + W^H_r) | \mathcal{F}_u
    \nobracket] \, c_H (r - u)^{H - 1 / 2} \cdot \, \mathd B_u\\
    & = P_{\tilde{c}_H | r - s |^{2 H}} b (r, x + W^{2, H}_{s, r}) +
    \int_s^r P_{\tilde{c}_H | r - u |^{2 H}} \nabla b (r, x + W^{2, H}_{u, r})
    c_H | r - u |^{H - 1 / 2} \cdot \mathd B_u
  \end{align*}
  where we used both the representation of $W^H$ in terms of a stochastic
  integral and the decomposition $W^H_r = W^{1, H}_{u, r} + W^{2, H}_{u, r}$
  with $W^{1, H}_{u, r}$ independent of $\mathcal{F}_u$. Integrating over $[s,
  t]$ and applying stochastic Fubini's theorem (which is allowed since we are
  assuming $b$ smooth and compactly supported) we obtain
  \begin{eqnarray*}
    \int_s^t b (r, x + W^H_r) \, \mathd t & = & \int_s^t P_{\tilde{c}_H | r -
    s |^{2 H}} b (r, x + W^{2, H}_{s, r}) \, \mathd r\\
    &  & + c_H \int_s^t \int_s^r P_{\tilde{c}_H | t - s |^{2 H}} \nabla b (t,
    x + W^{2, H}_{u, r}) | r - u |^{H - 1 / 2} \cdot \mathd B_u \mathd r\\
    & = & \int_s^t P_{\tilde{c}_H | r - s |^{2 H}} b (r, x + W^{2, H}_{s, r})
    \, \mathd r\\
    &  & + c_H \int_s^t \int_u^t P_{\tilde{c}_H | r - u |^{2 H}} \nabla b (r,
    x + W^{2, H}_{u, r}) | r - u |^{H - 1 / 2} \, \mathd r \, \cdot \mathd B_u
  \end{eqnarray*}
  which gives the conclusion.
\end{proof}

The previous result can be strengthened by considering for instance $b \in
C_b^1$ instead of smooth, or showing that we can find a set of probability $1$
on which the identity holds for all $0 \leqslant s \leqslant t \leqslant T$;
we don't do it here since it is not needed for our purposes. Instead, we need
to strengthen the result to the following functional equality.

\begin{theorem}
  \label{sec3.2 functional ito tanaka}Let $b : [0, T] \times \mathbb{R}^d
  \rightarrow \mathbb{R}$ be a smooth, compactly supported function, then for
  any fixed $0 \leqslant s \leqslant t \leqslant T$, $H \in (0, 1)$, with
  probability $1$ it holds
  \begin{eqnarray}
    T^{W^H} b_{s, t} & = & \int_s^t P_{\tilde{c}_H | r - s |^{2 H}} b \left(
    r, \cdot \, + W^{2, H}_{s, r} \right) \, \mathd r \label{sec3.2 ito tanaka
    formula} \\
    &  & + c_H \int_s^t \int_u^t P_{\tilde{c}_H | r - u |^{2 H}} \nabla b
    \left( r, \cdot \, + W^{2, H}_{u, r} \right) | r - u |^{H - 1 / 2} \,
    \mathd r \, \cdot \mathd B_u \nonumber
  \end{eqnarray}
  where the first integral must be interpreted as a Bochner integral, while
  the second one as a functional stochastic integral.
\end{theorem}

We postpone the proof of this result to Appendix~\ref{appendixA3}, as it is
quite technical and requires some knowledge of stochastic integration in UMD
spaces. Up to technical details, it is mostly a rewriting of the statement
already contained in Lemma~\ref{sec3.2 lemma ito tanaka} without further
insights.

\subsection{Regularity estimates in Bessel and Besov spaces}\label{sec3.3}

We provide here the regularity estimates for $T^w b$ when $w$ is sampled as a
fBm of parameter $H$, in view of
establishing~{\eqref{eq:prev-stat-averaging}}.

The main ingredients of the proof are the use of the functional
It\^{o}--Tanaka formula~{\eqref{sec3.2 ito tanaka formula}} together with
Burkholder's inequality (Theorem~\ref{appendixA3 thm burkholder} below), heat
kernel and interpolation estimates from Lemmata~\ref{appendixA2 heat kernel
estimates} and~\ref{appendixA2 interpolation inequality}. We refer the reader
to Appendices~\ref{appendixA2} and~\ref{appendixA3} for more information on
these tools. Let us point out that the strategy of proof is fairly general and
in principle could work also in other classes of spaces, up to the requirement
that the above tools are still available. However, in order to apply
Burkholder's inequality, we need to restrict to scales of $L^p$\mbox{-}based
spaces with $p \geqslant 2$. See Appendix~\ref{appendixA3} for a deeper
discussion of this point.

Although our main aim is to establish prevalence results, our results are also
new in the probabilistic setting and therefore we will try to give their
sharpest versions. In particular we will always achieve exponential
integrability whenever it is possible.

\begin{theorem}
  \label{sec3.3 thm averaging bessel}Let $W^H$ be a fBm of parameter $H$ and
  let $b \in L^q_t L^{s, p}_x$ for some $p, q \in [2, \infty)$. Then for any
  $\rho > 0$ satisfying
  \begin{equation}
    H \rho + \frac{1}{q} < \frac{1}{2}, \label{sec3.3 KR condition}
  \end{equation}
  $T^{W^H} b \in C_t^{\gamma} L_x^{s + \rho, p}$ for some $\gamma > 1 / 2$
  with probability $1$; moreover, there exist positive constants $\lambda, K$
  independent of $b$ such that
  \begin{equation}
    \mathbb{E} \left[ \exp \left( \lambda \frac{\| T^{W^H} b \|^2_{C^{\gamma}
    L^{s + \rho, p}}}{\| b \|_{L^q L^{s, p}}^2} \right) \right] \leqslant K.
    \label{sec3.3 thm exp integrability bessel}
  \end{equation}
\end{theorem}

\begin{proof}
  Without loss of generality, we can assume $s = 0$. Indeed, if $b \in L^q_t
  L^{s, p}_x$, then $b = G^s \, \tilde{b}$, where $\tilde{b} \in L^q_t L^p_x$
  and $\| b \|_{L^q L^{s, p}} = \| \tilde{b} \|_{L^q L^p}$; once the statement
  is shown for $\tilde{b}$, we can use the fact the commutating property of
  averaging operators $T^{W^H} b = T^{W^H} \left( G^s \, \tilde{b} \right) =
  G^s \, (T^{W^H} \tilde{b})$ to obtain the analogue statement for $b$ as
  well.
  
  Let us first assume $b$ to be a smooth function. By the Ito--Tanaka formula,
  \begin{eqnarray*}
    \int_s^t b \left( r, \cdot \, + W^H_r \right) \, \mathd r & = &
    \tmcolor{black}{\int_s^t P_{\tilde{c}_H | r - s |^{2 H}} \, b \left( r,
    \cdot \, + W^{2, H}_{s, r} \right) \, \mathd r}\\
    &  & + \, c_H \int_s^t \int_u^t P_{\tilde{c}_H | r - u |^{2 H}} \nabla b
    \left( r, \cdot \, + W^{2, H}_{u, r} \right) | r - u |^{H - 1 / 2} \,
    \mathd r \, \cdot \mathd B_u\\
    & = : & I^{(1)}_{s, t} + I_{s, t}^{(2)} .
  \end{eqnarray*}
  From now on for simplicity we will drop the constants $c_H$, $\tilde{c}_H$
  as they don't play any significant role in the following calculations. For
  the first term, we can apply the deterministic estimate:
  \begin{eqnarray*}
    \| I_{s, t}^{(1)} \|_{L^{\rho, p}} = \left\| \int_s^t P_{| r - s |^{2 H}}
    \, b \left( r, \cdot \, + W^{2, H}_{s, r} \right) \, \mathd r
    \right\|_{L^{\rho, p}} & \leqslant & \int_s^t \left\| P_{| r - s |^{2 H}}
    \, b_r \right\|_{L^{\rho, p}} \, \mathd r\\
    & \lesssim & \int_s^t | r - s |^{- \rho H} \left\|_{} \, b_r
    \right\|_{L^p} \, \, \mathd r\\
    & \leqslant & \| b \|_{L^q L^p} \, \left| \int_s^t | r - s |^{- \rho H \,
    q'} \mathd r \right|^{1 / q'}\\
    & \lesssim & \| b \|_{L^q L^p} \, | t - s |^{1 - 1 / q - \rho H}
  \end{eqnarray*}
  where we used the heat kernel estimates for Bessel spaces, see
  Lemma~\ref{appendixA2 heat kernel estimates}, and the fact that the
  $L^{\rho, p}$\mbox{-}norm of $b_r$ is not affected by a translation of
  $W^{2, H}_{r, s}$. Observe that $\rho H \, q' < 1$ is granted by
  condition~{\eqref{sec3.3 KR condition}}. Moreover,~{\eqref{sec3.3 KR
  condition}} implies that $1 - 1 / q - \rho H > 1 / 2$ and therefore we
  deduce that there exists $\gamma > 1 / 2$ such that, uniformly in $\omega
  \in \Omega$,
  \begin{equation}
    \| I^{(1)} \|_{C^{\gamma} L^{s, p}} \lesssim_{} \| b \|_{L^q L^p}
    \label{sec3.3 proof eq0} .
  \end{equation}
  For the second term, applying Burkholder's inequality~{\eqref{appendixA3
  burkholder}} (which is allowed since $L^{\rho, p}$ with $p \geqslant 2$ is a
  martingale type~2 space), we obtain
  \begin{equation}
    \mathbb{E} [\| I^{(2)}_{s, t} \|_{L^{\rho, p}}^{2 k}] \leqslant \left( C
    \, k \right)^k \, \mathbb{E} \left[ \left( \int_s^t \, \left\| \int_u^t
    P_{| r - u |^{2 H}} \nabla b \left( r, \cdot \, + W^{2, H}_{u, r} \right)
    | r - u |^{H - 1 / 2} \, \mathd r \, \right\|^2_{L^{\rho, p}} \mathd s
    \right)^k \right] \label{sec3.3 proof eq1} .
  \end{equation}
  We can then estimate the inner integral by deterministic estimates similar
  to the ones above:
  \begin{eqnarray*}
    \left\| \int_u^t P_{| r - u |^{2 H}} \nabla b \left( r, \cdot \, + W^{2,
    H}_{u, r} \right) | r - u |^{H - 1 / 2} \, \mathd r \, \right\|_{L^{\rho,
    p}} & \leqslant & \int_u^t \| P_{| r - u |^{2 H}} \nabla b_r \|_{L^{\rho,
    p}} | r - u |^{H - 1 / 2} \, \mathd r\\
    & \lesssim_{} & \int_u^t | r - u |^{- H (\rho + 1) + H - 1 / 2} \, \| b
    (r) \|_{L^p} \, \mathd r\\
    & \leqslant & \| b \|_{L^q L^p}  \left( \int_u^t | r - u |^{- (H \rho + 1
    / 2) q'} \mathd r \right)^{1 / q'}\\
    & \lesssim & \| b \|_{L^q L^p}  | t - u |^{1 / 2 - 1 / q - H \rho},
  \end{eqnarray*}
  where again we used the fact that $(H \rho + 1 / 2) q' < 1$, thanks to
  {\eqref{sec3.3 KR condition}}. Set $\varepsilon \assign 1 - 2 / q - 2 H
  \rho$; inserting the estimate inside~{\eqref{sec3.3 proof eq1}} we obtain
  that, for a suitable $C' > 0$, it holds
  \[ \mathbb{E} [\| I^{(2)}_{s, t} \|_{L^{\rho, p}}^{2 k}] \leqslant (C' k)^k
     \| b \|^{2 k}_{L^q L^p} \, | t - s |^{k (1 + \varepsilon)} . \]
  But then we have
  \begin{align*}
    \mathbb{E} \left[ \exp \left( \lambda \frac{\| I^{(2)}_{s, t} \|_{L^{\rho,
    p}}^2}{| t - s |^{1 + \varepsilon} \| b \|^2_{L^q L^p}} \right) \right] &
    = \, \sum_k \frac{\lambda^k}{k!} \, \mathbb{E} \left[ \frac{\| I^{(2)}_{s,
    t} \|_{L^{\rho, p}}^{2 k}}{| t - s |^{k (1 + \varepsilon)} \| b \|^{2
    k}_{L^q L^p}} \right]\\
    & \leqslant \, \sum_k \frac{(\lambda C')^k k^k}{k!} \lesssim \sum_k
    (\lambda C' e)^k < \infty
  \end{align*}
  as soon as $\lambda < (C' e)^{- 1}$. It follows from
  Lemma~\tmcolor{blue}{\tmcolor{black}{\ref{appendixA1 chaining lemma}}} that,
  for any $\varepsilon' < \varepsilon$, $I^{(2)} \in C_t^{1 / 2 +
  \varepsilon'} L_x^{\alpha, p}$ and that there exists another $\lambda > 0$
  (not relabelled for simplicity) such that
  \begin{equation}
    \mathbb{E} \left[ \exp \left( \lambda \frac{\| I^{(2)} \|^2_{C^{1 / 2 +
    \varepsilon'} L^{\rho, p}}}{\| b \|^2_{L^q L^p}} \right) \right] \leqslant
    K \label{sec3.3 proof eq2}
  \end{equation}
  for a constant $K$ independent of $b$, which together with~{\eqref{sec3.3
  proof eq0}} proves the claim for smooth $b$.
  
  Now let $b$ be a generic element of$L^q_t L^p_x$; let us consider the case
  $q < \infty$ first. We can then find a sequence $b_n$ of smooth functions
  such that $\| b - b_n \|_{L^q L^p} \rightarrow 0$ as $n \rightarrow \infty$;
  we know that in this case $\| T^{W^H} (b_n - b) \|_{C^0 L^p} \rightarrow 0$,
  uniformly on $\omega \in \Omega$. On the other hand, it follows
  from~{\eqref{sec3.3 thm exp integrability bessel}}, applied to $b_n - b_m$,
  that for any $k$ it holds
  \[ \mathbb{E} [\| T^{W^H} (b_n - b_m) \|^{2 k}_{C^{\gamma} L^{\rho, p}}]
     \lesssim_k \| b_n - b_m \|^{2 k}_{L^q L^p} \]
  which implies that the sequence $T^{W^H} b_n$ is Cauchy in $L^{2 k} (\Omega,
  \mathbb{P}; C_t^{\gamma} L_x^{\rho, p})$, hence it admits a limit. But then
  the limit must coincide with $T^{W^H} b \in L^{2 k} (\Omega, \mathbb{P};
  C_t^{\gamma} L_x^{\rho, p})$. Applying Fatou lemma we deduce
  \[ \mathbb{E} \left[ \exp \left( \lambda \frac{\| T^{W^H} b \|^2_{C^{\gamma}
     L^{\rho, p}}}{\| b \|_{L^q L^p}^2} \right) \right] \leqslant \liminf_{n
     \rightarrow \infty} \, \mathbb{E} \left[ \exp \left( \lambda \frac{\|
     T^{W^H} b_n \|^2_{C^{\gamma} L^{\rho, p}}}{\| b_n \|_{L^q L^p}^2} \right)
     \right] \leqslant K \]
  which gives the conclusion. In the case $q = \infty$, since $b \in L^{q'}_t
  L^p_x$ for every $q' < \infty$, for any fixed $\rho$ we can find $q'$ big
  enough such that~{\eqref{sec3.3 KR condition}} still holds and apply the
  result for such $q'$.
\end{proof}

\begin{remark}
  Theorem~\ref{sec3.3 thm averaging bessel} immediately implies that, under
  assumption~{\eqref{sec3.3 KR condition}}, the random averaging operator
  $T^{W^H} : b \mapsto T^w b$ is a linear bounded map from $L^q_t L^{s, p}_x$
  into $L^k (\Omega, \mathbb{P}; C^{\gamma}_t L^{s + \rho, p}_x)$, for any $k
  \in \mathbb{N}$. Observe the difference with Remark~\ref{sec3.1 remark
  limitation}.
\end{remark}

We can actually even improve the regularity result of Theorem~\ref{sec3.3 thm
averaging bessel}.

\begin{corollary}
  \label{sec3.3 corollary improvement}Let $b \in L^q_t L^{s, p}_x$ with $p \in
  [2, \infty)$, $\rho > 0$ and assume~{\eqref{sec3.3 KR condition}} holds.
  Then there exists $\gamma > 1 / 2$ and a function $K (\lambda)$ independent
  of $b$ such that
  \[ \mathbb{E} \left[ \exp \left( \lambda \frac{\| T^{W^H} b \|^2_{C^{\gamma}
     L^{s + \rho, p}}}{\| b \|_{L^q L^{s, p}}^2} \right) \right] \leqslant K
     (\lambda) < \infty \quad \forall \, \lambda \in \mathbb{R}. \]
\end{corollary}

\begin{proof}
  As before, we can assume without loss of generality $s = 0$. If $\rho$
  satisfies~{\eqref{sec3.3 KR condition}}, then there exists $\varepsilon > 0$
  such that also $\rho + \varepsilon$ satisfies~{\eqref{sec3.3 KR condition}};
  it then follows from Lemma~\ref{appendixA2 interpolation inequality} that
  \[ \| T^{W^H} b \|_{C^{\gamma} L^{\rho, p}} \lesssim \| T^{W^H} b
     \|_{C^{\gamma} L^p}^{1 - \theta} \| T^{W^H} b \|_{C^{\gamma} L^{\rho +
     \varepsilon, p}}^{\theta} \leqslant \| b \|_{L^q L^p}^{1 - \theta} \|
     T^{W^H} b \|_{C^{\gamma} L^{\rho + \varepsilon, p}}^{\theta} \]
  where $\theta = \varepsilon / (s + \varepsilon)$ and we used the fact that
  $q > 2$ due to condition~{\eqref{sec3.3 KR condition}}. It follows that
  \[ \frac{\| T^{W^H} b \|^{2 / \theta}_{C^{\gamma} L^{\rho, p}}}{\| b \|^{2 /
     \theta}_{L^q L^p}} \lesssim \frac{\| T^{W^H} b \|^2_{C^{\gamma} L^{\rho +
     \varepsilon, p}}}{\| b \|^2_{L^q L^p}} \]
  where $1 / \theta = (s + \varepsilon) / \varepsilon \backassign \beta$.
  Applying Theorem~\ref{sec3.3 thm averaging bessel} to $\rho + \varepsilon$,
  we obtain that there exist $\bar{\lambda}$, $\bar{K}$ independent of $b$
  such that
  \[ \mathbb{E} \left[ \exp \left( \bar{\lambda}  \frac{\| T^{W^H} b \|^{2
     \beta}_{C^{\gamma} L^{\rho, p}}}{\| b \|_{L^q L^p}^{2 \beta}} \right)
     \right] \leqslant \mathbb{E} \left[ \exp \left( C_{\varepsilon} 
     \bar{\lambda}  \frac{\| T^{W^H} b \|^2_{C^{\gamma} L^{\rho + \varepsilon,
     p}}}{\| b \|_{L^q L^p}^2} \right) \right] \leqslant \bar{K} . \]
  Since $\beta > 1$, the conclusion follows with the constant $K (\lambda)$
  given by the optimal deterministic constant such that $\exp (\lambda x^2)
  \leqslant K (\lambda) \exp (\bar{\lambda} x^{2 \beta}) / \bar{K}$ for all $x
  \geqslant 0$.
\end{proof}

In the limiting case in which~{\eqref{sec3.3 KR condition}} becomes an
equality, slightly more careful estimates still allow to obtain a regularity
result in space at the cost of lower time regularity.

\begin{theorem}
  \label{sec3.3 thm threshold bessel}Let $b \in L^q_t L^{s, p}_x$ with $p \in
  [2, \infty)$, $q \in (2, \infty)$ and let $\rho > 0$ satisfy
  \begin{equation}
    H \rho + \frac{1}{q} = \frac{1}{2} \label{sec3.3 threshold coefficient} .
  \end{equation}
  Then $T^{W^H} b \in C^0_t L_x^{s + \rho, p}$ with probability $1$ and there
  exist positive constant $\lambda$, $K$, independent of~$b$, such that
  \begin{equation*}
    \mathbb{E} \left[ \exp \left( \lambda \frac{\| T^{W^H} b \|^2_{C^0 L^{s +
    \rho, p}}}{\| b \|^2_{L^q L^{s, p}}} \right) \right] < K.
  \end{equation*}
\end{theorem}

\begin{proof}
  As before, we can assume $s = 0$, $b$ smooth; again we decompose $T^{W^H} b
  = I^{(1)} + I^{(2)}$. Going through the same calculations for $I^{(1)}$, we
  obtain
  \[ \| I_{s, t}^{(2)} \|_{L^{\alpha, p}} \lesssim \| b \|_{L^q_t L^p_x} \, |
     t - s |^{1 - 1 / q - \alpha H} = \| b \|_{L^q_t L^p_x} \, | t - s |^{1 /
     2} \]
  where the estimate is uniform in $\omega \in \Omega$; it follows immediately
  that
  \[ \mathbb{E} [\exp (\lambda \| I^{(2)} \|_{C_0 L^{\rho, p}}^2)] < \infty,
  \]
  and therefore we only need to focus on $I^{(2)}$. By Burkholder's
  inequality, we have
  \[ \mathbb{E} [\| I^{(2)} \|_{C^0 L^{\alpha, p}}^{2 k}] \leqslant \left( C
     \, k \right)^k \, \mathbb{E} \left[ \left( \int_0^T \, \left\| \int_u^T
     P_{| r - u |^{2 H}} \nabla b (r, \cdot + W^{2, H}_{u, r}) | r - u |^{H -
     1 / 2} \, \mathd r \, \right\|^2_{L^{\alpha, p}} \mathd s \right)^k
     \right] \]
  and as before we want to estimate the integral inside in a deterministic
  manner. Going through similar calculations we obtain
  \[ \int_0^T \left\| \int_u^T P_{| r - u |^{2 H}} \nabla b (r, \cdot + W^{2,
     H}_{u, r}) | r - u |^{H - 1 / 2} \, \mathd r \, \right\|^2_{L^{\alpha,
     p}} \mathd u \]
  \[ \qquad \qquad \qquad \lesssim \int_0^T \left( \int_u^T | r - u |^{- H
     \alpha - 1 / 2} \, \| b_r \|_{L^p} \, \mathd r_{} \right)^2 \mathd u \]
  and now due to the assumption on the coefficients, we can apply the
  Hardy\mbox{-}Littlewood\mbox{-}Sobolev inequality to obtain
  \[ \left( \int_0^T \left( \int_u^T | r - u |^{- H \alpha - 1 / 2} \, \| b
     (r) \|_{L^p} \, \mathd r_{} \right)^2 \mathd u \right)^{1 / 2} \lesssim
     \| b \|_{L^q_t L^p_x}, \]
  which implies
  \[ \mathbb{E} [\| I^{(2)} \|_{C^0 L^{\alpha, p}}^{2 k}] \leqslant (C' k)^k
     \| b \|^{2 k}_{L^q_t L^p_x} . \]
  The conclusion then follows by expanding the exponential and choosing
  $\lambda$ sufficiently small as before.
\end{proof}

Going through the exact same calculations as above, an analogue result can be
obtained in the case of Besov spaces $B^s_{p, q}$ with $p, q \in [2, \infty)$.
In order to avoid unnecessary repetitions, we omit the proof.

\begin{theorem}
  \label{sec3.3 thm averaging besov}Let $W^H$ be a fBm of parameter $H$ and
  let $b \in L^{\alpha}_t B^s_{p, q}$ for some $p, q \in [2, \infty)$. Then
  for any $\rho > 0$ satisfying
  \begin{equation}
    H \rho + \frac{1}{\alpha} < \frac{1}{2}, \label{sec3.3 KR condition 2}
  \end{equation}
  $T^{W^H} b \in C^{\gamma}_t B^{s + \rho}_{p, q}$ for some $\gamma > 1 / 2$
  with probability $1$; moreover, there exist a positive function $K
  (\lambda)$ independent of $b$ such that
  \begin{equation}
    \mathbb{E} \left[ \exp \left( \lambda \frac{\| T^{W^H} b \|^2_{C^{\gamma}
    B_{p, q}^{s + \rho}}}{\| b \|_{L^{\alpha} B_{p, q}^s}^2} \right) \right]
    \leqslant K (\lambda) < \infty \quad \forall \, \lambda \in \mathbb{R}.
    \label{sec3.3 thm exp integrability besov}
  \end{equation}
  If equality holds in~{\eqref{sec3.3 KR condition 2}}, then there exist
  positive constant $\tilde{\lambda}$, $\tilde{K}$, independent of $b$, such
  that
  \begin{equation*}
    \mathbb{E} \left[ \exp \left( \widetilde{\lambda} \frac{\| T^{W^H} b
    \|^2_{C^0 B_{p, q}^{s + \rho}}}{\| b \|^2_{L^{\alpha} B_{p, q}^s}} \right)
    \right] < \tilde{K} .
  \end{equation*}
\end{theorem}

We end this section with several remarks discussing various technical point
and extensions, and which can be skipped on a first reading.

\begin{remark}
  Heuristically, condition~{\eqref{sec3.3 KR condition 2}} can be seen as a
  time-space weighted regularity condition, where time counts as $1 / H$ times
  space (which is in agreement with parabolic regularity in the case $H = 1 /
  2$ of Brownian motion). Indeed, we know that the averaging operator $T^w$
  maps $L^{\alpha} B^s_{p, q}$ into $W^{1, \alpha} B^s_{p, q}$; if we assume
  that regularity can be distributed between time and space, it should also
  map $L^{\alpha} B^s_{p, q}$ into $W^{\theta, \alpha} B^{s + (1 - \theta) /
  H}_{p, q}$ for any $\theta \in (0, 1)$. In order to achieve $1 / 2 +
  \varepsilon$ regularity in time it is then required $\theta - 1 / \alpha > 1
  / 2$, which implies that the regularity gain in space is at most
  \[ \frac{1 - \theta}{H} < \frac{1}{H} \left( \frac{1}{2} - \frac{1}{\alpha}
     \right) \]
  which matches exactly condition~{\eqref{sec3.3 KR condition 2}} for $\rho$.
\end{remark}

\begin{remark}
  The restriction to work with $B^s_{p, q}$ with $q \in [2, \infty)$, is not
  particularly relevant since by Besov embedding if $b \in L^{\alpha}_t
  B^s_{p, q}$, then it also belongs to $L^{\alpha}_t B^s_{p, q'}$ for any $q'
  > q$ and to $L^{\alpha}_t B^{s - \varepsilon}_{p, q'}$ for any $q' < q$ and
  $\varepsilon > 0$, so that we can first embed it for a choice $q' \in [2,
  \infty)$ and then apply the estimate there. Also the restriction $p \neq
  \infty$ can be overcome, for instance by first localising it as \~{$b$} in a
  ball $B_R$ and then embedding it into some $p < \infty$; by the properties
  of averaging, we know that $T^{W^H} b = T^{W^H} \tilde{b}$ in $B_{R - \| W^H
  \|_{\infty}}$ and we can choose $R$ big enough such that $\mathbb{P} (R - \|
  W^H \|_{\infty} < R / 2)$ is very small, to deduce local estimates for
  $T^{W^H} b$ which hold with high probability. Alternatively, estimates for
  averaging in Besov-H\"{o}lder spaces have been given by a different
  technique in~{\cite{catelliergubinelli}}, Section~4.1. However for
  simplicity, when dealing with $b \in L^{\alpha}_t B^s_{\infty, \infty}$, we
  will always assume that $b$ has compact support in space, uniformly in time,
  so that we can embed it in $L^{\alpha}_t B^s_{p, p}$ for any $p < \infty$
  and then apply estimates there.
\end{remark}

\begin{remark}
  \label{sec3.3 remark limitation}The restriction to work with
  $L^p$\mbox{-}based spaces with $p \geqslant 2$ is more restrictive and it
  would be of fundamental importance to weaken it, especially reaching the
  case $p = 1$; this was already pointed out in Conjecture~1.2 from
  {\cite{catelliergubinelli}}. The reason is that, by the properties of
  averaging, we know that for any $K \in C^{\infty}_c$ and time independent
  $b$ it holds $K \ast T^w b = T^w (K \ast b) = (T^w K) \ast b$; \ if we were
  able to show that $T^w K \in C^{\gamma}_t W^{\rho, 1}_x$ with an estimate
  that only depends on the $L^1$\mbox{-}norm of $K$, then we could
  automatically deduce regularity estimates of the form $K \ast T^w b \in
  C^{\gamma}_t L^{\rho, p}$ with $b \in L^p$ for any $p \in [1, \infty]$. We
  could then consider a family of mollifiers obtained by rescaling $K$ (which
  all have the same $L^1$\mbox{-}norm, so the same estimate in $C^{\gamma}_t
  W^{\rho, 1}_x$) to get estimates for the map $b \mapsto T^w b$ in any
  $L^p$~based space with $p \in [1, \infty]$ (as above, only time independent
  $b$ considered).
\end{remark}

\begin{remark}
  \label{sec3.3 remark lnd}A closer look at the proofs shows that both the
  It\^{o}--Tanaka formula from Theorem~\ref{sec3.2 functional ito tanaka} and
  the regularity estimates from Theorems~\ref{sec3.3 thm averaging bessel}
  and~\ref{sec3.3 thm averaging besov} can be generalised to Gaussian
  processes $X$ different from fBm and of the form
  \[ X_t = \int_0^t K (t, s) \mathd B_s, \]
  for some deterministic matrix\mbox{-}valued function $K$, such that for some
  $H \in (0, 1)$ it holds
  \begin{equation}
    \tmop{Var} (X_t | \mathcal{F}_s \nobracket) \gtrsim | t - s |^{2 H} \quad
    \forall \, s < t \label{sec3.3 LND}
  \end{equation}
  where $\mathcal{F}_t = \sigma (B_s : s \leqslant t)$.
  Condition~{\eqref{sec3.3 LND}} is a type of strong local nondeterminism
  (SLND) and these type of processes satisfy many interesting properties,
  which are studied in detail in~{\cite{galeatigubinelli}}.
\end{remark}

\begin{remark}
  \label{sec3.3 remark embedding}It follows immediately from the above results
  and from Bessel (resp. Besov) embeddings (see Appendix~\ref{appendixA2})
  that if $b \in L^{\alpha}_t L^{s, p}_x$ (resp. $b \in L^{\alpha}_t B^s_{p,
  q}$) for some $\alpha > 2$, $p, q \in [2, \infty)$, then for any $\beta$
  such that
  \begin{equation}
    \beta < s + \frac{1}{H} \left( \frac{1}{2} - \frac{1}{\alpha} \right) -
    \frac{d}{p} \label{sec3.3 eq remark embedding}
  \end{equation}
  there exists $\gamma > 1 / 2$ such that $T^{W^H} b \in C^{\gamma}_t
  C^{\beta}_x$ with full probability. For instance in the case $s = 0$, i.e.
  $b \in L^{\alpha}_t L^p_x$, in order to require $T^{W^H} b \in C^{\gamma}_t
  C^0_x$ it is enough
  \[ \frac{1}{\alpha} + H \frac{d}{p} < \frac{1}{2}, \]
  while in order to require $T^{W^H} b \in C^{\gamma}_t C^1_x$ it suffices
  \[ \frac{1}{\alpha} + H \frac{d}{p} < \frac{1}{2} - H. \]
  If $b \in L^{\alpha}_t B^s_{\infty, \infty}$ with spatially compact support,
  uniform in time, then $T^{W^H} b \in C^{\gamma}_t C^n_x$ if
  \[ H < \frac{1}{n - s} \left( \frac{1}{2} - \frac{1}{\alpha} \right) . \]
\end{remark}

\begin{remark}
  \label{sec3.3 rem comparison}Finally, let us compare our results for
  $T^{W^H} b$ with existing literature; with the exception of the case $H = 1
  / 2$, in which classical stochastic calculus provides more refined
  information, the only references we are aware of are the
  aforementioned~{\cite{catelliergubinelli,le}}. The technique applied
  in~{\cite{catelliergubinelli}} allows to deal only with time independent
  $b$; however, introducing suitable weighted spaces, it does not require $b$
  to belong to $B^s_{p, p}$ for some $p < \infty$. The results from Section~7
  of~{\cite{le}}, where $b \in L^q_t B^s_{\infty, \infty}$ is considered, are
  in line with those from Remark~\ref{sec3.3 remark embedding}; still, the
  techniques used therein, based on moment estimates and
  Garsia-Rodemich-Rumsay lemma, do not provide global regularity estimates for
  $T^{W^H} b$ (only local ones) nor the exponential
  integrability~{\eqref{sec3.3 thm exp integrability bessel}}. Both such
  features will be fundamental in the solution theory presented the next
  section: global estimates avoid finite time blow-up of solutions,
  exponential integrability allows the use of Girsanov's theorem. Finally, let
  us point out that both references only provide estimates for $T_{s, t}^{W^H}
  b$ in $B^s_{\infty, \infty}$, not covering other scales $B^s_{p, q}$ with
  $p, q < \infty$.
\end{remark}

\section{Application to perturbed ODEs}\label{sec4}

Now we are going to transfer the prevalence results for the averaged
vector-field to prevalence of well-posedness to perturbed ODEs including
regularity of the flow. The key technical tool to achieve this connection is a
simple theory of nonlinear Young equations which we recall and adapt to our
specific setting.

\subsection{Perturbed ODEs as nonlinear Young differential
equations}\label{sec4.1}

In this section we provide a summary of the results contained
in~{\cite{catelliergubinelli}} on nonlinear Young differential equations
(YDEs). Sometimes we will provide slightly different statements which fit
better our context and in order to facilitate the understanding we will
provide self-contained proofs whenever possible.

\

Let us fix some notation first. Given $A \in C^{\gamma}_t C^{\nu}_x =
C^{\gamma} ([0, T] ; C^{\nu} (\mathbb{R}^d ; \mathbb{R}^d))$ for $\gamma, \nu
\in (0, 1)$, we denote by the norm $\| A \|_{C^{\gamma} C^{\nu}}$ and the
semi-norm $\llbracket A \rrbracket_{C^{\gamma} C^{\nu}}$ respectively the
quantities
\[ \| A \|_{C^{\gamma} C^{\nu}} = \sup_{t \in [0, T]} \| A_t (\cdot)
   \|_{C^{\nu}} + \sup_{s \neq t} \frac{\| A_{s, t} (\cdot) \|_{C^{\nu}}}{| t
   - s |^{\gamma}}, \]
and
\[ \llbracket A \rrbracket_{C^{\gamma} C^{\nu}} = \sup_{s \neq t}
   \frac{\llbracket A_{s, t} (\cdot) \rrbracket_{C^{\nu}}}{| t - s |^{\gamma}}
   = \sup_{s \neq t, x \neq y} \frac{| A (t, x) - A (t, y) - A (s, x) + A (s,
   y) |}{| t - s |^{\gamma}  | x - y |^{\nu}} . \]
One of the main results of~{\cite{catelliergubinelli}} is the rigorous
construction of the nonlinear Young integral.

\begin{theorem}
  \label{sec4.1 thm definition young integral}Let $\gamma, \rho, \nu \in (0,
  1)$ such that $\gamma + \nu \rho > 1$, $A \in C^{\gamma}_t C^{\nu}_x$ and
  $\theta \in C^{\rho}_t$. Then for any $[s, t] \subset [0, T]$ and for any
  sequence of partitions of $[s, t]$ with mesh converging to zero, the
  following limit exists and is independent of the chosen sequence of
  partitions:
  \[ \int_s^t A (\mathd u, \theta_u) \assign
     \tmcolor{blue}{\tmcolor{black}{\lim_{\tmscript{\begin{array}{c}
       | \Pi | \rightarrow 0
     \end{array}}}}} \sum_i A_{t_i, t_{t + 1}} (\theta_{t_i}) . \]
  The limit is usually referred as a nonlinear Young integral. Furthermore:
  \begin{enumeratenumeric}
    \item For all $s \leqslant r \leqslant t$ it holds $\int_s^r A (\mathd u,
    \theta_u) + \int_r^t A (\mathd u, \theta_u) = \int_s^t A (\mathd u,
    \theta_u)$.
    
    \item If $\partial_t A$ is continuous, then $\int_s^t A (\mathd u,
    \theta_u) = \int_s^t \partial_t A (u, \theta_u) \mathd u$.
    
    \item There exists a universal constant $C = C (\gamma, \rho, \nu)$ such
    that
    \[ \left| \int_s^t A (\mathd u, \theta_u) - A_{s, t} (\theta_s) \right|
       \leqslant C | t - s |^{\gamma + \nu \rho} \llbracket A
       \rrbracket_{C^{\gamma} C^{\nu}} \llbracket \theta \rrbracket_{C^{\rho}}
       . \]
    \item The map $(A, \theta) \mapsto \int_0^{\cdot} A (\mathd u, \theta_u)$
    is continuous as a function from $C^{\gamma}_t C^{\nu}_x \times C^{\rho}_t
    \rightarrow C^{\gamma}_t$, is linear in $A$ and there exists a constant
    $\tilde{C} = \tilde{C} (\gamma, \rho, \nu, T)$ such that
    \[ \left\| \int_0^{\cdot} A (\mathd u, \theta_u) \right\|_{C^{\gamma}}
       \leqslant \tilde{C}  \| A \|_{C^{\gamma} C^{\nu}}  (1 + \llbracket
       \theta \rrbracket_{C^{\rho}}) . \]
  \end{enumeratenumeric}
\end{theorem}

The statement is a (less general) version of Theorem~2.4
from~{\cite{catelliergubinelli}}; we omit the proof, but let us mention that
an elementary proof based on the Sewing Lemma has been also given
in~{\cite{hu}}. The statement above can be localised, i.e. it is enough to
require $A \in C^{\gamma}_t C^{\nu}_{\tmop{loc}}$ and in this case all the
estimates depend on the $C^{\gamma}_t C^{\nu}_x$\mbox{-}norm (resp. semi-norm)
of $A$ restricted to $[0, T] \times B_{\| \theta \|_{\infty}}$.

\

With this tool at hand, we can provide an alternative definition of
solutions to the perturbed ODE which is meaningful even when $b$ is
distributional in space. Since we want to apply the results from
Section~\ref{sec3}, from now on when we say that $b$ is distributional we are
always going to implicitly assume that there exists $q > 2$ such that $b \in
L^q_t E$, where $E$ is a suitable space of distributions as the ones described
in Section~\ref{sec3.1}.

\begin{definition}
  \label{sec4.1 defn solution}Let $b$ be a distributional drift such that $T^w
  b \in C^{\gamma}_t C_x^{\nu}$ for some $\gamma, \nu \in (0, 1]$ such that
  $\gamma (1 + \nu) > 1.$ Given $x_0 \in \mathbb{R}^d$, we say that $x$ is a
  solution to the ODE
  \begin{equation}
    x_t = x_0 + \int_0^t b (s, x_s) \mathd s + w_t \quad \forall \, t \in [0,
    T] \label{sec4.1 perturbed ode}
  \end{equation}
  if and only if $x \in w + C^{\gamma}$ and $\theta = x - w$ solves the
  non-linear Young differential equation
  \begin{equation}
    \theta_t = \theta_0 + \int_0^t T^w b (\mathd s, \theta_s) \quad \forall \,
    t \in [0, T] . \label{sec4.1 YDE}
  \end{equation}
\end{definition}

Observe that the condition $\gamma (1 + \nu) > 1$ immediately implies $\gamma
> 1 / 2$, in line with standard Young differential equations; in the case of
continuous $b$, it follows from the discussion in the introduction that the
condition $x \in w + C^{\gamma}$ is trivially satisfied and so the two
formulations~{\eqref{sec4.1 perturbed ode}} and~{\eqref{sec4.1 YDE}} are
equivalent, {\eqref{sec4.1 perturbed ode}}~being interpreted as the classical
integral equation.

\begin{remark}
  From now on we will mostly focus on solving~{\eqref{sec4.1 YDE}} with $T^w b
  = A$ being regarded as an abstract element in a class $C^{\gamma}_t
  C^{\nu}_x$; however, whenever $b$ is spatially bounded, the ODE formulation
  for $\theta$ is still useful, as it provides additional regularity estimates
  for $\theta$ compared to the ones given by the Young integral formulation:
  for instance if $b \in L^{\infty}_{t, x}$, then any solution $\theta$ of the
  integral equation is automatically Lipschitz with $\llbracket \theta
  \rrbracket_{\tmop{Lip}} \leqslant \| b \|_{L^{\infty}}$, while Point~3. of
  Theorem~\ref{sec4.1 thm definition young integral} only provides estimate
  for $\| \theta \|_{C^{\gamma}}$, where $\gamma < 1$ (usually we will take
  $\gamma$ as small as possible, namely $\gamma \sim 1 / 2$).
\end{remark}

\begin{theorem}
  \label{sec4.1 thm existence YDE}Let $\gamma > 1 / 2$, $\nu \in (0, 1)$ such
  that $\gamma (1 + \nu) > 1$ and assume that $T^w b \in C^{\gamma}_t
  C^{\nu}_x$. Then for any $\theta_0 \in \mathbb{R}^d$ there exists a solution
  $\theta \in C^{\gamma}$ to~{\eqref{sec4.1 YDE}}, defined on the whole
  interval $[0, T]$; furthermore, there exists a constant $C = C (\gamma
  \comma \nu, T)$ such that any solution to~{\eqref{sec4.1 YDE}} satisfies
  \begin{equation}
    \llbracket \theta \rrbracket_{C^{\gamma}} \leqslant C (1 + \| T^w b
    \|_{C^{\gamma}_t C_x^{\nu}}^2) \label{sec4.1 thm existence a priori
    estimate 1}
  \end{equation}
  as well as
  \begin{equation}
    \| \theta \|_{C^0} \leqslant C (1 + | \theta_0 | + \| T^w b
    \|_{C^{\gamma}_t C_x^{\nu}}^2) . \label{sec4.1 thm existence a priori
    estimate 2}
  \end{equation}
\end{theorem}

\begin{proof}
  The existence of solutions is granted under milder conditions on $T^w b$ by
  Theorem~2.9 from~{\cite{catelliergubinelli}}, so here we only show the
  a-priori estimates. Let $\theta \in C^{\gamma}$ be a solution and for any
  $\Delta > 0$ define the semi-norm
  \[ \llbracket \theta \rrbracket_{\gamma, \Delta} \assign
     \sup_{\tmscript{\begin{array}{c}
       s \neq t\\
       0 < | s - t | \leqslant \Delta
     \end{array}}} \frac{| \theta_{s, t} |}{| t - s |^{\gamma}} . \]
  Let $\Delta$ be a parameter to be fixed later; for any $s < t$ such that $|
  s - t | \leqslant \Delta$ it holds
  \begin{align*}
    | \theta_{s, t} | & = \, \left| \int_s^t T^w b (\mathd r, \theta_r)
    \right|\\
    & \leqslant \, | T^w b_{s, t} (\theta_s) | + C | t - s |^{\gamma (1 +
    \nu)}  \| T^w b \|_{C^{\gamma} C^{\nu}}^{}  \llbracket \theta
    \rrbracket^{\nu}_{\gamma, \Delta}\\
    & \leqslant | t - s |^{\gamma}  (\| T^w b \|_{C^{\gamma} C^{\nu}} + C
    \Delta^{\gamma \nu} \| T^w b \|_{C^{\gamma} C^{\nu}}^{}  \llbracket \theta
    \rrbracket^{\nu}_{\gamma, \Delta})\\
    & \leqslant | t - s |^{\gamma} (\| T^w b \|_{C^{\gamma} C^{\nu}} + C
    \Delta^{\gamma \nu} \| T^w b \|_{C^{\gamma} C^{\nu}}^{} + C
    \Delta^{\gamma, \nu} \| T^w b \|_{C^{\gamma} C^{\nu}}^{} \llbracket \theta
    \rrbracket_{\gamma, \Delta}),
  \end{align*}
  where in the last passage to used the trivial inequality $a^{\nu} \leqslant
  1 + a$ for all $a \geqslant 0$ and $\nu \in (0, 1]$. Dividing both sides by
  $| t - s |^{\gamma}$ and taking the supremum $s$, $t$ such that $| s - t |
  \leqslant \Delta$ we get
  \[ \llbracket \theta \rrbracket_{\gamma, \Delta} \leqslant \| T^w b
     \|_{C^{\gamma} C^{\nu}}  (1 + C \Delta^{\gamma \nu}) + C \Delta^{\gamma
     \nu} \| T^w b \|_{C^{\gamma} C^{\nu}}^{} \llbracket \theta
     \rrbracket_{\gamma, \Delta} . \]
  Choosing $\Delta$ small enough such that $C \Delta^{\gamma \nu} \| T^w b
  \|_{C^{\gamma} C^{\nu}} \leqslant 1 / 2$, we obtain
  \[ \llbracket \theta \rrbracket_{\gamma, \Delta} \leqslant 2 \| T^w b
     \|_{C^{\gamma} C^{\nu}}  (1 + C \Delta^{\gamma \nu}) \lesssim 1 + \| T^w
     b \|_{C^{\gamma} C^{\nu}} . \]
  If we can take $\Delta = T$, this provides an estimate for $\llbracket
  \theta \rrbracket_{C^{\gamma}}$, which together with $\| \theta \|_{C^0}
  \leqslant | \theta_0 | + T^{\gamma} \llbracket \theta
  \rrbracket_{C^{\gamma}}$ gives the conclusion. If this is not the case, we
  can choose $\Delta$ as above such that in addition $C \Delta^{\gamma \nu} \|
  T^w b \| \geqslant 1 / 4$ and then by the simple inequality (see for
  instance Exercise~4.24 from~{\cite{frizhairer}})
  \[ \llbracket \theta \rrbracket_{C^{\gamma}} \lesssim \llbracket \theta
     \rrbracket_{\gamma, \Delta} (1 + \Delta^{\gamma - 1}) . \]
  It follows that
  \begin{align*}
    \llbracket \theta \rrbracket_{C^{\gamma}} & \lesssim \, (1 + \| T^w b
    \|_{C^{\gamma} C^{\nu}}) (1 + \Delta^{\gamma - 1})\\
    & \lesssim \, (1 + \| T^w b \|_{C^{\gamma} C^{\nu}}) (1 + \| T^w b
    \|_{C^{\gamma} C^{\nu}}^{(1 - \gamma) / (\gamma \nu)})\\
    & \lesssim \, 1 + \| T^w b \|_{C^{\gamma}}^2,
  \end{align*}
  where in the last line we used the fact that $\gamma (1 + \nu) > 1$ implies
  $(1 - \gamma) / (\gamma \nu) < 1$. The conclusion again follows by the
  standard inequality $\| \theta \|_{C^{\gamma}} \lesssim | \theta_0 | +
  T^{\gamma} \llbracket \theta \rrbracket_{C^{\gamma}}$.
\end{proof}

Given that in general we consider $\gamma$ to be very close to $1 / 2$, in
order to have existence in general we need $\nu$ to be arbitrarily close to
$1$, thus we will usually require directly $T^w b \in C^{\gamma}_t
\tmop{Lip}_x$ (with the quantities $\| T^w b \|_{C^{\gamma} \tmop{Lip}}$ and
$\llbracket T^w b \rrbracket_{C^{\gamma} \tmop{Lip}}$ defined as above).

\

To establish uniqueness of solutions, we need the following lemma of
independent interest.

\begin{lemma}
  \label{sec4.1 technical lemma}Let $\gamma, \nu, \rho \in (0, 1]$ be such
  that $\gamma + \nu \rho > 1$, $A \in C^{\gamma}_t C^{1 + \nu}_x$; then for
  any $\theta^1$ and $\theta^2 \in C^{\rho}$ it holds
  \[ \int_0^t A (\mathd s, \theta^1_s) - \int_0^t A (\mathd s, \theta^2_s) =
     \int_0^t (\theta^1_s - \theta^2_s) \cdot \mathd V_s, \]
  where $V \in C^{\gamma} ([0, T] ; \mathcal{L} (\mathbb{R}^d ;
  \mathbb{R}^d))$ is given by
  \[ V_{\cdot} = \int_0^1 \int_0^{\cdot} \nabla A (\mathd s, \theta^2_s + x
     (\theta^1_s - \theta^2_s)) \mathd x. \]
  The integral is meaningful as a Bochner integral and $\| V \|_{C^{\gamma}
  \mathcal{L}} \lesssim \| A \|_{C^{\gamma} C^{1 + \nu}} (1 + \llbracket
  \theta^1 \rrbracket_{C^{\rho}} + \llbracket \theta^2
  \rrbracket_{C^{\rho}})$.
\end{lemma}

\begin{proof}
  Suppose first that in addition $\partial_t A \in C^0_t C^2_x$, then by
  Taylor expansion
  \begin{align*}
    \int_0^t A (\mathd s, \theta^1_s) - \int_0^t A (\mathd s, \theta^2_s) & =
    \int_0^t [\partial_t A (s, \theta^1_s) - \partial_t A (s, \theta^2_s)]
    \mathd s\\
    & = \int_0^t (\theta^1_s - \theta^2_s) \cdot \int_0^1 \partial_t \nabla A
    (s, \theta^2_s + x (\theta^1_s - \theta^2_s)) \mathd x \mathd s\\
    & = \int_0^t (\theta^1_s - \theta^2_s) \cdot \frac{\mathd}{\mathd s}
    \left( \int_0^s \int_0^1 \partial_t \nabla A (u, \theta^2_u + x
    (\theta^1_u - \theta^2_u)) \mathd x \mathd u \right) \mathd s\\
    & = \int_0^t (\theta^1_s - \theta^2_s) \cdot \mathd \left( \int_0^1
    \int_0^s \partial_t \nabla A (u, \theta^2_u + x (\theta^1_u - \theta^2_u))
    \mathd u \mathd x \right)\\
    & = \int_0^t (\theta^1_s - \theta^2_s) \cdot \mathd \left( \int_0^1
    \int_0^s \nabla A (\mathd u, \theta^2_u + x (\theta^1_u - \theta^2_u))
    \mathd x \right)\\
    & \backassign \int_0^t (\theta^1_s - \theta^2_s) \cdot \mathd V_s
  \end{align*}
  where all manipulations in this case are allowed by the properties of Young
  integral and the fact that we are assuming $A$ regular; in particular by
  hypothesis $\nabla A \in C^{\gamma}_t C^{\nu}_x$ and $\theta^i \in
  C^{\rho}_t$ with $\gamma + \nu \rho > 1$, so the interpretation of the
  integrals as nonlinear Young integrals is legit. Observe that the map $A
  \mapsto V (A)$ is linear by construction and we have the estimate
  \begin{align*}
    \| V \|_{C^{\gamma} \mathcal{L}} & = \, \left\| \int_0^1 \int_0^{\cdot}
    \nabla A (\mathd u, \theta^2_u + x (\theta^1_u - \theta^2_u)) \mathd x
    \right\|_{C^{\gamma} \mathcal{L}}\\
    & \leqslant \int_0^1 \left\| \int_0^{\cdot} \nabla A (\mathd u,
    \theta^2_u + x (\theta^1_u - \theta^2_u)) \right\| \mathd x\\
    & \lesssim \int_0^1 \| \nabla A \|_{C^{\gamma} C^{\nu}} (1 + \llbracket
    \theta^1 \rrbracket_{C^{\rho}} + \llbracket \theta^2
    \rrbracket_{C^{\rho}}) \mathd x,
  \end{align*}
  which gives the conclusion in this case. The general case follows by
  approximation, considering a sequence of regular $A^n \rightarrow A$ locally
  in $C^{\gamma - \delta}_t C^{1 + \nu - \delta}_x$, on a ball of radius $R >
  \| \theta^i \|_{\infty}$, for $\delta$ small enough such that $\gamma -
  \delta + (\nu - \delta) \rho > 1$.
\end{proof}

With the above lemma at hand, we can provide a comparison principle, which
estimates the difference between solutions. It comes in two versions, which
apply to different scenarios.

\begin{theorem}[Comparison Principle, Version 1]
  \label{sec4.1 comparison thm v1}Let $\gamma > 1 / 2$ and assume that $b^1$,
  $b^2$ are distributional drifts such that $T^w b^i \in C^{\gamma}_t C^2_x$
  with $\| T^w b^i \|_{C^{\gamma} C^2} \leqslant R$. Let $\theta^i \in
  C^{\gamma}$, $i = 1, 2$ be solutions respectively of the YDEs
  \[ \theta_t^i = \theta_0^i + \int_0^t T^w b^i (\mathd s, \theta^i_s)  \quad
     \forall \, t \in [0, T] . \]
  Then there exists a constant $C = C (\gamma, T, R)$ such that
  \begin{equation}
    \| \theta^1_{\cdot} - \theta^2_{\cdot} \|_{C^{\gamma}} \leqslant C (|
    \theta^1_0 - \theta^2_0 | + \| T^w b^1 - T^w b^2 \|_{C^{\gamma}_t
    \tmop{Lip}}) . \label{sec4.1 comparison v1.1}
  \end{equation}
  Similarly, let $b^i$ be s.t. $b^i \in L^{\infty}_{t, x}$ and $T^w b^i
  \in C^{\gamma}_t C^{3 / 2}_x$ with $\max_i \{ \| b^i \|_{L^{\infty}}, \| T^w
  b^i \|_{C^{\gamma} C^{3 / 2}} \} \leqslant R$, let $\theta^i$ be Lipschitz
  solutions of the YDEs; then there exists $\tilde{C} = \tilde{C} (\gamma, T,
  R)$ such that
  \begin{equation}
    \| \theta^1_{\cdot} - \theta^2_{\cdot} \|_{C^{\gamma}} \leqslant \tilde{C}
    (| \theta^1_0 - \theta^2_0 | + \| T^w b^1 - T^w b^2 \|_{C^{\gamma}_t
    \tmop{Lip}}) . \label{sec4.1 comparison v1.2}
  \end{equation}
\end{theorem}

\begin{proof}
  We show in detail the derivation of~{\eqref{sec4.1 comparison v1.1}} and
  briefly sketch the one of~{\eqref{sec4.1 comparison v1.2}} as the structure
  of the proof is the same. By the assumptions and Lemma~\ref{sec4.1 technical
  lemma} applied to $A = T^w b^1$, which is allowed for the choice $\nu = 1$,
  $\rho = \gamma$, the difference $v = \theta^1 - \theta^2$ satisfies
  \begin{align*}
    v_t & = \, v_0 + \left[ \int_0^t T^w b^1 (\mathd s, \theta^1_s) - \int_0^t
    T^w b^1 (\mathd s, \theta^2_s) \right] + \left[ \int_0^t T^w b^1 (\mathd
    s, \theta^2_s) - \int_0^t T^w b^2 (\mathd s, \theta^2_s) \right]\\
    & = v_0 + \int_0^t v_s \cdot \mathd V_s + \psi_t .
  \end{align*}
  This is a linear Young differential equation, for which standard estimates
  are available; Theorem~\ref{sec4.1 thm existence YDE}, Lemma~\ref{sec4.1
  technical lemma} and properties of nonlinear Young integral provide
  \[ \llbracket \theta^i \rrbracket_{C^{\gamma}} \lesssim \, 1, \quad \| V
     \|_{C^{\gamma} \mathcal{L}} \lesssim \| T^w b^1 \|_{C^{\gamma} C^2} (1 +
     \llbracket \theta^1 \rrbracket_{C^{\gamma}} + \llbracket \theta^2
     \rrbracket_{C^{\gamma}}) \lesssim 1, \]
  \[ \llbracket \psi \rrbracket_{C^{\gamma}} \lesssim \| T^w b^1 - T^w b^2
     \|_{C^{\gamma}_t \tmop{Lip}} \]
  where the constants appearing all depend on $\gamma, T, R$. Combining this
  estimates with Lemma~\ref{appendixA1 lemma bound linear YDE} from
  Appendix~\ref{appendixA1} yields the conclusion.
  
  The proof in the second case is analogue, but we have the additional
  estimate $\llbracket \theta^i \rrbracket_{\tmop{Lip}} \leqslant \| b^i
  \|_{L^{\infty}}$ coming from the ODE integral interpretation of the YDE and
  so we can apply as above Lemma~\ref{sec4.1 technical lemma} to $T^w b^1$
  this time for the choice $\nu = 1 / 2$, $\rho = 1$.
\end{proof}

\begin{remark}
  It follows immediately from the above result that if $T^w b \in C^{\gamma}_t
  C_x^2$ or $b \in L^{\infty}_{t, x}$ and $T^w b \in C^{\gamma}_t C^{3 /
  2}_x$, then for any $\theta_0 \in \mathbb{R}^d$ there exists a unique
  solution to the YDE~{\eqref{sec4.1 YDE}} and moreover the solution map
  $\theta_0 \mapsto \theta_{\cdot}$ is Lipschitz continuous w.r.t $\theta_0$;
  the solution constructed this way is also stable under approximation of $T^w
  b$ by other drifts $T^w \tilde{b}$, which can be combined with
  Lemma~\ref{sec3.1 lemma 3}, as we can take $\tilde{b} = b^{\varepsilon} =
  \rho^{\varepsilon} \ast b$ for some spatial mollifier $\rho^{\varepsilon}$.
\end{remark}

The above version of the Comparison Principle is of straightforward
application, as it only requires good regularity estimates on $T^w b$. The
next version is instead slightly more subtle and can be regarded as a
conditional Comparison Principle, as it allows to deduce estimates under less
regularity on $T^w b$ imposing the existence of a solution with suitable
properties; however, the existence of such solutions is not granted a priori
by the deterministic theory and in order to construct them probabilistic tools
will be needed, specifically Girsanov transform.

\begin{theorem}[Comparison Principle, Version 2]
  \label{sec4.1 comparison thm v2}Let $\gamma > 1 / 2$ and assume that $b^1$,
  $b^2$ are distributional drifts such that $T^w b^i \in C^{\gamma}_t
  \tmop{Lip}_x$ with $\| T^w b^i \|_{C^{\gamma} \tmop{Lip}} \leqslant R$. Let
  $\theta^i \in C^{\gamma}$, $i = 1, 2$ be solutions respectively of the YDEs
  \[ \theta_t^i = \theta_0^i + \int_0^t T^w b^i (\mathd s, \theta^i_s)  \quad
     \forall \, t \in [0, T] . \]
  and assume that $\theta^1$ is such that $T^{w + \theta^1} b \in C^{\gamma}_t
  \tmop{Lip}_x$ with $\| T^{w + \theta^1} b^1 \|_{C^{\gamma} \tmop{Lip}}
  \leqslant R$. Then there exists a constant $C = C (\gamma, T)$ such that
  \begin{equation}
    \| \theta^1_{\cdot} - \theta^2_{\cdot} \|_{C^{\gamma}} \leqslant C \exp (C
    R^{1 / \gamma})^{} (| \theta^1_0 - \theta^2_0 | + \| T^w b^1 - T^w b^2
    \|_{C^{\gamma}_t \tmop{Lip}}) . \label{sec4.1 comparison v2}
  \end{equation}
\end{theorem}

The proof requires the following technical lemma.

\begin{lemma}
  \label{sec4.1 technical lemma 2}Let $w$, $\theta$ be such that $T^w b$,
  $T^{w + \theta} b \in C^{\gamma}_t \tmop{Lip}_x$ and $\theta \in C^{1 /
  2}_t$ for some $\gamma > 1 / 2$. Then for any $\tilde{\theta} \in C^{1 /
  2}_t$ it holds
  \begin{equation}
    \int_0^{\cdot} T^{w + \theta} b (\mathd s, \tilde{\theta}_s) =
    \int_0^{\cdot} T^w b (\mathd s, \tilde{\theta}_s + \theta_s) .
    \label{sec4.1 eq techlem}
  \end{equation}
\end{lemma}

\begin{proof}
  If $b$ is jointly continuous in $(t, x)$ then the result is straightforward
  by the equivalence between the Young integral formulation and the standard
  integral formulation. Next, if $b$ satisfies the hypothesis and in addition
  $b \in L^1_t C^{\alpha}_x$ for some $\alpha > 0$, then for any $t > 0$ and
  for any sequence of partitions $\Pi_n$ of $[0, t]$ such that $| \Pi_n |
  \rightarrow 0$ it holds
  \begin{eqnarray*}
    \left| \int_0^t T^{w + \theta} b (\mathd s, \tilde{\theta}_s) \right. & -
    & \left. \int_0^t T^w b (\mathd s, \theta_s + \tilde{\theta}_s) \right|\\
    & = & \lim_{n \rightarrow \infty} \left| \sum_i \int_{t_i}^{t_{i + 1}} b
    (s, w_s + \theta_s + \tilde{\theta}_{t_i}) - b (s, w_s + \theta_{t_i} +
    \tilde{\theta}_{t_i}) \mathd s \right|\\
    & \leqslant & \lim_{n \rightarrow \infty} \sum_i \int_{t_i}^{t_{i + 1}}
    \| b (s) \|_{C^{\alpha}_x}  | \theta_s - \theta_{t_i} |^{\alpha} \mathd
    s\\
    & \leqslant & \| \theta \|_{C^{1 / 2}} \lim_{n \rightarrow \infty} |
    \Pi_n |^{\alpha / 2} \sum_i \int_{t_i}^{t_{i + 1}} \| b (s)
    \|_{C^{\alpha}_x} \mathd s = 0
  \end{eqnarray*}
  which proves the statement in this case. For a general $b$, consider
  $b^{\varepsilon} = \rho^{\varepsilon} \ast b$, where $\rho^{\varepsilon}$ is
  a sequence of spatial mollifiers; for $b^{\varepsilon}$ by the previous step
  identity~{\eqref{sec4.1 eq techlem}} is true and by Lemma~\ref{sec3.1 lemma
  3} $T^w b^{\varepsilon} \rightarrow T^w b$ locally in $C^{\gamma - \delta}_t
  C^{1 - \delta}_x$, similarly for $T^{w + \theta} b^{\varepsilon} \rightarrow
  T^{w + \theta} b$. Choosing $\delta$ small such that $\gamma - \delta + (1 -
  \delta) \gamma > 1$ and using the continuity of Young integral we obtain the
  conclusion in the general case.
\end{proof}

\begin{proof}[of Theorem~\ref{sec4.1 comparison thm v2}]
  The idea of the proof is the same as that of Theorem~\ref{sec4.1 comparison
  thm v1}, and it is based on finding a Young differential equation for $v =
  \theta^2 - \theta^1$, only we now need to exploit the additional information
  on $T^{w + \theta^1} b^1$. By the assumptions combined with
  Lemma~\ref{sec4.1 technical lemma 2}, $v$ satisfies
  \begin{eqnarray*}
    v_t & = & v_0 + \left[ \int_0^t T^w b^1 (\mathd s, \theta^2_s) - \int_0^t
    T^w b^1 (\mathd s, \theta^1_s) \right] + \left[ \int_0^t T^w b^2 (\mathd
    s, \theta^2_s) - \int_0^t T^w b^1 (\mathd s, \theta^2_s) \right]\\
    & = & v_0 + \int_0^t T^{w + \theta^1} (\mathd s, v_s) - \int_0^t T^{w +
    \theta^1} (\mathd s, 0) + \psi_t\\
    & = & v_0 + \int_0^t A (\mathd s, v_s) + \psi_t,
  \end{eqnarray*}
  where $\psi_t$ is defined in the usual way and $A (t, x) = T^{w + \theta^1}
  (t, x) - T^{w + \theta^1} (t, 0)$, so that $A \in C^{\gamma}_t \tmop{Lip}_x$
  with $\llbracket A \rrbracket_{C^{\gamma} \tmop{Lip}} = \llbracket T^{w +
  \theta^1} b \rrbracket_{C^{\gamma} \tmop{Lip}}$ and $A (t, 0) = 0$ for all
  $t \in [0, T]$. We can then apply the estimates from Lemma~\ref{appendixA1
  lemma bound nonlinear YDE} from Appendix~\ref{appendixA1} to deduce
  \[ \| v \|_{C^{\gamma}} \lesssim \exp (C \llbracket A
     \rrbracket_{C^{\gamma}}^{1 / \gamma}) (| v_0 | + \llbracket \psi
     \rrbracket_{C^{\gamma}}) \]
  for some constant $C = C (\gamma, T)$ which together with the estimate
  \[ \llbracket \psi \rrbracket_{C^{\gamma}} \lesssim \| T^w b^2 - T^w b^1
     \|_{C^{\gamma} \tmop{Lip}} (1 + \llbracket \theta^1
     \rrbracket_{C^{\gamma}} + \llbracket \theta^2 \rrbracket_{C^{\gamma}})
     \lesssim \| T^w b^2 - T^w b^1 \|_{C^{\gamma} \tmop{Lip}} (1 + R^2) \]
  yields the conclusion.
\end{proof}

\begin{remark}
  \label{sec4.1 final remark}It follows immediately from Theorem~\ref{sec4.1
  comparison thm v2} that, if there exists a solution $\theta$ to the YDE
  associated to $T^w b$ with initial data $\theta_0$ such that $T^{w + \theta}
  \in C^{\gamma}_t \tmop{Lip}_x$, then this is necessarily the unique solution
  with initial data $\theta_0$ and it is stable under perturbation. This
  provides a nice ``duality principle'': existence of solutions is granted if
  $T^w b \in C^{\gamma}_t \tmop{Lip}_x$, uniqueness instead if there exists a
  solution with similar averaging properties. In the case $b$ is continuous,
  so that by Peano Theorem existence of a solution $x = w + \theta \in w +
  \tmop{Lip}$ is automatic, the statement can be rephrased as the fact that
  uniqueness for the Cauchy problem associated to $x_0$ holds under the
  condition $T^x b \in C^{\gamma}_t \tmop{Lip}_x$ for some $\gamma > 1 / 2$.
\end{remark}

\begin{remark}
  \label{sec4.1 remark localization}For the sake of simplicity we considered
  from the start $T^w b \in C^{\gamma}_t C^{\beta}_x$ in order to develop a
  global theory in space, but many results from Section~\ref{sec4} can be
  localised, thanks to Remark~\ref{sec3.1 remark localization}, in a similar
  fashion to what is done in Section~2.3 of~{\cite{catelliergubinelli}}. For
  instance local existence holds for $T^w b \in C^{\gamma}_t
  \tmop{Lip}_{\tmop{loc}}$, while local existence and uniqueness holds for
  $T^w b \in C^{\gamma}_t C^2_{\tmop{loc}}$; in the second version of the
  Comparison Principle, if there exists a solution $x$ defined on $[0,
  T^{\ast})$ such that $T^x b \in C^{\gamma}_t \tmop{Lip}_{\tmop{loc}}$, then
  it is the unique solution on $[0, T^{\ast})$. Analogue considerations hold
  for the results from Section~\ref{sec4.3} on the regularity of the flow.
\end{remark}

\subsection{Prevalence for the Cauchy problem}\label{sec4.2}

In this section we focus on establishing conditions under which, for a given
drift $b$ and a given initial datum $x_0 \in \mathbb{R}^d$, for almost every
$\varphi \in C^{\delta}_t$ the Cauchy problem (from now on referred to as
$(\tmop{CP}_{x_0})$)
\begin{equation}
  x_t = x_0 + \int_0^t b (s, x_s) + \varphi_t \label{sec4.2 cauchy problem}
\end{equation}
is well-posed, for suitable values of $\delta$. Here by well-posedness for
$(\tmop{CP}_{x_0})$ we mean the following: $\varphi$ is such that $T^{\varphi}
b \in C^{\gamma}_t \tmop{Lip}_x$ for some $\gamma > 1 / 2$, so that it makes
sense to talk about solutions to~{\eqref{sec4.2 cauchy problem}} in the sense
of Definition~\ref{sec4.1 defn solution}, and there exists a unique such
solution in the class $x \in \varphi + C^{\gamma}$. The main results we are
going to prove are the following.

\begin{theorem}
  \label{sec4.2 thm main 1}Let $b \in C^{\alpha}_x$ for some $\alpha \in (-
  \infty, 1)$, $b$ being compactly supported, and let $x_0 \in \mathbb{R}^d$
  be fixed. Let $\delta \in [0, 1)$ satisfy
  \[ \delta < \frac{1}{2 (1 - \alpha)} . \]
  Then for almost every $\varphi \in C^{\delta}_t$ the Cauchy problem
  $(\tmop{CP}_{x_0})$ is well-posed.
\end{theorem}

\begin{theorem}
  \label{sec4.2 thm main 2}Let $b \in C^{\alpha}_x$ for some $\alpha \in (-
  \infty, 1)$, $b$ being compactly supported, and let $x_0 \in \mathbb{R}^d$
  be fixed. Let $\mu^H$ denote the law of fBm of parameter $H$ and suppose
  that
  \[ \alpha > 1 - \frac{1}{2 H} . \]
  Then path\mbox{-}by\mbox{-}path uniqueness holds for $(\tmop{CP}_{x_0})$ and
  $w$ sampled according to $\mu^H$. Moreover there exists $\gamma > 1 / 2$
  which only depends on $\alpha$ such that
  \[ \mu^H \left( w \, : T^w b \in C^{\gamma}_t \tmop{Lip}_x, \,\exists \text{ a solution $x \in w + C^{\gamma}$ s.t. } T^x b \in C^{\gamma}_t
     \tmop{Lip}_x \right) = 1. \]
\end{theorem}

In the second statement we have used the terminology
``path\mbox{-}by\mbox{-}path uniqueness'' as it appears frequently in
regularisation by noise results, see~{\cite{flandoli}}, but in the framework
introduced above it just amounts to stating that there exists $\gamma > 1 / 2$
such that
\[ \mu^H \left( w \, : T^w b \in C^{\gamma}_t \tmop{Lip}_x \text{ and
   $(\tmop{CP}_{x_0})$ is well-posed} \right) = 1. \]
The section is organised as follows: we first prove Theorem~\ref{sec4.2 thm
main 1} in Section~\ref{sec4.2.1} relying on the validity of
Theorem~\ref{sec4.2 thm main 2}; then we pass to the proof of the latter,
which is based on an application of Theorem~\ref{sec4.1 comparison thm v2} in
combination with Girsanov transform for fBm, which is introduced in
Section~\ref{sec4.2.2}. The proof of Theorem~\ref{sec4.2 thm main 2} is
completed in Section~\ref{sec4.2.3}, along with several other results of the
same nature. We leave the details to the following subsections, but let us
point out already here that we will exploit crucially the general principle

\begin{center}
  \begin{tabular}{|l|}
    \hline
    $T^{W^H} b \in C^{\gamma}_t \tmop{Lip}_x$\\
    \hline
  \end{tabular}\enspace +\enspace\begin{tabular}{|l|}
    \hline
    Girsanov\\
    \hline
  \end{tabular}\enspace$\Longrightarrow$\enspace\begin{tabular}{|l|}
    \hline
    path-by-path uniqueness\\
    \hline
  \end{tabular}.
\end{center}

Such a principle is not new and was crucially exploited in~{\cite{davie}}
and~{\cite{catelliergubinelli}}. However, we believe it is the first time it
is properly formalised as in Lemma~\ref{sec4.2.3 criterion path-by-path
uniqueness} and its general structure allows to apply it in other situations.

\subsubsection{Proof of Theorem~\ref{sec4.2 thm main 1}}\label{sec4.2.1}

We need a few preparations first. Recall that in order to establish prevalence
of well-posedness for $(\tmop{CP}_{x_0})$ in $C_t^{\delta}$, we need to find a
set $\mathcal{A} \subset C_t^{\delta}$ and a tight probability $\mu$ on
$C_t^{\delta}$ such that: i) $\mathcal{A}$ is Borel w.r.t. the topology of
$C_t^{\delta}$; ii) for all $w \in \mathcal{A}$, $(\tmop{CP}_{x_0})$ is
well-posed; iii) for all $\varphi \in C_t^{\delta}$, $\mu (\varphi +
\mathcal{A}) = 1$.

A good candidate for the set $\mathcal{A}$ is given by Theorem~\ref{sec4.1
comparison thm v2} as follows: for $\gamma > 1 / 2$, define
\begin{equation}
  \mathcal{A}_{\gamma} = \left\{ w \in C^{\delta}_t : \, T^w b \in
  C^{\gamma}_t \tmop{Lip}_x,\,\exists \text{ a solution $x \in w +
  C^{\gamma}$ s.t. } T^x b \in C^{\gamma}_t \tmop{Lip}_x \right\} .
  \label{sec4.2.1 defn set}
\end{equation}
For such an $\mathcal{A}_{\gamma}$, it is now rather clear by the statement of
Theorem~\ref{sec4.2 thm main 2} that we plan to use as a measure $\mu^H$ for
suitable choice of $H$. But we first need to check that condition i) holds,
which is the aim of the following lemma.

\begin{lemma}
  \label{sec4.2.1 lemma borel set uniqueness}Let $\gamma > 1 / 2$, then the
  set $\mathcal{A}_{\gamma}$ is Borel measurable in the topology of
  $C^{\delta}$ for any $\delta \geqslant 0$.
\end{lemma}

\begin{proof}
  The idea of the proof is the usual one: we write the set
  $\mathcal{A}_{\gamma}$ as the countable union
  \begin{align*}
\mathcal{A}_{\gamma} = \bigcup_{N \geqslant 1} \mathcal{A}_N 
\assign \bigcup_{N \geqslant 1} \left\{ w \in E : \, \| T^w b
     \|_{C^{\gamma}_t \tmop{Lip}_x} \leqslant N,\,\exists \text{ a solution
     $x$ s.t. } \| T^x b \|_{C^{\gamma}_t \tmop{Lip}_x} \leqslant N \right\} .
  \end{align*}
  In order to conclude it is then sufficient to show that, for each $N$, the
  set $\mathcal{A}_N$ is closed under the topology of $C^{\delta}$. We can
  restrict ourselves to the case $C^0$, since any other convergence we
  consider is stronger than this one.
  
  Let $w^n$ be a sequence of elements of $\mathcal{A}_N$ such that $w^n
  \rightarrow w$, then by Lemma~\ref{sec3.1 lemma 4} we know that $T^w b \in
  C^{\gamma}_t \tmop{Lip}_x$ with the bound $\| T^w b \|_{C^{\gamma}_t
  \tmop{Lip}_x} \leqslant N$. For each $n$, denote by $x^n = \theta^n + w^n$
  the associated solution of $(\tmop{CP}_{x_0})$ such that $\| T^{x^n} b
  \|_{C^{\gamma}_t \tmop{Lip}_x} \leqslant N$; by the a priori estimates from
  Theorem~\ref{sec4.1 thm existence YDE}, together with $\| T^{w^n} b
  \|_{C^{\gamma}_t \tmop{Lip}_x} \leqslant N$, we deduce that $\| \theta^n
  \|_{C^{\gamma}}$ are uniformly bounded. We can therefore (up to subsequence)
  consider $\theta^n \rightarrow \theta$ in $C^{\gamma - \varepsilon}$ for
  suitable $\varepsilon > 0$. Since $w^n + \theta^n \rightarrow w + \theta$ in
  $C^0$, again it must hold $\| T^{w + \theta} b \|_{C^{\gamma}_t
  \tmop{Lip}_x} \leqslant N$.
  
  In order to conclude it remains to show that $x = \theta + w$ is a solution
  of the $(\tmop{CP}_{x_0})$ associated to $T^w b$. Since the sequence
  $T^{w^n} b \rightarrow T^w b$ in the sense of distributions and it is
  uniformly bounded in $C^{\gamma}_t \tmop{Lip}_x$, reasoning as in the proof
  of Lemma~\ref{sec3.1 lemma 3} we deduce that also local convergence in \
  $C^{\gamma - \varepsilon}_t C^{1 - \varepsilon}_x$ holds, for any
  $\varepsilon > 0$. Choosing $\varepsilon > 0$ small enough such that $\gamma
  - \varepsilon + (1 - \varepsilon) / (\gamma - \varepsilon) > 1$, by
  continuity of nonlinear Young integral it holds $\int_0^{\cdot} T^{w^n} b
  (\mathd s, \theta^n_s) \rightarrow \int_0^{\cdot} T^w b (\mathd s,
  \theta_s)$ in $C^{\gamma - \varepsilon}$. Taking the limit as $n \rightarrow
  \infty$ of
  \[ \theta^n_t = x_0 - w^n_0 + \int_0^t T^{w^n} b (\mathd s \comma
     \theta^n_s) \]
  we deduce that $x$ is a solution w.r.t. $T^w b$ of $(\tmop{CP}_{x_0})$,
  which concludes the proof.
\end{proof}

\begin{proof*}{Proof of Theorem~\ref{sec4.2 thm main 1}}
  In order to conclude it suffices to show that we can find $\gamma > 1 / 2$
  and $H > \delta$ such that $\mu^H (\varphi + \mathcal{A}_{\gamma}) = 1$ for
  all $\varphi \in C_t^{\delta}$. Let us choose $\varepsilon > 0$ small enough
  such that
  \begin{equation}
    H \assign \delta + \varepsilon < \frac{1}{2 (1 - \alpha)} \label{sec4.2.1
    eq1} .
  \end{equation}
  We need to find $\gamma > 1 / 2$ such that for any fixed $\varphi \in
  C_t^{\delta}$,
  \[ \mu^H \left( w \in C_t^{\delta} : \, T^{w + \varphi} b \in C^{\gamma}_t
     \tmop{Lip}_x, \text{ $\exists$ $x \in (w + \varphi) + C_t^{\gamma}$
     solution to $(\tmop{CP}_{x_0})$ s.t. } T^x b \in C^{\gamma}_t
     \tmop{Lip}_x \right) = 1 \]
  By definition of the averaging operator we have $T^{w + \varphi} b = T^w
  \tilde{b}$, where $\tilde{b} (t, \cdot) = b \left( t, \cdot \, + \varphi_t
  \right)$; moreover, $x \in (w + \varphi) + C_t^{\gamma}$ solves
  $(\tmop{CP}_{x_0})$ if and only if $\tilde{x} \assign x - \varphi \in w +
  C_t^{\gamma}$ is again a solution to another Cauchy problem of the same
  type. Indeed, by definition of solution, $\theta = x - (w + \varphi) =
  \tilde{x} - w$ must solve
  \[ \theta_t = x_0 - (w_0 + \varphi_0) + \int_0^t T^{w + \varphi} b (\mathd
     s, \theta_s) = \tilde{x}_0 - w_0 + \int_0^t T^w \tilde{b} (\mathd s,
     \theta_s) \]
  where $\tilde{x}_0 = x_0 - \varphi_0$, so that $\tilde{x}$ is a solution to
  the Cauchy problem associated to $\tilde{x}_0$, $\tilde{b}$ and $w$.
  Moreover by properties of averaging operators it holds $T^{\tilde{x}}
  \tilde{b} = T^x b$.
  
  By the translation invariance of the $C^{\alpha}_x$-norm, it holds
  $\tilde{b} \in C^{\alpha}_x$, $\| \tilde{b} \|_{C^{\alpha}} = \| b
  \|_{C^{\alpha}}$; moreover $\tilde{b}$ has still compact support in space,
  uniformly in time. Since condition~{\eqref{sec4.2.1 eq1}} implies $\alpha >
  1 - (2 H)^{- 1}$, we can apply Theorem~\ref{sec4.2 thm main 2} for the
  choice $\tilde{x}_0$, $T^w \tilde{b}$ to find $\gamma > 1 / 2$ (independent
  of $\varphi$) such that
  \begin{align*}
    1 & = \mu^H \left( w \in C_t^{\delta} : \, T^w \tilde{b} \in
    C^{\gamma}_t \tmop{Lip}_x \text{$\exists\, \tilde{x} \in w +
    C_t^{\gamma}$ solution to $(\tmop{CP}_{\tilde{x}_0})$ s.t. } T^{\tilde{x}}
    \tilde{b} \in C^{\gamma}_t \tmop{Lip}_x \right)\\
    & = \mu^H \left( w \in C_t^{\delta} : \, T^{w + \varphi} b \in
    C^{\gamma}_t \tmop{Lip}_x \text{ $\exists\, x \in (w + \varphi) +
    C_t^{\gamma}$ solution to $(\tmop{CP}_{x_0})$ s.t. } T^x b \in
    C^{\gamma}_t \tmop{Lip}_x \right)
  \end{align*}
  which gives the conclusion.
\end{proof*}

\begin{remark}
  For simplicity we have preferred to give the statement of
  Theorem~\ref{sec4.2 thm main 1} as above, but it will be clear from the
  contents of Section~\ref{sec4.2.3} that similar prevalence statements can be
  formulated under other hypothesis on $b$ and $\delta$ simply by going
  through the same proof and applying in the end either Theorem~\ref{sec4.2.3
  thm path-by-path uniqueness} or Corollary~\ref{sec4.2.3 corollary path
  uniqueness}.
\end{remark}

\subsubsection{Girsanov transform}\label{sec4.2.2}

Before introducing Girsanov Theorem, we need to recall another representation
formula for fBm, different from the one given in Section~\ref{sec2.2}, which
can be found in~{\cite{nualart2006}},~{\cite{picard}}. The representation is
based on fractional calculus, which we also quickly introduce and for which we
refer the interested reader to~{\cite{samko}}.

Given $f \in L^1 (0, T)$ and $\alpha > 0$, the fractional integral of order
$\alpha$ of $f$ is defined as
\begin{equation}
  (I^{\alpha} f)_{\cdot} = \frac{1}{\Gamma (\alpha)} \int_0^{\cdot} (t -
  s)^{\alpha - 1} f_s \mathd s \label{sec4.2.2 frac integral}
\end{equation}
where $\Gamma$ denotes the Gamma function. For $\alpha \in (0, 1)$ and $p >
1$, the map $I^{\alpha}$ is an injective bounded operator on $L^p$ and we
denote by $I^{\alpha} (L^p)$ the image of $L^p$ under the $I^{\alpha}$, which
is a Banach space endowed with the norm $\| f \|_{I^{\alpha} (L^p)} \assign \|
g \|_{L^p}$ if $f = I^{\alpha} g$. On this domain, $I^{\alpha}$ admits an
inverse, which is the fractional derivative of order $\alpha$, given by
\begin{equation}
  (D^{\alpha} f)_t = \frac{1}{\Gamma (1 - \alpha)}  \frac{\mathd}{\mathd x} 
  \int_0^t \frac{f_s}{(t - s)^{\alpha}} \mathd s = \frac{1}{\Gamma (1 -
  \alpha)}  \left( \frac{f_t}{t^{\alpha}} + \alpha \int_0^t \frac{f_t -
  f_s}{(t - s)^{\alpha + 1}} \mathd s \right) . \label{sec4.2.2 frac
  derivative}
\end{equation}
With this notation in mind, a fBm of Hurst parameter $H \in (0, 1)$ can be
constructed starting from a standard Brownian motion $B$ on the interval $[0,
T]$ by setting $W^H = K_H (\mathd B)$, where the operator $K_H$ is defined as
\[ K_H f = \left\{\begin{array}{ll}
     I^1 s^{H - 1 / 2} I^{H - 1 / 2} s^{1 / 2 - H} h & \tmop{if} H \geqslant 1
     / 2\\
     I^{2 H} s^{1 / 2 - H} I^{1 / 2 - H} s^{H - 1 / 2} h & \tmop{if} H
     \leqslant 1 / 2
   \end{array}\right. \]
where the notation $s^{\beta}$ denotes the multiplication operator with the
function $s \mapsto s^{\beta}$. It can be shown that this definition of $W^H$
is meaningful and that the operator $K_H$ corresponds to a Volterra kernel
$K_H (t, s)$, so that the above representation is equivalent to
\begin{equation}
  W^H_t = \int_0^t K_H (t, s) \mathd B_s . \label{sec4.2.2 canonical
  representation}
\end{equation}
The explicit expression for $K_H$ in the case $H > 1 / 2$ is given by
\begin{equation}
  K_H (t, s) = c_H s^{1 / 2 - H} \int_s^t (u - s)^{H - 3 / 2} u^{H - 1 / 2}
  \mathd u ; \label{sec4.2.2 volterra kernel}
\end{equation}
in the case $H < 1 / 2$ it is more complicated and we omit it as we will not
need it. It can be shown that the operator $K_H$ can be inverted, which
implies that the processes $B$ and $W^H$ generate the same filtration, which
makes it a {\tmem{canonical representation}}; moreover this implies that given
any fBm $W^H$ on a probability space, it is possible to construct the
associated $B$ by setting $B_{\cdot} = \int_0^{\cdot} (K^{- 1}_H W^H)_s \mathd
s$. The inverse operator $K_H^{- 1}$ is given by
\begin{equation}
  K^{- 1}_H f = \left\{\begin{array}{ll}
    s^{H - 1 / 2} D^{H - 1 / 2} s^{1 / 2 - H} f' & \tmop{if} H > 1 / 2\\
    s^{1 / 2 - H} D^{1 / 2 - H} s^{H - 1 / 2} D^{2 H} f \quad & \tmop{if} H <
    1 / 2
  \end{array}\right. . \label{sec4.2.2 inverse operator}
\end{equation}
We will use the following terminology: given a filtered space $(\Omega,
\mathcal{F}, \{ \mathcal{F}_t \}_{t \geqslant 0}, \mathbb{P})$, we say that a
process $W^H$ is an $\mathcal{F}_t$\mbox{-}fBm if it is a fBm under
$\mathbb{P}$ and the associated $B$ is an $\mathcal{F}_t$\mbox{-}Bm in the
usual sense.

\begin{theorem}[Girsanov]
  \label{sec4.2.2 girsanov thm}Let $(\Omega, \mathcal{F}, \{ \mathcal{F}_t
  \}_{t \geqslant 0}, \mathbb{P})$ be a filtered probability space, $W^H$ be
  an $\mathcal{F}_t$\mbox{-}fBm of parameter $H \in (0, 1)$ and $h$ be an
  $\mathcal{F}_t$\mbox{-}adapted process with continuous trajectories s.t.
  $h_0 = 0$. Let $B$ be the Bm associated to $W^H$, namely such that $W^H =
  K_H \mathd B$. Suppose that $K_H^{- 1} h \in L_t^2$ with probability $1$ and
  that
  \begin{equation}
    \mathbb{E} \left[ \frac{\mathd \mathbb{P}}{\mathd \mathbb{Q}} \right] = 1
    \label{sec4.2.2 condition girsanov}
  \end{equation}
  where the variable $\mathd \mathbb{P}/ \mathd \mathbb{Q}$ is given by
  \begin{equation}
    \frac{\mathd \mathbb{P}}{\mathd \mathbb{Q}} = \exp \left( - \int_0^T
    (K_H^{- 1} h)_s \mathd B_s - \frac{1}{2} \int_0^T | (K_H^{- 1} h)_s |^2
    \mathd s \right) . \label{sec4.2.2 density girsanov}
  \end{equation}
  Then the shifted process $\tilde{W}^H \assign W^H + h$ is an
  $\mathcal{F}_t$-fBm with parameter $H$ under the probability $\mathbb{Q}$. A
  sufficient condition in order for~{\eqref{sec4.2.2 condition girsanov}} to
  hold is given by Novikov's condition
  \begin{equation}
    \mathbb{E} \left[ \exp \left( \frac{1}{2} \int_0^T | (K_H^{- 1} h)_s |^2
    \mathd s \right) \right] < \infty . \label{sec4.2.2 novikov}
  \end{equation}
\end{theorem}

The result is taken from~{\cite{nualartouknine}}, Theorem~2, with the
exception of the final part which is just classical Novikov condition; in the
original statement from~{\cite{nualartouknine}}, the process $h$ is taken of
the form $h_{\cdot} = \int_0^{\cdot} u_s \mathd s$, but this doesn't play any
role in the proof, which indeed holds also in the case $h$ is not of bounded
variation.

In order to apply Theorem~\ref{sec4.2.2 girsanov thm} in cases of interest, we
first need to establish conditions under which~{\eqref{sec4.2.2 novikov}}
holds, which requires a good control of $\| K^{- 1}_H h \|_{L^2}$ in terms of
$h$.

Since $K^{- 1}_H$ is defined in terms of fractional derivatives, the following
fact will be quite useful: if $f \in C^{\beta}$ and $f_0 = 0$, then
$D^{\alpha} f$ is well defined for any $\alpha < \beta$ and moreover
$D^{\alpha} f \in C^{\gamma}$ for any $\gamma < \beta - \alpha$ together with
the estimate
\begin{equation}
  \| D^{\alpha} f \|_{C^{\gamma}} \lesssim_{\gamma, \alpha} \| f
  \|_{C^{\beta}} . \label{sec4.2.2 holder fractional derivative}
\end{equation}
For a self-contained proof of this fact see Theorem~2.8 from {\cite{picard}}
(on a finite interval $[0, T]$, the space $\mathbb{H}^{\beta, 0}$ considered
therein corresponds to the functions $f \in C^{\beta}_t$ such that $f_0 = 0$).

\begin{lemma}
  \label{sec4.2.2 lemma fractional derivative}Let $\alpha \in (0, 1 / 2)$ and
  $h \in C_t^{\beta}$ for some $\beta > \alpha$, $h_0 = 0$. Then $s^{\alpha}
  D^{\alpha} s^{- \alpha} h \in L_t^2$ and there exists a constant $C = C
  (\alpha, \beta, T)$ such that
  \begin{equation}
    \| s^{\alpha} D^{\alpha} s^{- \alpha} h \|_{L^2} \leqslant C \| h
    \|_{C^{\beta}} . \label{sec4.2.2 estimate fractional derivative 1}
  \end{equation}
  In particular, for any $H \in (0, 1)$, if $h \in C_t^{\beta}$ for some
  $\beta > H + 1 / 2$, $h_0 = 0$, then $K^{- 1}_H \in L_t^2$ and there exists
  a constant $C = C (H, \beta, T)$ such that
  \begin{equation}
    \| K_H^{- 1} h \|_{L^2} \leqslant C \| h \|_{C^{\beta}} . \label{sec4.2.2
    estimate fractional derivative 2}
  \end{equation}
\end{lemma}

\begin{proof}
  We have
  \[ (s^{\alpha} D^{\alpha} s^{- \alpha} h) (t) = \Gamma (1 - \alpha)^{- 1}
     \left[ h_t + \alpha t^{\alpha} \int_0^t \frac{t^{- \alpha} h_t - s^{-
     \alpha} h_s}{(t - s)^{\alpha + 1}} \mathd s \right] . \]
  Since $h \in C^{\beta}$, it clearly also belongs to $L^2$, so we only need
  to control the term
  \begin{align*}
    t^{\alpha} \left| \int_0^t \, \frac{t^{- \alpha} h_t - s^{- \alpha}
    h_s}{(t - s)^{\alpha + 1}} \mathd s \right| & \leqslant \, t^{\alpha}
    \int_0^t \, \frac{t^{- \alpha} | h_t - h_s | + (s^{- \alpha} - t^{-
    \alpha}) | h_s |}{(t - s)^{\alpha + 1}} \mathd s\\
    & \leqslant \, \| h \|_{C^{\beta}} t^{\alpha} \int_0^t \, \frac{t^{-
    \alpha} (t - s)^{\beta} + (s^{- \alpha} - t^{- \alpha})}{(t - s)^{\alpha +
    1}} \mathd s\\
    & = \, \| h \|_{C^{\beta}} t^{- \alpha} \left[ t^{\beta} \int_0^1
    \frac{1}{(1 - u)^{1 + \alpha - \beta}} \mathd u + \int_0^1 u^{- \alpha} 
    \frac{(1 - u^{\alpha})}{(1 - u)^{1 + \alpha}} \mathd u \right]\\
    & \lesssim_T \, \| h \|_{C^{\beta}} t^{- \alpha} .^{}
  \end{align*}
  Since $\alpha \in (0, 1 / 2)$, $t^{- \alpha} \in L^2_t$ and so we deduce
  that the overall expression belongs to $L^2_t$, as well as
  estimate~{\eqref{sec4.2.2 estimate fractional derivative 1}}. Regarding the
  second statement, the case $H = 1 / 2$ is straightforward since $K^{- 1}_H h
  = h'$. In the case $H > 1 / 2$, by the formula for $K^{- 1}_H$ combined with
  estimates~{\eqref{sec4.2.2 holder fractional derivative}}
  and~{\eqref{sec4.2.2 estimate fractional derivative 1}} for the choice
  $\alpha = H - 1 / 2$, choosing $\varepsilon > 0$ sufficiently small we have
  \[ \| K^{- 1}_H h \|_{L^2} \lesssim \| h' \|_{C^{H - 1 / 2 + \varepsilon}}
     \lesssim \| h' \|_{C^{\beta}} ; \]
  the case $H < 1 / 2$ is analogous.
\end{proof}

\begin{remark}
  We have given an explicit proof of Lemma~\ref{sec4.2.2 lemma fractional
  derivative}, but a similar (stronger) type of result can be achieved by a
  more abstract argument. Indeed it follows from the proof of Theorem~5.4
  from~{\cite{picard}} that $\| s^{\alpha} D^{\alpha} (s^{- \alpha} h)
  \|_{L^2} \sim \| D^{\alpha} h \|_{L^2}$ and similarly $\| K^{- 1}_H h
  \|_{L^2} \sim \| D^{H + 1 / 2} h \|_{L^2}$; we have already seen that if $h
  \in C^{\beta}$ with $\beta > \alpha$ and $h_0 = 0$, then $D^{\alpha} h$ is a
  continuous function, so its $L^2$\mbox{-}norm is trivially finite. The
  inclusion $C^{\beta} \subset I^{\alpha} (L^2)$ is strict and therefore the
  hypothesis of Lemma~\ref{sec4.2.2 lemma fractional derivative} are non
  optimal, but they are rather useful when dealing with functions $h$ not of
  bounded variation.
\end{remark}

We can now state a general result on the applicability of Girsanov transform
together with a good control on the density defining $\mathbb{Q}$.

\begin{theorem}
  \label{sec4.2.2 theorem girsanov transform}Let $(\Omega, \mathcal{F}, \{
  \mathcal{F}_t \}_{t \geqslant 0}, \mathbb{P})$ be a filtered probability
  space, $W^H$ be an $\mathcal{F}_t$\mbox{-}fBm of parameter $H \in (0, 1)$
  and $h$ be an $\mathcal{F}_t$\mbox{-}adapted process with trajectories in
  $C^{\beta}_t$, $\beta > H + 1 / 2$, s.t. $h_0 = 0$ and
  \begin{equation}
    \mathbb{E} [\exp (\lambda \| h \|_{C^{\beta}}^2)] < \infty \quad \forall
    \, \lambda \in \mathbb{R}. \label{sec4.2.2 hypothesis application
    girsanov}
  \end{equation}
  Then Girsanov transform for \~{W}$^H = h + W^H$ is applicable, i.e.
  $\tilde{W}^H$ is an $\mathcal{F}_t$\mbox{-}fBm of parameter $H$ under the
  probability measure $\mathbb{Q}$ given by~{\eqref{sec4.2.2 density
  girsanov}}. Moreover the measures $\mathbb{Q}$ and $\mathbb{P}$ are
  equivalent and it holds
  \[ \mathbb{E}_{\mathbb{P}} \left[ \left( \frac{\mathd \mathbb{Q}}{\mathd
     \mathbb{P}} \right)^n + \left( \frac{\mathd \mathbb{P}}{\mathd
     \mathbb{Q}} \right)^n \right] < \infty \quad \forall \, n \in \mathbb{N}.
  \]
\end{theorem}

\begin{proof}
  By hypothesis~{\eqref{sec4.2.2 hypothesis application girsanov}} and
  Lemma~\ref{sec4.2.2 lemma fractional derivative} it follows immediately that
  \[ \mathbb{E} [\exp (\lambda \| K^{- 1} h \|_{L^2}^2)] < \infty \quad
     \forall \, \lambda \in \mathbb{R}. \]
  Therefore Novikov criterion is satisfied and Girsanov transform is
  applicable. The proof of second part of the statement follows from classical
  arguments, but we include it for the sake of completeness. Let us prove
  integrability of the moments: for any $\alpha \geqslant 1$ it holds
  \begin{align*}
    \mathbb{E}_{\mathbb{P}} \left[ \left( \frac{\mathd \mathbb{Q}}{\mathd
    \mathbb{P}} \right)^{\alpha} \right] & = \, \mathbb{E}_{\mathbb{P}} \left[
    \exp \left( \alpha \int_0^T (K^{- 1}_H h) \cdot \mathd B - \alpha^2 \|
    K^{- 1}_H h \|_{L^2}^2 + \left( \alpha^2 - \frac{\alpha}{2} \right)  \|
    K^{- 1}_H h \|_{L^2}^2 \right) \right]\\
    & \leqslant \, \mathbb{E}_{\mathbb{P}} \left[ \exp \left( 2 \alpha
    \int_0^T (K^{- 1}_H h) \cdot \mathd B - 2 \alpha^2  \| K^{- 1}_H h
    \|_{L^2}^2 \right) \right]^{1 / 2} \mathbb{E}_{\mathbb{P}} [\exp ((2
    \alpha^2 - \alpha) \| K^{- 1}_H h \|_{L^2}^2)]^{1 / 2}\\
    & = \, \mathbb{E}_{\mathbb{P}} [\exp ((2 \alpha^2 - \alpha) \| K^{- 1}_H
    h \|_{L^2}^2)]^{1 / 2} < \infty,
  \end{align*}
  where in the second line we used the fact that the integrand in the first
  term is again a probability density by Novikov's criterion, this time
  applied to the process $\tilde{h} = 2 \alpha h$. Now in order to show that
  the measures $\mathbb{Q}$ and $\mathbb{P}$ are equivalent, we need to show
  that the inverse density $\mathd \mathbb{P}/ \mathd \mathbb{Q}$ is
  integrable w.r.t. $\mathbb{Q}$. Again by Girsanov, since we have $W^H =
  \tilde{W}^H - h$, the inverse density is given by
  \[ \frac{\mathd \mathbb{P}}{\mathd \mathbb{Q}} = \exp \left( \int_0^T (K^{-
     1}_H h) (s) \cdot \mathd \tilde{B}_s - \frac{1}{2} \int_0^T | (K^{- 1}_H
     h) (s) |^2 \mathd s \right) \]
  where $\tilde{B}$ now denotes the standard Bm associated to $\tilde{W}^H$,
  i.e. such that $\tilde{W}^H_t = \int^t_0 K_H (t, s) \mathd \tilde{B}_s$.
  Since we have
  \begin{align*}
    \mathbb{E}_{\mathbb{Q}} \left[ \exp \left( \frac{1}{2} \int_0^T | (K^{-
    1}_H h) (s) |^2 \mathd s \right) \right] & =\mathbb{E}_{\mathbb{P}} \left[
    \exp \left( \frac{1}{2} \int_0^T | (K^{- 1}_H h) (s) |^2 \mathd s \right) 
    \frac{\mathd \mathbb{P}}{\mathd \mathbb{Q}} \right]\\
    & \leqslant \mathbb{E}_{\mathbb{P}} \left[ \exp \left( \int_0^T | (K^{-
    1}_H h) (s) |^2 \mathd s \right) \right]^{1 / 2} \mathbb{E}_{\mathbb{P}}
    \left[ \left( \frac{\mathd \mathbb{P}}{\mathd \mathbb{Q}} \right)^2
    \right]^{1 / 2} < \infty
  \end{align*}
  we can conclude, again by applying Novikov, that $\mathd \mathbb{P}/ \mathd
  \mathbb{Q}$ is integrable w.r.t. $\mathbb{Q}$. Reasoning as before it can be
  shown that $\mathd \mathbb{P}/ \mathd \mathbb{Q}$ admits moments of any
  order w.r.t. $\mathbb{Q}$, which gives the conclusion.
\end{proof}

\subsubsection{Path\mbox{-}by\mbox{-}path uniqueness for SDEs driven by
additive fBm}\label{sec4.2.3}

Girsanov's Theorem allows to construct a probabilistically weak solution of
$(\tmop{CP}_{x_0})$, which we define in the following way.

\begin{definition}
  \label{sec4.2.3 defn weak solution} We say that $(\Omega, \mathcal{F}, \{
  \mathcal{F}_t \}_{t \geqslant 0}, \mathbb{Q}, W^H_{\cdot}, X_{\cdot})$ is a
  weak solution of the Cauchy problem
  \begin{equation}
    X_t = x_0 + \int_0^t b (s, X_s) \mathd s + W^H_t \label{sec4.2.3 eq2}
  \end{equation}
  if $(\Omega, \mathcal{F}, \{ \mathcal{F}_t \}_{t \geqslant 0}, \mathbb{Q})$
  is filtered probability space, $W^H$ is an $\mathcal{F}_t$\mbox{-}fBm of
  parameter $H$ under the probability $\mathbb{Q}$ and
  $\mathbb{Q}$\mbox{-}a.s. the following holds: there exists $\gamma > 1 / 2$
  such that $T^{W^H} b \in C^{\gamma}_t \tmop{Lip}_x$, $X \in W^H +
  C^{\gamma}$ and $X$ is a solution of~{\eqref{sec4.2.3 eq2}} in the sense of
  Definition~\ref{sec4.1 defn solution}.
\end{definition}

We have given a non classical notion of weak solution, which is well suited
when dealing with a distributional $b$; depending on the context, this is not
the only possible definition, see for instance~{\cite{athreya}}
and~{\cite{flandolirusso}} for different choices.

We are now ready to provide a general principle to establish
path\mbox{-}by\mbox{-}path uniqueness.

\begin{lemma}
  \label{sec4.2.3 criterion path-by-path uniqueness}Let $W^H$ be an
  $\mathcal{F}_t$-fBm of parameter $H$ on $(\Omega, \mathcal{F}, \{
  \mathcal{F}_t \}_{t \geqslant 0}, \mathbb{P})$, $x_0 \in \mathbb{R}^d$;
  suppose that:
  \begin{enumeratenumeric}
    \item $b$ is a distributional drift such that, for some $\gamma > 1 / 2$,
    $T^{W^H} b \in C^{\gamma}_t \tmop{Lip}_x$ $\mathbb{P}$\mbox{-}a.s.;
    
    \item Girsanov theorem is applicable to the process $W^H - h$, $h_{\cdot}
    = \int_0^{\cdot} b (s, x_0 + W^H) = T^{W^H} (\cdot, x_0)$.
  \end{enumeratenumeric}
  Then path\mbox{-}by\mbox{-}path uniqueness for $(\tmop{CP}_{x_0})$ holds.
\end{lemma}

\begin{proof}
  Consider $\gamma > 1 / 2$ as in the assumption and the set
  $\mathcal{A}_{\gamma}$ defined as in~{\eqref{sec4.2.1 defn set}}; by
  Theorem~\ref{sec4.1 comparison thm v2}, in order to conclude it is enough to
  show that $\mu^H (\mathcal{A}_{\gamma}) = 1$. By hypothesis, the first half
  of the statement defining $\mathcal{A}_{\gamma}$ is already satisfied on a
  set of full probability, so we only need to concentrate on the second half.
  By the definition of $h$, the process $X = x_0 + W^H$ satisfies
  \begin{equation}
    X_t = x_0 + \int_0^t T^{W^H} (\mathd s, x_0) + [W^H_t - h_t] \backassign
    x_0 + \int_0^t T^X (\mathd s, 0) + \tilde{W}^H_t ; \label{sec4.2.3 eq6}
  \end{equation}
  by hypothesis Girsanov theorem is applicable, so we can construct a new
  probability measure $\mathbb{Q}$ which is absolutely continuous w.r.t. to
  $\mathbb{P}$ such that $\tilde{W}^H$ is an $\mathcal{F}_t$\mbox{-}fBm under
  $\mathbb{Q}$. Observe that $\mathbb{P}$\mbox{-}a.s. $T^X b = \tau^{x_0}
  T^{W^H} b \in C^{\gamma}_t \tmop{Lip}_x$ and so $\mathbb{P}$-a.s. the
  difference $X_{\cdot} - \tilde{W}^H_{\cdot} = x_0 + T^X \left( \cdot \,, 0
  \right) \in C^{\gamma}$ (if $T^X b \in C^{\gamma}_t \tmop{Lip}_x$, then it
  also belongs to $C^0_x C^{\gamma}_t$); then by Lemma~\ref{sec4.1 technical
  lemma 2}, on a set of full measure $\mathbb{P}$ equation~{\eqref{sec4.2.3
  eq6}} is equivalent to
  \[ X_t = x_0 + \int_0^t T^W (\mathd s, X_s - \tilde{W}_s^H) + \tilde{W}^H_t
  \]
  and so $X$ is $\mathbb{P}$\mbox{-}a.s. a solution to $(\tmop{CP}_{x_0})$ in
  the sense of Definition~\ref{sec4.1 defn solution}. Since $\mathbb{Q} \ll
  \mathbb{P}$, all the above statements also hold on a set of
  $\mathbb{Q}$\mbox{-}full measure. But then since $\tilde{W}^H$ has law
  $\mu^H$ under $\mathbb{Q}$, we obtain
  \begin{align*}
    \mu^H (\mathcal{A}_{\gamma}) & = \mathbb{Q} \left( T^{\tilde{W}^H} b \in
    C^{\gamma}_t \tmop{Lip}_x, \text{ $\exists$ a solution } x \in
    \tilde{W}^H + C^{\gamma} \text{ such that } T^x b \in C^{\gamma}_t
    \tmop{Lip}_x \right)\\
    & \geqslant \mathbb{Q} \left( T^{\tilde{W}^H} b \in C^{\gamma}_t
    \tmop{Lip}_x \text{ and $X$ is a solution to $(\tmop{CP}_{x_0})$ satisfying } T^X b \in C^{\gamma}_t \tmop{Lip}_x \right)\\
    & = 1
  \end{align*}
  which gives the conclusion.
\end{proof}

\begin{remark}
  We cannot apply directly the Yamada-Watanabe theorem to deduce existence of
  a strong solution under the assumptions of Lemma~\ref{sec4.2.3 criterion
  path-by-path uniqueness}, because our path\mbox{-}by\mbox{-}path uniqueness
  statement holds only in the class $w + C^{\gamma}$ and not in the class of
  all possible continuous paths (although in the case of continuous $b$ the
  two classes coincide). There is however a more direct way to show that the
  path\mbox{-}by\mbox{-}path unique solution $X$ is adapted to the filtration
  generated by $W^H$. Consider a sequence $\varepsilon_n \, \downarrow 0$ and
  $b^n \assign \rho^{\varepsilon_n} \ast b$, where as usual $\{
  \rho^{\varepsilon} \}_{\varepsilon > 0}$ is a sequence of spatial
  mollifiers, and consider $X^n$ solution to
  \[ \mathd X^n_t = b^n (t, X^n_t) \mathd t + \mathd W^H ; \]
  by classical theory $X^n$ is unique and adapted to the filtration generated
  by $W^H$. Then by Theorem~\ref{sec4.1 comparison thm v2} (possibly combined
  with Lemma~\ref{sec3.1 lemma 3}), $\mathbb{P}$\mbox{-}a.s. $X^n_{\cdot}
  \rightarrow X_{\cdot}$ in $C^{\gamma}$, which implies that $X$ is adapted as
  well and thus a strong solution.
\end{remark}

All the results obtained so far are of abstract nature. Now we are going to
show how to apply them to establish path\mbox{-}by\mbox{-}path uniqueness for
$(\tmop{CP}_{x_0})$ in our context. In particular, Theorem~\ref{sec4.2 thm
main 2} is a direct consequence of the following more general result.

\begin{theorem}
  \label{sec4.2.3 thm path-by-path uniqueness}Let $b$ be a given drift, $H \in
  (0, 1)$. Assume one of the following:
  \begin{itemize}
    \item if $H > 1 / 2$, then there exist $\alpha > 1 - 1 / (2 H) > 0$ and
    $\beta > H - 1 / 2 > 0$ such that $b \in C^0_t C^{\alpha}_x$ and
    \[ | b (t, x) - b (s, y) | \leqslant C (| x - y |^{\alpha} + | t - s
       |^{\beta}) \quad \text{for all } \, s, \, t \in [0, T], \, \, x, y \in
       \mathbb{R}^d ; \]
    \item if $H \leqslant 1 / 2$, then $b \in L^{\infty}_t C^{\alpha}_x$ for
    $\alpha > 1 - 1 / (2 H)$, such that $b$ has compact support, uniformly in
    time; here $\alpha < 0$ is allowed. 
  \end{itemize}
  Then for any $x_0 \in \mathbb{R}^d$ path\mbox{-}by\mbox{-}path uniqueness
  holds for $(\tmop{CP}_{x_0})$.
\end{theorem}

\begin{proof}
  In both cases, in order to conclude, we need to show that we can apply
  Lemma~\ref{sec4.2.3 criterion path-by-path uniqueness} to the process
  $h_{\cdot} = \int_0^{\cdot} b (s, x_0 + W^H_s) \mathd s$; in order to do so,
  we will check that the conditions of Theorem~\ref{sec4.2.2 theorem girsanov
  transform} are satisfied. Up to shifting $b$, we can assume without loss of
  generality $x_0 = 0$.
  
  Let $H > 1 / 2$, then by the hypothesis $b \in C^0_t C^{\alpha}_x$ and
  Theorem~\ref{sec3.3 thm averaging besov} we know that $T^{W^H} \in
  C^{\gamma}_t C^{1 + \varepsilon}_x$ (at least locally) for some $\gamma > 1
  / 2$ and $\varepsilon > 0$; the process $h$ belongs to $C_t^{H + 1 / 2 +
  \varepsilon}$ if and only if the map $t \mapsto b (t, W^H_t) \in C_t^{H - 1
  / 2 + \varepsilon}$. Recall that for any $\gamma < H$, $W^H \in
  C_t^{\gamma}$; then by the hypothesis it holds
  \[ | b (t, W^H_t) - b (s, W^H_s) | \leqslant C (| t - s |^{\beta} + | W^H_t
     - W^H_s |^{\alpha}) \leqslant C (| t - s |^{\beta} + \llbracket W^H
     \rrbracket^{\alpha}_{\gamma} | t - s |^{\alpha \gamma}) \]
  and so we can find $\varepsilon > 0$ small enough such that $\gamma = H -
  \varepsilon$ and
  \[ \left\llbracket b \left( \cdot \,, W^H_{\cdot} \right)
     \right\rrbracket_{C^{H + 1 / 2 + \varepsilon}} \lesssim 1 + \llbracket
     W^H \rrbracket^{\alpha}_{C^{H - \varepsilon}} . \]
  As the exponent $\alpha < 1$, by Fernique Theorem we deduce that
  \[ \mathbb{E} [\exp (\lambda \| h \|_{C^{H + 1 / 2 - \varepsilon}}^2)]
     \lesssim \mathbb{E} [\exp (\lambda C \llbracket W^H \rrbracket_{C^{H -
     \varepsilon}}^{2 \alpha})] < \infty \quad \forall \, \lambda \in
     \mathbb{R}. \]
  Consider now the case $H < 1 / 2$. By Theorem~\ref{sec3.3 thm averaging
  besov} (as the support of $b$ is compact uniformly in time, we have the
  embedding $C^{\alpha} \hookrightarrow B^{\alpha}_{p, p}$ for any $p <
  \infty$) we know that
  \[ T^{W^H} b \in C^{\gamma}_t C^{\alpha + 1 / 2 H - \varepsilon}_x
     \hookrightarrow C^{\gamma}_t C^{1 + \varepsilon}_x \]
  for some $\gamma > 0$ and $\varepsilon > 0$ sufficiently small, therefore
  the process $h_t = T^{W^H} b (t, 0)$ is a well defined element of
  $C^{\gamma}_t$. We now want to show that it actually belongs to $C_t^{H + 1
  / 2 + \varepsilon}$; \quad we can do so by interpolation, using the fact
  that $T^{W^H}$ has higher spatial regularity. Indeed by properties of the
  averaging operator $T^{W^H} b \in \tmop{Lip}_t C^{\alpha}_x$ and so for any
  $\theta \in (0, 1)$ it holds
  \[ \| h_. \|_{C^{1 - \theta / 2}} \leqslant \| T^{W^H} b \|_{C^{1 - \theta /
     2}_t C^{\beta}_x} \leqslant \| T^{W^H} b \|^{1 - \theta}_{\tmop{Lip}_t
     C^{\alpha}_x}  \| T^{W^H} b \|^{\theta}_{C^{1 / 2}_t C^{\alpha + 1 / (2
     H) - \varepsilon}_x} \]
  where $\beta = (1 - \theta) \alpha + \theta (\alpha + (2 H)^{- 1} -
  \varepsilon)$ and thanks to the hypothesis we can choose $\theta \in (0, 1)$
  s.t.
  \[ \left\{\begin{array}{l}
       \beta = \alpha + \frac{\theta}{2 H} - \varepsilon \theta > 0\\
       1 - \frac{\theta}{2} > H + \frac{1}{2}
     \end{array}\right. \Longleftrightarrow \, \alpha - \varepsilon \theta > -
     \frac{\theta}{2 H} > 1 - \frac{1}{2 H} . \]
  For this choice of $\theta$ therefore we obtain
  \[ \| h \|_{C^{H + 1 / 2 +}}^2 \lesssim \| T^{W^H} b \|^{2 \theta}_{C^{1 /
     2}_t C^{\alpha + 1 / (2 H) - \varepsilon}_x} \]
  and since the exponent $2 \theta < 2$, and we have exponential integrability
  for the term on the r.h.s. by Theorem~\ref{sec3.3 thm averaging besov}, we
  get the conclusion.
\end{proof}

In the regime $H > 1 / 2$, the hypothesis required on $b$ is the same as
in~{\cite{nualartouknine}}, although therein path-wise uniqueness is shown
only in the case $d = 1$, while here we obtain path\mbox{-}by\mbox{-}path
uniqueness in any dimension. In the case $H = 1 / 2$, we can allow $b \in
L^{\infty}_t C^{\alpha}_x$ for any $\alpha > 0$; this result is comparable to
the one from~{\cite{davie}}, in which sharper estimates allow to reach $b \in
L^{\infty}_{t, x}$, see also~{\cite{shaposhnikov1,shaposhnikov2}} for further
extensions. Observe that in the regime $H < 1 / 2$ we can allow $b$ to be only
distributional; in this case, we recover the results
from~{\cite{catelliergubinelli}}. Unfortunately, the original proof
from~{\cite{catelliergubinelli}} is wrong, due to an incorrect version of the
formula defining $K_H^{- 1}$ (see the formula for $H^n$ just before Lemma~4.8
therein), which is why we have decided to give an alternative proof rather
than directly invoking the results from~{\cite{catelliergubinelli}}.

\

The driving principle given by Lemma~\ref{sec4.2.3 criterion path-by-path
uniqueness} is fairly general and can be applied under different hypothesis on
$b$, especially when we combine it with Theorems~\ref{sec3.3 thm averaging
bessel} and~\ref{sec3.3 thm averaging besov}.

\begin{corollary}
  \label{sec4.2.3 corollary path uniqueness}Let $H < 1 / 2$ and $b \in L^q_t
  B^{\alpha}_{p, p}$ with $(q, p) \in [2, \infty)^2$, $\alpha < 0$ such that
  \begin{equation}
    \frac{1}{q} + H \left( \frac{d}{p} - \alpha \right) < \frac{1}{2} - H.
    \label{sec4.2.3 generalized condition}
  \end{equation}
  Then for any $x_0 \in \mathbb{R}^d$, path\mbox{-}by\mbox{-}path uniqueness
  for $(\tmop{CP}_{x_0})$ under $\mu^H$ holds. A similar statement holds for
  $b \in L^q_t L^p_x$ with
  \begin{equation}
    \frac{1}{q} + H \frac{d}{p} < \frac{1}{2} - H. \label{sec4.2.3 qp
    condition}
  \end{equation}
\end{corollary}

\begin{proof}
  It follows from hypothesis~{\eqref{sec4.2.3 generalized condition}},
  combined with Theorem~\ref{sec3.3 thm averaging besov} and the Besov
  embeddings $B^{\alpha + s}_{p, p} \hookrightarrow C^{\alpha + s - d / p}$,
  that we can choose $s$ satisfying~{\eqref{sec3.3 KR condition 2}} such that
  $T^{W^H} b \in C^{\gamma}_t C^1_x$ for some $\gamma > 1 / 2$. As before, we
  can now assume $x_0 = 0$ and it remains to show that the process $h_t =
  T^{W^H} b (t, 0) \in C^{\beta}_t$ for some $\beta > H + 1 / 2$ and satisfies
  integrability conditions like those of Theorem~\ref{sec4.2.2 theorem
  girsanov transform}. By the properties of the averaging operator, on a set
  of full probability it holds
  \[ T^{W^H} b \in C^{1 - 1 / q}_t B^{\alpha}_{p, p} \cap C^{1 / 2}_t
     B^{\alpha + s}_{p, p} \]
  for any $s$ such that $H s + 1 / q < 1 / q$. Therefore by interpolation, for
  any $\theta \in (0, 1)$, it holds
  \[ T^{W^H} b \in C^{(1 - 1 / q) (1 - \theta) + \theta / 2}_t B^{(1 - \theta
     \alpha) + \theta (\alpha + s)}_{p, p} \hookrightarrow C^{(1 - 1 / q) (1 -
     \theta) + \theta / 2}_t C^{(1 - \theta) \alpha + \theta (\alpha + s) - d
     / p} . \]
  In order to deduce that $T^{W^H} b \left( \cdot \,, 0 \right) \in
  C^{\beta}_t$ with $\beta > H + 1 / 2$, we need to find parameters $s > 0$
  and $\theta \in (0, 1)$ such that
  \[ \left\{\begin{array}{l}
       H s + \frac{1}{q} < \frac{1}{2},\\
       \left( 1 - \frac{1}{q} \right) (1 - \theta) + \frac{\theta}{2} > H +
       \frac{1}{2},\\
       (1 - \theta) \alpha + \theta (\alpha + s) - \frac{d}{p} > 0.
     \end{array}\right. \]
  A few algebraic manipulations show that the above system is equivalent to
  condition~{\eqref{sec4.2.3 generalized condition}}; from interpolation we
  then obtain, for $\beta = (1 - 1 / q) (1 - \theta) + \theta / 2$ as above,
  \[ \| h \|_{C^{\beta}} \lesssim \| T^{W^H} b \|_{C^{1 / 2}_t B^{\alpha +
     s}_{p, p}}^{\theta} \]
  and since the parameter $\theta \in (0, 1)$, we deduce that $\| h
  \|_{C^{\beta}}$ satisfies~{\eqref{sec4.2.2 hypothesis application
  girsanov}}.
  
  In the case $b \in L^q_t L^p_x$, using the embedding $L^p_x \hookrightarrow
  B^{- \varepsilon}_{p, p}$ for any $\varepsilon > 0$ (see
  Appendix~\ref{appendixA2}) and applying the previous result for
  $\varepsilon$ sufficiently small we get the conclusion.
\end{proof}

In the case $b \in L^q_t L^p_x$, it was already shown in~{\cite{le}} that
pathwise uniqueness holds. Here we have strengthened the result to
path\mbox{-}by\mbox{-}path uniqueness. The case $b \in L^q_t B^{\alpha}_{p,
p}$ with $\alpha < 0$ to the best of our knowledge has not been considered in
the literature so far. Condition~{\eqref{sec4.2.3 generalized condition}}
actually holds also in the regime $\alpha > 0$, but this is not particularly
interesting as one can use fractional Sobolev embeddings
(see~{\cite{dinezza}}) to deduce $L^q_t B^{\alpha}_{p, p} \hookrightarrow
L^q_t L^{p^{\ast}}_x$ with
\[ \frac{1}{p^{\ast}} = \frac{1}{p} - \frac{\alpha}{d} \]
and then reduce it to the case~{\eqref{sec4.2.3 qp condition}}.

\begin{remark}
  The guiding principle of Lemma~\ref{sec4.2.3 criterion path-by-path
  uniqueness} is fairly general, but there are situations in which it is
  possible to establish path\mbox{-}by\mbox{-}path uniqueness even if Girsanov
  theorem is not applicable (or at least we are currently not able to find
  suitable estimates in order to apply it). Consider for instance the case of
  $H > 1 / 2$ and $b \in L^{\infty}_t C^{\alpha}_x$ for $\alpha \in (0, 1)$
  such that
  \[ \alpha > \frac{3}{2} - \frac{1}{2 H} ; \]
  observe that the condition is non trivial for every $H \in (1 / 2)$. Then by
  Theorem~\ref{sec3.3 thm averaging besov} (possibly after a localisation
  procedure) $T^{W^H} b \in C^{\gamma}_t \tmop{Lip}_x$ (at least locally) and
  so by Theorem~\ref{sec4.1 comparison thm v1} path\mbox{-}by\mbox{-}path
  uniqueness holds for the whole ODE. However, lack of continuity in time of
  $b$ prevents us from applying Girsanov.
\end{remark}

\subsection{Regularity of the flow}\label{sec4.3}

\subsubsection{Variational formula for flow of
diffeomorphisms}\label{sec4.3.1}\label{sec4.3.1}

It follows from Theorem~\ref{sec4.1 comparison thm v1} that, if $b$ and $T^w
b$ satisfy the regularity assumptions, the solution map $(\theta_0, t) \mapsto
\theta_t$ is Lipschitz in space, uniformly in time (more precisely, it follows
from~{\eqref{sec4.1 comparison v1.1}} and~{\eqref{sec4.1 comparison v1.2}}
that it is $C^{\gamma}_t \tmop{Lip}_{\tmop{loc}}$). However we cannot yet talk
about a flow, as we haven't shown the invertibility of the solution map, nor
the flow property; this is accomplished by the following two lemmas.

\begin{lemma}
  \label{sec4.3 lemma reversed YDE}Let $T^w b \in C^{\gamma}_t C^{\beta}_x$
  and $\theta \in C^{\alpha}_t$ such that $\gamma + \beta \alpha > 1$. Then
  setting $\tilde{w}_t = w_{T - t}$, $\tilde{b}_t = b_{T - t}$, it holds
  \begin{equation}
    \int_0^t T^w b (\mathd s, \theta_s) = - \int_{T - t}^T T^{\tilde{w}}
    \tilde{b} (\mathd s, \theta_{T - s}) . \label{sec4.3 eq lemma reversed}
  \end{equation}
  In particular, if $\theta$ is a solution of the YDE
  \[ \theta_t = \theta_0 + \int_0^t T^w b (\mathd s, \theta_s), \]
  then $\tilde{\theta}_t = \theta_{T - t}$ satisfies the time-reversed YDE
  \[ \tilde{\theta}_t = \tilde{\theta}_0 + \int_0^t T^{\tilde{w}} \tilde{b}
     (\mathd s, \tilde{\theta}_s) . \]
\end{lemma}

\begin{proof}
  Let $\Pi$ be a partition of $[0, t]$ given by $0 = t_0 < t_1 < \cdots < t_n
  = t$ and define $\tilde{t}_i = T - t_i$, which defines a partition of $[T -
  t, T]$ (up to the fact that it is decreasing w.r.t. $i$); it holds
  \begin{align*}
    \sum_i T^w b_{t_i, t_{i + 1}}^{} (\theta_{t_i}) & = \, \sum_i T^w b_{T -
    \tilde{t}_i, T - \tilde{t}_{i + 1}} (\theta_{T - \tilde{t}_i})\\
    & = \, - \sum_i T^w b_{T - \tilde{t}_{i + 1}, T - \tilde{t}_i} (\theta_{T
    - \tilde{t}_{i + 1}}) + J
  \end{align*}
  where the remainder term $J$ satisfies
  \begin{align*}
    | J | & = \, \left| \sum_i [T^w b_{T - \tilde{t}_{i + 1}, T - \tilde{t}_i}
    (\theta_{T - \tilde{t}_i + 1}) - T^w b_{T - \tilde{t}_{i + 1}, T -
    \tilde{t}_i} (\theta_{T - \tilde{t}_i})] \right|\\
    & \leqslant \, \| T^w b \|_{C^{\gamma}_t C^{\beta}_x}  \| \theta
    \|_{C_t^{\alpha}} \sum_i | t_{i + 1} - t_i |^{\alpha + \beta \gamma}
    \lesssim \| \Pi \|^{\alpha + \beta \gamma - 1} = o (\| \Pi \|) .
  \end{align*}
  By basic properties of the averaging operator we have $T^w b_{T - t, T - s}
  (x) = T^{\tilde{w}} \tilde{b}_{s, t} (x)$ and so overall we obtain
  \[ \sum_i T^w b_{t_i, t_{i + 1}}^{} (\theta_{t_i}) = - \sum_i T^{\tilde{w}}
     \tilde{b}_{\tilde{t}_i, \tilde{t}_{i + 1}} (\theta_{T - \tilde{t}_{i +
     1}}) + o (\| \Pi \|) \]
  Taking a sequence of partitions $\Pi_N$ such that $\| \Pi_N \| \rightarrow
  0$ and taking the limits on both sides we obtain the first statement.
  Regarding the second statement, if $\theta$ is a solution of the YDE, then
  by~{\eqref{sec4.3 eq lemma reversed}} for any $t \in [0, T]$ it holds
  \begin{align*}
    \theta_{T - t} - \theta_T & = - \int_{T - t}^T T^w b (\mathd s, \theta_s)
    = \int_0^t T^{\tilde{w}} \tilde{b} (\mathd s, \theta_{t - s})
  \end{align*}
  which implies the conclusion.
\end{proof}

Similar arguments also provide the following lemma, whose proof is therefore
omitted.

\begin{lemma}
  \label{sec4.3 lemma translated YDE}Let $T^w b \in C^{\gamma}_t C^{\nu}_x$
  with $\gamma (1 + \nu) > 1$ and let $\theta$ be a solution of
  \[ \theta_t = \theta_s + \int_s^t T^w b (\mathd r, \theta_r) \quad \forall
     \, t \in [s, T] . \]
  Then setting $\tilde{\theta}_t = \theta_{s + t}$, $\widetilde{w_t} = w_{s +
  t}$ and $\tilde{b}_t = b_{s + t}$, it holds
  \[ \tilde{\theta}_t = \tilde{\theta}_0 + \int_0^t T^{\tilde{w}} \tilde{b}
     (\mathd r, \tilde{\theta}_r) \quad \forall \, t \in [0, T - s] . \]
\end{lemma}

We are now ready to provide sufficient conditions for the existence of a
Lipschitz flow.

\begin{theorem}
  \label{sec4.3 theorem existence flow}Let $b$, $T^w b$ satisfy the
  assumptions of Theorem~\ref{sec4.1 comparison thm v1}. Then the YDE admits a
  locally Lipschitz flow; namely, setting $\Delta_T \assign \{ (s, t) \in [0,
  T]^2 : s \leqslant t \}$, there exists a map $\Phi : \Delta_T \times
  \mathbb{R}^d \rightarrow \mathbb{R}^d$ with the following properties:
  \begin{enumerateroman}
    \item $\Phi (t, t, x) = x$ for all $t \in [0, T]$ and $x \in
    \mathbb{R}^d$;
    
    \item $\Phi (s, \cdot, x) \in C^{\gamma} ([s, T] ; \mathbb{R}^d)$ for all
    $s \in [0, T]$ and $x \in \mathbb{R}^d$;
    
    \item for all $(s, t, x) \in \Delta_T \times \mathbb{R}^d$ it satisfies
    \[ \Phi (s, t, x) = x + \int_s^t T^w b (\mathd r, \Phi (s, r, x)) ; \]
    \item for all $0 \leqslant s \leqslant u \leqslant t \leqslant T$ and all
    $x \in \mathbb{R}^d$, it holds $\Phi (u, t, \Phi (s, u, x)) = \Phi (s, t,
    x)$;
    
    \item there exists $C = C (\gamma, T, \| T^w b \|_{C^{\gamma}_t
    C^2_x})$ (resp. $C = C (\gamma, T, \| T^w b \|_{C^{\gamma}_t C^{3 / 2}_x}
    \vee \| b \|_{L^{\infty}_{t, x}})$) such that
    \[ | \Phi (s, t, x) - \Phi (s, t, y) | \leqslant C | t - s |^{\gamma} | x
       - y | \quad \text{for all } (s, t) \in \Delta_T, \, x, y \in
       \mathbb{R}^d ; \]
    moreover $\Phi (s, t, \cdot)$ as a function from $\mathbb{R}^d$ to itself
    is invertible and the same inequality holds for its inverse, which we
    denote by $\psi (s, t, \cdot) = \Phi (s, t, \cdot)^{- 1}$.
  \end{enumerateroman}
\end{theorem}

\begin{proof}
  The proof is a straightforward application of Theorem~\ref{sec4.1 comparison
  thm v1} and Lemmata~\ref{sec4.3 lemma reversed YDE} and~\ref{sec4.3 lemma
  translated YDE}. In both cases of time reversal and translation we have $\|
  T^{\tilde{w}} \tilde{b} \|_{C^{\gamma} C^2} \leqslant \| T^w b
  \|_{C^{\gamma} C^2}$ (same for $\| \cdot \|_{C^{\gamma} C^{3 / 2}}$ and $\|
  \cdot \|_{L^{\infty}_{t, x}}$) so that uniqueness holds also for the
  reversed/translated YDE, with the same continuity estimates; this provides
  respectively invertibility of the solution map and flow property.
\end{proof}

Further estimates for $\Phi$ are available, since Theorem~\ref{sec4.1
comparison thm v1} actually implies that
\[ \llbracket \Phi (s, \cdot, x) \rrbracket_{C^{\gamma} ([s, T] ;
   \mathbb{R}^d)} \lesssim 1, \quad \llbracket \Phi (s, \cdot, x) - \Phi (s,
   \cdot, y) \rrbracket_{C^{\gamma} ([s, T] ; \mathbb{R}^d)} \lesssim | x - y
   | \]
uniformly in $s \in [0, T]$, $x, y \in \mathbb{R}^d$.

Let us denote by $\Phi_t$ the map $\Phi_t (x) = \Phi (0, t, x)$; from now on
we are only going to consider the map $\Phi_{\cdot} = \Phi_t (x)$, which by an
abuse of notation and language, will be just denoted by $\Phi$ and referred to
as the flow of the YDE. This is just to keep the notation simple and indeed
all the proofs below can be easily adapted to the whole flow $\Phi (s, t, x)$.

We will keep using the incremental notation $\Phi_{s, t} (x) = \Phi_t (x) -
\Phi_s (x)$; it follows from the above estimates that $\Phi \in C^{\gamma}_t
\tmop{Lip}_{\tmop{loc}}$, since
\begin{align*}
  | \Phi_{s, t} (x) - \Phi_{s, t} (y) | & \leqslant | t - s |^{\gamma}
  \llbracket \Phi (s, \cdot, x) - \Phi (s, \cdot, y) \rrbracket_{C^{\gamma}}
  \lesssim | t - s |^{\gamma} | x - y | .
\end{align*}
Similarly, we define $\psi_t (x) = \psi (0, t, x)$, so that $\psi_t =
\Phi_t^{- 1}$ as a map from $\mathbb{R}^d$ to itself.

We now state a specialised version of Theorem~\ref{sec4.1 comparison thm v1}
which is quite useful for practical purposes, as it clearly identifies a way
to approximate the flow associated to $T^w b$, which by the YDE formulation is
well defined when $b$ is only a distribution, by means of more regular flows,
associated to drifts $b^{\varepsilon}$ for which also the ODE interpretation
is meaningful.

\begin{lemma}
  \label{sec4.3 lemma convergence mollified flows}Let $b$, $T^w b$ satisfy the
  hypothesis of Theorem~\ref{sec4.1 comparison thm v1} and let $\{
  \rho^{\varepsilon} \}_{\varepsilon > 0}$ be a family of spatial mollifiers,
  $b^{\varepsilon} = \rho^{\varepsilon} \ast b$. Then $b^{\varepsilon}$
  satisfies the hypothesis of Theorem~\ref{sec4.1 comparison thm v1} for any
  $\varepsilon > 0$; denote by $\Phi^{\varepsilon}$ and $\Phi$ the flows
  associated respectively to $b^{\varepsilon}$ and $b$. Then
  $\Phi^{\varepsilon} \rightarrow \Phi$ uniformly on compact sets; more
  precisely, for any $\tilde{\gamma} < \gamma$ and any fixed $R > 0$ it holds
  \begin{equation}
    \lim_{\varepsilon \rightarrow 0} \sup_{x \in B_R} \| \Phi (\cdot, x) -
    \Phi^{\varepsilon} (\cdot, x) \|_{C^{\tilde{\gamma}}} = 0. \label{sec4.3
    lemma convergence mollified flow}
  \end{equation}
  In the case $b \in L^{\infty}_{t, x}$, the above convergence actually holds
  for any $\tilde{\gamma} < 1$.
\end{lemma}

\begin{proof}
  We only prove the statement in the case $T^w b \in C^{\gamma}_t C^2_x$, the
  other one being almost identical. By the properties of mollifiers it holds
  $T^w b^{\varepsilon} = (T^w b)^{\varepsilon}$, so that $\| T^w
  b^{\varepsilon} \|_{C^{\gamma}_t C^2_x} \leqslant \| T^w b \|_{C^{\gamma}_t
  C^2_x}$ for all $\varepsilon > 0$, thus the hypothesis of
  Theorem~\ref{sec4.1 comparison thm v1} are satisfied uniformly in
  $\varepsilon > 0$. Once we fix $R > 0$, by the a priori estimates from
  Theorem~\ref{sec4.1 thm existence YDE} we have a uniform bound of the form
  \[ \sup_{\varepsilon > 0} \sup_{x \in B_R} \| \Phi^{\varepsilon} (\cdot, x)
     \|_{C^{\gamma}} \leqslant C < \infty ; \]
  in particular we can localise $T^w b$ and $T^w b^{\varepsilon}$ in such a
  way that they all have support contained in a sufficiently big ball (say for
  instance $B_{2 R}$) in such a way that for $x \in B_R$, $\Phi \left( \cdot
  \,, x \right)$ and $\Phi^{\varepsilon} \left( \cdot \,, x \right)$ are not
  affected by it. Now take any $\tilde{\gamma} \in (1 / 2, \gamma)$, then
  by~{\eqref{sec4.1 comparison v1.1}} in order to conclude it is enough to
  show that $T^w b^{\varepsilon} \rightarrow T^w b$ locally in
  $C^{\tilde{\gamma}}_t \tmop{Lip}_x$; but this is an immediate consequence of
  Lemma~\ref{sec3.1 lemma 3}.
\end{proof}

From now on we will adopt the following notation: whenever all the Young
integrals involved are well defined, we write
\[ \int_0^t f_s A (\mathd s \comma \theta_s) \assign \int_0^t f_s \mathd
   \left( \int_0^{\cdot} A (\mathd r, \theta_r) \right), \]
so that in particular, whenever $\varphi$ is regular enough for $T^w \varphi$
to make sense both as a Young integral and a Lebesgue integral, it holds
\[ \int_0^t f_s T^w \varphi (\mathd s \comma \theta_s) = \int_0^t f_s \varphi
   (s, \theta_s + w_s) \mathd s. \]
We are now ready to further improve the regularity of the flow $\Phi$ and
provide a variational equation for $D_x \Phi$, as well as an expression for
its Jacobian. In the case $A = T^w b \in C^{\gamma}_t C^2_x$ a similar result
was proved in~{\cite{hu}}, Section~3.3; our derivation is of different nature
and based on approximating $b$ by more regular $b^{\varepsilon}$, for which
standard ODE theory applies. The case $T^w b \in C^{\gamma}_t C^{3 / 2}_x$
appears to be new.

\begin{theorem}
  \label{sec4.3 thm variational equation}Let $b$, $T^w b$ satisfy the
  hypothesis of Theorem~\ref{sec4.1 comparison thm v1}. Then $\Phi$ associated
  to $b$ is a flow of diffeomorphisms and belongs to $C^{\gamma}_t
  C^1_{\tmop{loc}}$; it satisfies the variational equation
  \begin{equation}
    D_x \Phi_t (x) = I + \int_0^t D_x \Phi_s (x) \circ T^w D_x b (\mathd r,
    \Phi_r (x)) \label{sec4.3 variational equation}
  \end{equation}
  which is meaningful as a YDE; here $\circ$ denotes the matrix-type product
  given by $A \circ B = B A$.
  
  The Jacobian $J \Phi_t (x) = \det (D_x \Phi_t (x))$ satisfies the identity
  \begin{equation}
    J \Phi_t (x) = \exp \left( \int_0^t \tmop{div} T^w b (\mathd s, \Phi_s
    (x)) \right) \label{sec4.3 jacobian equation}
  \end{equation}
  and there exists $C = C (\gamma, T, \| T^w b \|_{C^{\gamma} C^2})
  > 0$ (resp. $C (\gamma, T, \| T^w b \|_{C^{\gamma} C^{3 / 2}} \vee \| b
  \|_{L^{\infty}})$) such that
  \[ C^{- 1} \leqslant J \Phi_t (x) \leqslant C \quad \forall \, (t, x) \in
     [0, T] \times \mathbb{R}^d . \]
\end{theorem}

\begin{proof}
  As before, to avoid repetitions we give a detailed proof only in the case
  $T^w b \in C^{\gamma}_t C^2_x$; we provide in the end the main differences
  of the proof in the case $b \in L^{\infty}_{t, x}$, $T^w b \in C^{\gamma}_t
  C^{3 / 2}_x$.
  
  We divide the proof in several steps, but the main idea is the following: in
  the case of spatially smooth $b$, the result is just a reformulation of the
  standard ODE results; in the general case we can recover the result by
  reasoning by approximation with the help of Lemma~\ref{sec4.3 lemma
  convergence mollified flows}.
  
  {\tmstrong{Step 1:}} \tmtextit{Proof in the case of regular $b$.} Let us
  first assume in addition that $b \in L^q_t C_x^2$ for some $q > 2$; then in
  this case we know that the YDE formulation is equivalent to the ODE one, so
  that the flow $\Phi$ associated to $b$ satisfies
  \[ \Phi_t (x) = x + \int_0^t b (s, \Phi_s (x) + w_s) \mathd s ; \]
  moreover by standard ODE theory we have the variational equation
  \begin{align*}
    D_x \Phi_t (x) & = \, I + \int_0^t D_x \Phi_s (x) \circ D_x b (s, \Phi_s
    (x) + w_s) \mathd s\\
    & = \, I + \int_0^t D_x \Phi_s (x) \circ \frac{\mathd}{\mathd s} \left(
    \int_0^s D_x b (r, \Phi_r (x) + w_r) \mathd r \right)\\
    & = \, I + \int_0^t D_x \Phi_s (x) \circ T^w D_x b (\mathd r, \Phi_r (x))
    .
  \end{align*}
  The term in the last line now makes perfectly sense as a Young integral, as
  the term
  \[ \int_0^{\cdot} T^w D b (\mathd r, \Phi_r (x)) \]
  is a well defined $C_t^{\gamma}$ map for $\gamma > 1 / 2$, since $b \in
  L^q_t C^1_x$, proving the first part of the claim.
  
  {\tmstrong{Step 2:}} \tmtextit{Approximation and characterisation of the
  limit as $\varepsilon \rightarrow 0$ of $D_x \Phi^{\varepsilon}$.} Consider
  a sequence $T^w b^{\varepsilon}$, $\Phi^{\varepsilon}$ defined by spatial
  mollification as in Lemma~\ref{sec4.3 lemma convergence mollified flows}. By
  Step~1, for any $\varepsilon > 0$, $D_x \Phi^{\varepsilon}$ satisfies the
  variational equation, which for fixed $x$ is a linear YDE in the unknown
  $D_x \Phi^{\varepsilon} (\cdot, x)$ with drift $\int_0^{\cdot} T^w D_x
  b^{\varepsilon} (\mathd r, \Phi^{\varepsilon}_r (x))$; thanks to the a
  priori bounds given by Theorem~\ref{sec4.1 thm existence YDE}, which for
  fixed $x$ are uniform in $\varepsilon$, we have the estimate
  \[ \left\| \int_0^{\cdot} T^w D b^{\varepsilon} (\mathd r,
     \Phi^{\varepsilon}_r (x)) \right\|_{C^{\gamma}} \lesssim \| T^w D
     b^{\varepsilon} \|_{C^{\gamma} \tmop{Lip}_x}  \| \Phi^{\varepsilon}
     (\cdot, x) \|_{C^{\gamma}} \lesssim \| T^w b \|_{C^{\gamma} C^2} \]
  which implies by Proposition~\ref{appendixA1 lemma bound linear YDE} in
  Appendix~\ref{appendixA1} that for fixed $x$ we have the uniform estimate
  \[ \sup_{\varepsilon > 0} \| D_x \Phi^{\varepsilon} (\cdot, x)
     \|_{C^{\gamma}} < \infty . \]
  As in the proof of Lemma~\ref{sec4.3 lemma convergence mollified flows}, for
  any $\delta > 0$ we have $T^w D b^{\varepsilon} \rightarrow T^w D b$ locally
  in $C^{\gamma - \delta}_t C^{2 - \delta}_x$, as well as $\Phi^{\varepsilon}
  (\cdot, x) \rightarrow \Phi (\cdot, x)$ in $C^{\gamma - \delta}_t$, thus
  choosing $\delta$ sufficiently small such that $(\gamma - \delta) (2 -
  \delta) > 1$ by the continuity of nonlinear Young integral it holds
  \[ \int_0^{\cdot} T^w D b^{\varepsilon} (\mathd r, \Phi^{\varepsilon}_r (x))
     \rightarrow \int_0^{\cdot} T^w D b (\mathd r, \Phi_r (x)) \quad \text{in
     } C^{\gamma - \delta}_t . \]
  By the a priori estimates on $D_x \Phi^{\varepsilon} (\cdot, x)$, we can
  extract a subsequence converging to a limit in $C^{\beta}$ for any $1 / 2 <
  \beta < \gamma$; let us denote this limit by $g (\cdot, x)$ (the notation
  will be clear in a second). By Step~1, $D_x \Phi^{\varepsilon}$ satisfy
  variational equations with drifts $\int_0^{\cdot} T^w D_x b^{\varepsilon}
  (\mathd r, \Phi^{\varepsilon} (r, x)) \rightarrow \int_0^{\cdot} T^w D_x b
  (\mathd r, \Phi (r, x))$, which implies that $g (\cdot, x)$ must satisfy the
  linear YDE
  \[ g (t, x) = I + \int_0^t g (s, x) \circ \nobracket T^w D_x b (\mathd r,
     \Phi_r (x) \nobracket) . \]
  But the solution to this linear equation unique, thus so is the limit of any
  subsequence we can extract, showing that the whole sequence $\{ D_x
  \Phi^{\varepsilon} (\cdot, x) \}_{\varepsilon > 0}$ converges to such $g
  (\cdot, x)$. The reasoning holds for any $x \in \mathbb{R}^d$.
  
  {\tmstrong{Step 3:}} \tmtextit{Continuity of the map $(t, x) \mapsto g (t,
  x)$.} This step is very similar to the previous one, so we only sketch it.
  Continuity in $t$ is clear, we only need to prove continuity in $x$; by the
  continuity of the flow, for any sequence $x_n \rightarrow x$ we have $\Phi
  (\cdot, x_n) \rightarrow \Phi (\cdot, x)$ in $C^{\gamma - \delta}$ for any
  $\delta > 0$ and since all $x_n$ lie in a bounded ball, we have uniform
  estimate both on $\int_0^{\cdot} T^w D b (\mathd r, \Phi_r (x_n))$ and $g
  (\cdot, x_n)$. Therefore by the usual compactness argument we deduce that $g
  (\cdot, x_n)$ converge in $C^{\gamma - \delta}$ to the unique solution of
  the YDE associated to $\int_0^{\cdot} T^w D b (\mathd r, \Phi_r (x))$,
  namely $g (\cdot, x)$.
  
  {\tmstrong{Step 4:}} \tmtextit{Flow of diffeomorphisms.} We know that for
  any $R > 0$, the flows $\Phi^{\varepsilon}$ are spatially Lipschitz in
  $B_R$, uniformly in $[0, T]$ and $\varepsilon > 0$, and that they converge
  uniformly on compact sets to $\Phi$, while their spatial derivatives $D_x
  \Phi^{\varepsilon}$ converge to the continuous function $g (t, x)$.
  Therefore we deduce that $g (t, x) = D_x \Phi_t (x)$, thus showing that
  $\Phi$ is $C^1$ in space, uniformly in time; moreover by construction $g$ is
  the unique solution to the variational equation~{\eqref{sec4.3 variational
  equation}}. The reasoning applies to $\psi = \Phi^{- 1}$ as well, as it can
  be represented through the flow associated to the time reversed drift
  $T^{\tilde{w}} \tilde{b}$, which enjoys the same regularity as $T^w b$.
  
  {\tmstrong{Step 5:}} \tmtextit{Jacobian.} As before, let us first assume $b$
  spatially smooth, then by standard ODE theory it holds
  \[ J \Phi_t (x) = \exp \left( \int_0^t \tmop{div} b (s, \Phi_s (x) + w_s)
     \mathd s \right) = \exp \left( \int_0^t \tmop{div} T^w b (\mathd s,
     \Phi_s (x)) \right) \]
  which gives equation~{\eqref{sec4.3 jacobian equation}} in this case. The
  general case is accomplished as above by an approximation procedure, using
  the continuity of Young integrals. Regarding the bound on $J \Phi$, by
  Point~4 of Theorem~\ref{sec4.1 thm definition young integral} combined with
  the a priori estimates on $\Phi$, we obtain
  \[ J \Phi_t (x) \lesssim \| \tmop{div} T^w b \|_{C^{\gamma} \tmop{Lip}}
     \left( 1 + \left\llbracket \Phi \left( \cdot \,, x \right)
     \right\rrbracket_{C^{\gamma}} \right) \leqslant C \]
  which gives the upper bound; the lower bound follows from $(J \Phi_t (x))^{-
  1} = J \psi_t (\Phi_t (x)) \leqslant C$.
  
  {\tmstrong{Step 6:}} \tmtextit{Differences in the case $b \in L^{\infty}_{t,
  x}$ with $T^w b \in C^{\gamma}_t C^{3 / 2}_x$.} The proof in this case goes
  along the exact same lines, with only slightly different regularity
  estimates. Indeed in this case we know that $\Phi (\cdot, x)$ is Lipschitz
  with $\llbracket \Phi (\cdot, x) \rrbracket_{\tmop{Lip}} \leqslant \| b
  \|_{L^{\infty}}$ for all $x \in \mathbb{R}^d$ and so the drift associated to
  the variational equation is controlled by
  \[ \left\| \int_0^{\cdot} T^w D b^{\varepsilon} (\mathd r,
     \Phi^{\varepsilon} (r, x)) \right\|_{C^{\gamma}} \lesssim \| T^w D
     b^{\varepsilon} \|_{C^{\gamma} C^{1 / 2}}  \llbracket \Phi^{\varepsilon}
     (\cdot, x) \rrbracket_{\tmop{Lip}} \leqslant \| T^w b \|_{C^{\gamma} C^{3
     / 2}}  \| b \|_{L^{\infty}} . \]
  Moreover by Lemma~\ref{sec4.3 lemma convergence mollified flows}, we now
  have $\Phi^{\varepsilon} \left( \cdot \,, x \right) \rightarrow \Phi (\cdot,
  x)$ in $C^{1 - \delta}$ for all $\delta > 0$ and so all the reasonings
  related to compactness and continuity of Young integrals still work. A
  similar reasoning goes for equation~{\eqref{sec4.3 jacobian equation}} and
  the two-sided estimates for $J (t, x)$.
\end{proof}

\begin{remark}
  A closer look at the proof shows that the result can be further generalised
  to include the case of $T^w b \in C^{\gamma}_t C^{\nu}_x$ with $b \in L^p_t
  L^{\infty}_x$, under the conditions $\gamma > 1 / 2$ and $\nu \geqslant 1 +
  q / 2$, $q$ being the conjugate of $p$, i.e. $1 / p + 1 / q = 1$.
\end{remark}

\begin{remark}
  \label{sec4.3 remark transport}Recall that in the case of spatially smooth
  $b$, differentiating the relation $\psi_t (\Phi_t (x)) = x$ w.r.t. $t$, one
  obtains that $\psi$ satisfies the PDE
  \begin{equation}
    \partial_t \psi_t (x) + D_x \psi_t (x) b (t, x + w_t) = 0 \quad \text{for
    all } (t, x) \in [0, T] \times \mathbb{R}^d . \label{sec4.3 transport PDE
    psi}
  \end{equation}
  Equation~{\eqref{sec4.3 transport PDE psi}} still holds if $b \in C^0_{t,
  x}$ and $T^w b \in C^{\gamma}_t C^{3 / 2}_x$, since in this case $\Phi$ is
  locally $C^1_{t, x}$ and the same holds for $\psi$.
  
  In the general case $T^w b \in C^{\gamma}_t C^2_x$, reasoning by
  approximation, if $\psi \in C^{\gamma}_t C^1_{\tmop{loc}}$ then the equation
  is still satisfied in the following generalised sense:
  \begin{equation}
    \psi_t (x) - x = \int_0^t D_x \psi_s (x) T^w b (\mathd s, x) \quad
    \text{for all } (t, x) \in [0, T] \times \mathbb{R}^d \label{sec4.3
    transport PDE psi 2}
  \end{equation}
  where the r.h.s. is a Young integral in time, for fixed $x \in
  \mathbb{R}^d$.
  
  However, the regularity requirement $\psi \in C^{\gamma}_t C^1_{\tmop{loc}}$
  does not need to hold; in general the only information available is $\psi
  \in C^{\gamma}_t C^0_{\tmop{loc}} \cap C^0_t C^1_{\tmop{loc}}$. Indeed, by
  the group property
  \[ \psi (0, s, \cdot) \circ \psi (s, t, \cdot) = \psi (0, t, \cdot) \]
  it holds
  \begin{align*}
    | \psi_t (x) - \psi_s (x) | & = | \psi_s (\psi (s, t, x)) - \psi_s (x) |\\
    & \leqslant \llbracket \psi_s \rrbracket_{\tmop{Lip}}  | \psi (s, t, x) -
    x |\\
    & \lesssim | t - s |^{\gamma}
  \end{align*}
  where the estimate is uniform in $x \in \mathbb{R}^d$; establishing $\psi
  \in C^{\gamma}_t C^1_{\tmop{loc}}$ requires an analogue estimate for $D_x
  \psi$, where
  \begin{align*}
    D_x \psi_t (x) - D_x \psi_s (x) & = D_x \psi_s (\psi (s, t, x)) D_x \psi
    (s, t, x) - D_x \psi_s (x) .
  \end{align*}
  It's easy to see from the above expression that if $\psi \in C^0_t
  C^2_{\tmop{loc}}$ (which by time reversal is equivalent to $\Phi \in C^0_t
  C^2_{\tmop{loc}}$), then it belongs to $C^{\gamma}_t C^1_{\tmop{loc}}$ as
  well. As shown in the next section, this condition is met if $T^w b$ is
  regular enough. 
\end{remark}

\subsubsection{Higher regularity}\label{sec4.3.2}

Similarly to the standard ODE case, we can show that the flow $\Phi$ inherits
the spatial regularity of $T^w b$, i.e. to a more regular averaged functional
$T^w b$ corresponds a more regular flow of solutions.

\begin{theorem}
  \label{sec4.3 thm higher regularity}Let $n \in \mathbb{N}, n \geqslant 1$,
  $\gamma > 1 / 2$ and assume that one of the following conditions holds:
  \begin{itemize}
    \item $T^w b \in C^{\gamma}_t C^{n + 1}_x$; or
    
    \item $T^w b \in C^{\gamma}_t C^{n + 1 / 2}_x$ and $b \in L^{\infty}_{t,
    x}$.
  \end{itemize}
  Then the YDE associated to $\theta$ admits a locally $C^n_x$-regular flow
  $(t, x) \mapsto \Phi (t, x)$.
\end{theorem}

\begin{proof}
  As before, we give a detailed proof in the case $T^w b \in C^{\gamma}_t C^{n
  + 1}_x$ and in the end highlight the main differences in the other case. The
  idea of the proof, similarly to that of Theorem~\ref{sec4.3 thm variational
  equation}, is to reason by approximation and establish first that, for
  $b^{\varepsilon} = \rho^{\varepsilon} \ast b$, it holds $\Phi^{\varepsilon}
  \in C^{\gamma}_t C_x^{n + 1}$ with an estimate which is uniform in
  $\varepsilon > 0$; then the conclusion follows from taking the limit as
  $\varepsilon \rightarrow 0$. In order to get uniform estimates, we will show
  that for any $k \leqslant n$, $D^k_x \Phi^{\varepsilon}$ satisfies a
  variational type equation in which the leading term is a linear Young
  integral. We split the proof in several steps.
  
  {\tmstrong{Step 1:}} \tmtextit{$k$\mbox{-}th order variation equation.} We
  start by assuming $b \in L^q_t C^{\infty}_x$ in addition to the assumptions,
  so that by standard ODE theory the associated flow has $C^{\infty}$ spatial
  regularity. We now adopt the following convention: the symbol $\circ$
  denotes a suitably chosen matrix product, which can change from line to
  line. We claim that, for any $1 \leqslant k \leqslant n$, $D^k_x \Phi$
  satisfies the variational-type equation
  \begin{equation}
    D^k_x \Phi (t, x) = \int_0^t D^k_x \Phi (s, x) \circ \nobracket D_x T^w b
    (\mathd r, \Phi (r, x) \nobracket) + F_k (\{ D_x^i \Phi (\cdot, x) \}_{i
    \leqslant k - 1}) \label{sec4.3 recursive variational eq}
  \end{equation}
  where the first integral makes sense in the Young sense and the $F_k$ are
  ``polynomial'' functions of the form
  \[ F_k (\{ D_x^i \Phi (\cdot, x) \}) = \sum_{\alpha = 1}^k \sum_{\beta}
     a_{\beta} \int_0^t \bigotimes^{k - 1}_{\tmscript{\begin{array}{c}
       i = 1
     \end{array}}} (D^i \Phi (s, x))^{\otimes \beta_i} \circ \nobracket
     D^{\alpha} T^w b (\mathd r, \Phi (r, x) \nobracket) \tmcolor{blue}{} \]
  where the internal sum is taken over all possible $\beta = (\beta_1, \ldots,
  \beta_{k - 1})$ with $\beta_i \in \{ 1, \ldots, k \}$ such that $\sum_i i
  \beta_i = k$ and $a_{\beta}$ are suitable coefficients of combinatorial
  nature. Observe that, in terms of the variable $D^k_x \Phi (\cdot, x)$,
  equation~{\eqref{sec4.3 recursive variational eq}} is a linear YDE of the
  form $y_t = \int_0^t A_{\mathd s} y_s + h_t$, as the second term does not
  have any dependency on $D^k_x \Phi$.
  
  The proof is by induction on $k$, the case $k = 1$ being immediate. In the
  case $k = 2$, differentiating both terms in the variational equation
  associated to the drift $(t, x) \mapsto b (t, x + w_t)$
  \[ D_x \Phi (t, x) = I + \int_0^t D_x \Phi (s, x) \circ D_x b (s, \Phi (s,
     x) + w_s) \mathd s, \]
  we obtain
  \begin{align*}
    D_x^2 \Phi (t, x) & = \int_0^t D_x^2 \Phi (s, x) \circ D_x b (s, \Phi (s,
    x) + w_s) \mathd s + \int_0^t D_x \Phi (s, x)^{\otimes 2} \circ D_x^2 b
    (s, \Phi (s, x) + w_s) \mathd s\\
    & = \int_0^t D_x^2 \Phi (s, x) \circ \nobracket D_x T^w b (\mathd r, \Phi
    (r, x) \nobracket) + \int_0^t D_x \Phi (s, x)^{\otimes 2} \circ \nobracket
    D_x^2 T^w b (\mathd r, \Phi (r, x) \nobracket)
  \end{align*}
  which is exactly of the form~{\eqref{sec4.3 recursive variational eq}}. Now
  assume that the statement is true for $k$, then
  differentiating~{\eqref{sec4.3 recursive variational eq}} on both sides we
  obtain
  \begin{align*}
    D_x^{k + 1} \Phi (t, x) & = \int_0^t D_x^{k + 1} \Phi (s, x) \circ
    \nobracket D_x T^w b (\mathd r, \Phi (r, x) \nobracket)\\
    & + \int_0^t (D_x^k \Phi (s, x) \otimes D_x \Phi (s, x)) \circ \nobracket
    D_x^2 T^w b (\mathd r, \Phi (r, x) \nobracket) + \tilde{F}_{k + 1} (\{ D^i
    \Phi (\cdot, x) \}_{i \leqslant k})
  \end{align*}
  where $\tilde{F}_{k + 1} (\{ D^i \Phi (\cdot, x) \}_{i \leqslant k + 1}) =
  D_x F_k (\{ D^i \Phi (\cdot, x) \}_{i \leqslant k})$ and it is easy to check
  that it is still of ``polynomial type''.
  
  {\tmstrong{Step 2:}} \tmtextit{Inductive estimate on $D^k_x \Phi$.} Fix $R >
  0$; we claim that, for any $2 \leqslant k \leqslant n$, there exists a
  constant $C_k < \infty$ (which depends on~$R$), which is independent of $\|
  b \|_{L^q_t C^{\infty}_x}$, such that
  \[ \sup_{i \leqslant k} \sup_{x \in B_R}  \| D^i \Phi (\cdot, x)
     \|_{C^{\gamma}_t} < \infty . \]
  Again the proof is inductive, mainly relying on the fact that $D^i \Phi$
  solves the linear YDE~{\eqref{sec4.3 recursive variational eq}} in
  combination with the a priori bounds given by Proposition~\ref{appendixA1
  lemma bound linear YDE}.
  
  We start by proving the claim in the case $k = 2$. In this case we already
  know by Theorems~\ref{sec4.1 thm existence YDE} and~\ref{sec4.3 thm
  variational equation} that \ $\sup_{x \in B_R}  (\| \Phi (\cdot, x) \| + \|
  D \Phi (\cdot, x) \|_{C^{\gamma}_t}) \leqslant C$; moreover by properties of
  Young integral we have
  \begin{align*}
    \left\| \int_0^{\cdot} D_x \Phi (s, x)^{\otimes 2} \circ \nobracket D_x^2
    T^w b (\mathd r, \Phi (r, x) \nobracket) \right\|_{C^{\gamma}} & \lesssim
    \| D_x \Phi (\cdot, x) \|_{C^{\gamma}}^2  \left\| \int_0^{\cdot} D_x^2 T^w
    b (\mathd r, \Phi (r, x)) \right\|_{C^{\gamma}}\\
    & \lesssim \| D_x \Phi (\cdot, x) \|_{C^{\gamma}}^2  \| \Phi (\cdot, x)
    \|_{C^{\gamma}}  \| D_x^2 T^w b \|_{C^{\gamma}_t C^1_x}\\
    & \lesssim \| T^w b \|_{C^{\gamma}_t C^3_x}
  \end{align*}
  as well as the bound $\left\| \int_0^{\cdot} D_x T^w b (\mathd r, \Phi (r,
  x)) \right\|_{C^{\gamma}} \lesssim \| T^w b \|_{C^{\gamma}_t C^2_x}  \| \Phi
  (\cdot, x) \|_{C^{\gamma}}$. Applying again Proposition~\ref{appendixA1
  bound linear YDE} yields the conclusion in this case.
  
  Assume now that the claim holds for $k$, then by the inductive hypothesis
  all the term appearing in the sum defining $F_{k + 1}$ can be estimated by
  \[ \begin{array}{l}
       \left\| \int_0^{\cdot} \bigotimes_{\tmscript{\begin{array}{c}
         i, \beta_i
       \end{array}}} (D_x^i \Phi (s, x))^{\otimes \beta_i} \circ \nobracket
       D_x^{\alpha} T^w b (\mathd r, \Phi (r, x) \nobracket)
       \right\|_{C^{\gamma}}\\
       \qquad \qquad \qquad \lesssim \prod_{i, \beta_i} C_k^{\beta_i}  \left\|
       \int_0^{\cdot} D_x^{\alpha} T^w b (\mathd r, \Phi (r, x))
       \right\|_{C^{\gamma}}\\
       \qquad \qquad \qquad \lesssim \prod_{i, \beta_i} C_k^{\beta_i}
       C_k^{\nu}  \| D_x^{\alpha} T^w b \|_{C^{\gamma}_t C^{\nu}_x} \lesssim
       \| T^w b \|_{C^{\gamma}_t C^{\nu + n}_x}
     \end{array} \]
  which together with the estimate for $\left\| \int_0^{\cdot} D T^w b (\mathd
  r, \Phi (r, x)) \right\|_{C^{\gamma}}$ and the application of
  Proposition~\ref{appendixA1 lemma bound linear YDE} yields a new constant
  $C_{k + 1}$.
  
  {\tmstrong{Step 3:}} \tmtextit{Approximation procedure.} Let
  $b^{\varepsilon} = \rho^{\varepsilon} \ast b$ denote by $\Phi$ and
  $\Phi^{\varepsilon}$ the flows associated to $b$ and $b^{\varepsilon}$
  respectively. Then $\| T^w b^{\varepsilon} \|_{C^{\gamma}_t C^{n + 1}_x}
  \leqslant \| T^w b \|_{C^{\gamma}_t C^{n + 1}_x}$ for all $\varepsilon > 0$
  and so by the previous step we deduce that for any $R > 0$ there exists a
  suitable constant $C$ such that
  \[ \sup_{\varepsilon > 0}  \| \Phi^{\varepsilon} \|_{L^{\infty} W^{n,
     \infty} (B_R)} \leqslant C. \]
  But $\Phi^{\varepsilon} \rightarrow \Phi$ uniformly in $[0, T] \times B_R$,
  which together with the weak-$\ast$ compactness of balls in $W^{n, \infty}
  (B_R)$ implies that $\Phi \in L^{\infty}_t W^{n, \infty} (B_R)$. A slightly
  more refined argument, analogue to the one from Theorem~\ref{sec4.3 thm
  variational equation}, allows to show that, for any fixed $x \in
  \mathbb{R}^d$, $D^n \Phi^{\varepsilon} (\cdot, x)$ must converge as
  $\varepsilon \rightarrow 0$ to the unique solution of the variational-type
  equation~{\eqref{sec4.3 recursive variational eq}} associated to $\Phi$;
  with this information at hand it is then possible to show that the limit
  varies continuously in $x$ and must coincide with $D^n \Phi (\cdot, x)$,
  thus showing that $\Phi$ is not only in $W^{n, \infty}_{\tmop{loc}}$ but
  also $C^n$. We omit the details in order to avoid unnecessary repetitions.
  
  {\tmstrong{Step 4:}} \tmtextit{The case $b \in L^{\infty}_{t, x}$ with $T^w
  b \in C^{\gamma}_t C^{n + 1 / 2}_x$.} In this case Step~1 and Step~3 are
  identical to the ones above, the only change is in the estimates from
  Step~2, as we can use the information $\left\| \Phi \left( \cdot \,, x
  \right) \right\|_{\tmop{Lip}_t} < \infty$ uniformly in $x \in B_R$ to
  require less regularity for $T^w b$. For instance in the case $k = 2$ we
  have the estimates
  \[ \begin{array}{ll}
       \left\| \int_0^{\cdot} D_x \Phi (s, x)^{\otimes 2} \circ \nobracket
       D_x^2 T^w b (\mathd r, \Phi (r, x) \nobracket) \right\|_{C^{\gamma}} &
       \lesssim \| D_x \Phi (\cdot, x) \|_{C^{\gamma}}^2  \left\|
       \int_0^{\cdot} D_x^2 T^w b (\mathd r, \Phi (r, x))
       \right\|_{C^{\gamma}}\\
       & \lesssim \| D_x \Phi (\cdot, x) \|_{C^{\gamma}}^2  \| \Phi (\cdot,
       x) \|_{\tmop{Lip}}  \| D_x^2 T^w b \|_{C^{\gamma}_t C^{1 / 2}_x}\\
       & \lesssim \| T^w b \|_{C^{\gamma}_t C^{5 / 2}_x}
     \end{array} \]
  and $\left\| \int_0^{\cdot} D_x T^w b (\mathd r, \Phi (r, x))
  \right\|_{C^{\gamma}} \lesssim \| T^w b \|_{C^{\gamma}_t C^{3 / 2}_x}  \|
  \Phi (\cdot, x) \|_{\tmop{Lip}}$. The general inductive step similar.
\end{proof}

\section{Application to transport type PDEs}\label{sec5}

\tmcolor{blue}{}The aim of this section is to apply the theory of
Section~\ref{sec4} in order to solve perturbed first order linear PDEs of the
form
\begin{equation}
  \partial_t u + b \cdot \nabla u + c u + \dot{w} \cdot \nabla u = 0,
  \label{sec5 eq1}
\end{equation}
where $\dot{w}$ denotes the time derivative of $w$; at this stage, the
equation is only formal. However, if we assumed everything smooth, then
applying the change of variables $\tilde{u} (t, x) = u (t, x + w_t)$
(similarly for $\tilde{b}$, $\tilde{c}$), {\eqref{sec5 eq1}} would be
equivalent to
\begin{equation}
  \partial_t \tilde{u} + \tilde{b} \cdot \nabla \tilde{u} + \tilde{c} 
  \tilde{u} = 0. \label{sec5 eq2}
\end{equation}
Equation~{\eqref{sec5 eq2}} is now meaningful in the classical sense if for
instance $\tilde{b}, \tilde{c} \in C^0_{t, x}$, which is equivalent to $b, c
\in C^0_{t, x}$; it also makes sense in the weak sense under suitable
integrability assumptions on $b, c$. Moreover the transformation that defines
$\tilde{u}$ in function of $u$ is well defined whenever $w$ is a continuous
path.

\

Based on the above reasoning, we will adopt the convention that $u$ is a
solution to~{\eqref{sec5 eq1}} if and only if $\tilde{u}$ defined as above is
a solution to~{\eqref{sec5 eq2}} and we will study systematically the latter
equation. Let us mention that in the case $w$ is a rough path, it is possible
to give meaning to~{\eqref{sec5 eq1}} and the passage from~{\eqref{sec5 eq1}}
to~{\eqref{sec5 eq2}} can be rigorously justified, see~{\cite{catellier}}.

\

Although the above discussion holds for general $c$, we will focus only on
two cases of interest, given by transport and continuity equations, namely for
$c = 0$ and $c = \tmop{div} b$ (resp. $\tilde{c} = 0$ and $\tilde{c} =
\tmop{div} \tilde{b}$).

\

While in Section~\ref{sec4} all the proofs were almost identical for $b \in
C^0_{t, x}$ with $T^w b \in C^{\gamma}_t C^{3 / 2}_x$ and $T^w b \in
C^{\gamma}_t C^2_x$, here the difference becomes relevant and the first case
is much easier to treat compared to the latter; to our surprise, even if the
existence of a Lipschitz flow for the associated ODE is already known, the
case $T^w b \in C^{\gamma}_t C^2_x$ requires the application of refined tools
like commutators and the Sewing lemma. For this reason, we split the results
in two subsections, with the proofs becoming gradually more complex, so that
the difficulties arising in the second case become apparent.

\subsection{The case of continuous bounded $b$}\label{sec5.1}

Let us mention that in this case the transport equation has been treated with
similar techniques in~{\cite{catellier}}, while the continuity equation in
Chapter~9 from~{\cite{maurelli}}. More recently, in the case $b \in
L^{\infty}_{t, x}$, the transport equation has been investigated with
different techniques in~{\cite{amine2020}}.

\

We start by considering the case $c \equiv 0$. Recall that $\tilde{b} (t, x)
= b (t, x + w_t)$ and that in this case the YDE associated to $\theta$
corresponds to the ODE associated to $\tilde{b}$, for which existence of a
locally $C^1_{t, x}$ flow $\Phi$ is known. Let us also recall the notation
from Section~\ref{sec4.3}, namely $\Phi_t (x) = \Phi (0, t, x)$, $\psi (s, t,
\cdot) = \Phi (s, t, \cdot)^{- 1}$ and $\psi_t = \Phi_t^{- 1}$. With a slight
abuse, from now on we will denote $\tilde{u}$ with $u$ instead.

\begin{proposition}
  \label{sec5.1 thm transport}Let $b \in C^0_{t, x}$ such that $T^w b \in
  C^{\gamma}_t C^{3 / 2}_x$, then for any $u_0 \in C^1_x$ there exists a
  unique solution of
  \begin{equation}
    \partial_t u + \tilde{b} \cdot \nabla u = 0 \label{sec5.1 transport eq}
  \end{equation}
  with initial condition $u_0$, which is given by $u_t (x) = u_0 (\psi_t
  (x))$.
\end{proposition}

\begin{proof}
  Recall that by Remark~\ref{sec4.3 remark transport} $\psi \in C^1_{t, x}$
  solves equation
  \[ \partial_t \psi (t, x) + D_x \psi (t, x)  \tilde{b} (t, x) = 0. \]
  Therefore $u (t, x) \assign u_0 (\psi (t, x)) \in C^1_{t, x}$ and satisfies
  \[ \partial_t u (t, x) + \nabla u (t, x) \cdot \tilde{b} (t, x) = \nabla u_0
     (\psi (t, x)) \cdot [\partial_t \psi (t, x) + D_x \psi (t, x) \tilde{b}
     (t, x)] = 0 \]
  which shows that it is a solution.
  
  Conversely, let $u$ be a solution and for a given $x \in \mathbb{R}^d$
  define $z_t = u (t, \Phi_t (x))$. $\dot{\Phi}_t (x) = \tilde{b} (t, \Phi_t
  (x))$, therefore $z$ solves
  \[ \dot{z}_t = \partial_t u (t, \Phi_t (x)) + \nabla u (t, \Phi_t (x)) \cdot
     \tilde{b} (t, \Phi_t (x)) = 0 \]
  which implies that $u (t, \Phi_t (x)) = u_0 (x)$ for all $x$ and thus $u (t,
  x) = u_0 (\psi (t, x))$.
\end{proof}

We now turn to the case $c = \tmop{div} b$, i.e. the continuity equation.
Since in general $\tmop{div} b$ is only defined as a distribution, it makes
sense to interpret the equation in a weak sense.

We adopt the following notation: $\mathcal{M}_x (\mathbb{R}^d) =
\mathcal{M}_x$ denotes the Banach space of all finite signed Radon measures on
$\mathbb{R}^d$, endowed with the total variation norm. We say that $v \in
L^{\infty}_t \mathcal{M}_x$ is weakly continuous if the map $t \mapsto v_t$ is
continuous $\mathcal{M}_x$ endowed with the weak-$\ast$ topology, equivalently
if for any $\varphi \in C^{\infty}_c (\mathbb{R}^d)$, the map $t \mapsto
\langle v_t, \varphi \rangle$ is continuous.

\begin{definition}
  \label{sec5.1 defn continuity equation}Let $\tilde{b} \in C^0_{t, x}$, $v
  \in L^{\infty}_t \mathcal{M}_x$. We say that $v$ is a weak solution of the
  continuity equation
  \begin{equation}
    \partial_t v + \nabla \cdot (\tilde{b} v) = 0 \label{sec5.1 continuity
    equation}
  \end{equation}
  if $v$ is weakly continuous and for any $\varphi \in C^{\infty}_c ([0, T]
  \times \mathbb{R}^d)$ it holds
  \begin{equation}
    \langle v_t, \varphi_t \rangle - \langle v_0, \varphi_0 \rangle = \int_0^t
    \langle v_s, \partial_t \varphi_s + \tilde{b}_s \cdot \nabla \varphi_s
    \rangle \mathd s. \label{sec5.1 weak continuity equation}
  \end{equation}
\end{definition}

\begin{proposition}
  \label{sec5.1 thm continuity}Let $b \in C^0_{t, x}$ such that $T^w b \in
  C^{\gamma}_t C^{3 / 2}_x$, then for any $v_0 \in \mathcal{M}_x
  (\mathbb{R}^d)$ there exists a unique weak solution $v$ of~{\eqref{sec5.1
  continuity equation}} with initial data $v_0$, which is given by
  \begin{equation}
    v_t (\mathd x) = \exp \left( - \int_0^t \tmop{div} T^w b (\mathd s, \psi
    (s, t, x)) \right) v_0 (\mathd x) \label{sec5.1 solution continuity ugly}
  \end{equation}
  or equivalently $v_t (\mathd x)$ defined by duality as
  \begin{equation}
    \int_{\mathbb{R}^d} \varphi (x) v_t (\mathd x) = \int_{\mathbb{R}^d}
    \varphi (\Phi_t (x)) v_0 (\mathd x), \qquad \forall \varphi \in
    C^{\infty}_c . \label{sec5.1 solution continuity eq}
  \end{equation}
\end{proposition}

\begin{remark}
  Whenever $\tmop{div} b \in C^0_{t, x}$, equation~{\eqref{sec5.1 solution
  continuity ugly}} corresponds to the classical formulation
  \[ v_t (\mathd x) = \exp \left( - \int_0^t \tmop{div} \tilde{b} (s, \psi (s,
     t, x)) \mathd s \right) v_0 (\mathd x) . \]
  Under the above assumptions equation~{\eqref{sec5.1 solution continuity
  ugly}} is still meaningful as a nonlinear Young integral, since $\tmop{div}
  T^w b = T^w \tmop{div} b \in C^{\gamma}_t C^{1 / 2}_x$ and the map $s
  \mapsto \psi (s, t, x)$ is Lipschitz. However, we will only use
  formula~{\eqref{sec5.1 solution continuity eq}} in the proof, as it is more
  practical for explicit computations. 
\end{remark}

\begin{proof}
  Let $v$ be defined by~{\eqref{sec5.1 solution continuity eq}}, then for any
  $\varphi \in C^{\infty}_c ([0, T] \times \mathbb{R}^d)$ we have
  \begin{eqnarray*}
    \langle v_t, \varphi_t \rangle - \langle v_0, \varphi_0 \rangle & = &
    \int_{\mathbb{R}^d} [\varphi_t (\Phi_t (x)) - \varphi_0 (x)] v_0 (\mathd
    x)\\
    & = & \int_0^t \int_{\mathbb{R}^d} \frac{\mathd}{\mathd s} (\varphi_s
    (\Phi_s (x))) v_0 (\mathd x) \mathd s\\
    & = & \int_0^t \int_{\mathbb{R}^d} [\partial_t \varphi_s (\Phi_s (x)) +
    \nabla \varphi_s (\Phi_s (x)) \cdot b_s (\Phi (s, x))] v_0 (\mathd x)
    \mathd s\\
    & = & \int_0^t \int_{\mathbb{R}^d} [\partial_t \varphi_s (x) + \nabla
    \varphi_s (x) \cdot b_s (x)] v_s (\mathd x) \mathd s,
  \end{eqnarray*}
  which shows that $v$ is a weak solution of~{\eqref{sec5.1 continuity
  equation}}.
  
  Since equation~{\eqref{sec5.1 weak continuity equation}} is linear, it is
  enough to establish uniqueness in the case $v_0 \equiv 0$. Let $v$ be a
  given weak solution, then by standard density arguments~{\eqref{sec5.1 weak
  continuity equation}} extends to all $\varphi \in C^1_c ([0, T] \times
  \mathbb{R}^d)$; take $\varphi_t (x) = u (\psi_t (x))$ with $u \in
  C^{\infty}_c (\mathbb{R}^d)$, so that $\varphi \in C^1_c ([0, T] \times
  \mathbb{R}^d)$ and it solves $\partial_t \varphi + \nabla \varphi \cdot b =
  0$. Then we obtain
  \[ \int u (\psi_t (x)) v_t (\mathd x) = \langle v_t, \varphi_t \rangle =
     \langle v_0, \varphi_0 \rangle = 0 \quad \forall \, u \in C^{\infty}_c .
  \]
  By usual density arguments, the relation then extends to all continuous
  bounded $u$; for fixed $t$, taking $u (x) = \tilde{u} (\Phi_t (x))$, we
  deduce that $\langle \tilde{u}, v_t \rangle = 0$ for all $\tilde{u} \in
  C^0_b$, which implies $v_t \equiv 0$ for all $t$.
\end{proof}

\subsection{The case of distributional $b$}

We now pass to the case $T^w b \in C^{\gamma}_t C^2_x$, without assuming any
regularity on the distribution~$b$. To the best of our knowledge, this case
has never been considered in literature so far; although perturbed linear PDEs
have been previously treated in {\cite{catellier,nilssen2020}}, it is always
assumede therein at least $b \in L^{\infty}_{t, x}$ (which can be treated
analogously to Section~\ref{sec5.1}). However, our approach in the ``Young
regime'', namely for time regularity $\gamma > 1 / 2$, is undoubtedly similar
(and even simpler) to that in the ``rough regime'' $\gamma \in (1 / 3, 1 / 2]$
treated in~{\cite{bailleul}}. The use of a commutator lemma also reflects the
work~{\cite{diperna}} and Chapter~9 from~{\cite{maurelli}}. Abstract transport
equations in H\"{o}lder media have been treated also in~{\cite{hu}}; however
the results there are, in our opinion, not completely clear, see
Remark~\ref{sec5.2 remark limited use} below.

\begin{definition}
  \label{sec5.2 defn transport}Let $T^w b \in C^{\gamma}_t C^1_x$; we say that
  $u \in C^{\gamma}_t C^0_{\tmop{loc}}$ is a solution of the Young transport
  equation
  \begin{equation}
    u (\mathd t, x) + \nabla u (t, x) \cdot T^w b (\mathd t, x) = 0
    \label{sec5.2 transport eq}
  \end{equation}
  if for all $\varphi \in C^{\infty}_c (\mathbb{R}^d)$ and all $t \in [0, T]$,
  the following Young integral equation holds:
  \begin{equation}
    \langle u_t, \varphi \rangle = \langle u_0, \varphi \rangle + \int_0^t
    \langle u_s, \, \tmop{div} \left( T^w_{\mathd t} b \, \varphi \right)
    \rangle . \label{sec5.2 transport eq integral}
  \end{equation}
\end{definition}

\begin{remark}
  \label{sec5.2 remark defn transport}The integral appearing in~{\eqref{sec5.2
  transport eq integral}} is meaningful as a Young integral, since by
  assumptions the map $t \mapsto \tmop{div} (T^w b (t, \cdot) \varphi)$
  belongs to $C^{\gamma}_t C^0_c$ while $t \mapsto u_t \in C^{\gamma}_t
  C^0_{\tmop{loc}}$. An equivalent more pratical formulation of~{\eqref{sec5.2
  transport eq integral}} is the following one: for any $\varphi \in
  C^{\infty}_c (\mathbb{R}^d)$, we have the estimate
  \begin{equation}
    \left| \langle u_{s, t}, \varphi \rangle - \langle u_s, \, \tmop{div}
    (T^w_{s, t} b \varphi) \rangle \right| \lesssim_K | t - s |^{2 \gamma} \|
    \varphi \|_{C^1}  \llbracket u \rrbracket_{C^{\gamma} C^0_K}  \llbracket
    T^w b \rrbracket_{C^{\gamma} C^1} \label{sec5.2 transport eq3}
  \end{equation}
  which is uniform over $(s, t) \in \Delta_T$ but depends on $K = \tmop{supp}
  \varphi$; choosing $\varphi = \rho^{\varepsilon} (x - \cdot)$ with $x \in
  B_R$ and $\rho^{\varepsilon}$ standard mollifier, since $T^w_{s, t} b \cdot
  \nabla u_s$ is a well defined distribution, we obtain
  \begin{equation}
    \sup_{x \in B_R} | \rho^{\varepsilon} \ast u_{s, t} - \rho^{\varepsilon}
    \ast (T^w_{s, t} b \cdot \nabla u_s) | \lesssim_{\varepsilon, R} | t - s
    |^{2 \gamma} . \label{sec5.2 transport eq5}
  \end{equation}
  If in addition $u \in C^0_t C^1_{\tmop{loc}}$, we can integrate by parts
  in~{\eqref{sec5.2 transport eq3}} back to obtain
  \[ \left| \langle u_{s, t}, \varphi \rangle + \langle T^w_{s, t} b \cdot
     \nabla u_s, \, \varphi \rangle \right| \lesssim_K | t - s |^{2 \gamma} \|
     \varphi \|_{C^1}  \llbracket u \rrbracket_{C^{\gamma} C^0_K}  \llbracket
     T^w b \rrbracket_{C^{\gamma} C^1} ; \]
  if $u \in C^{\gamma}_t C^1_{\tmop{loc}}$, then this necessarily implies the
  pointwise identity
  \begin{equation}
    u (t, x) = u_0 (x) + \int_0^t T^w b (\mathd s, x) \cdot \nabla u (s, x)
    \quad \text{for all } (t, x) \in [0, T] \times \mathbb{R}^d \label{sec5.2
    transport eq4}
  \end{equation}
  which is meaningul since $T^w b (\cdot, x), \nabla u (\cdot, x) \in
  C^{\gamma}_t$. It is therefore clear that for regular $b$, any classical
  solution of~{\eqref{sec5.1 transport eq}} is also a solution in the sense of
  Definition~\ref{sec5.2 defn transport}.
\end{remark}

We start by showing that our candidate solution satisfies
Definition~\ref{sec5.2 defn transport}.

\begin{lemma}
  \label{sec5.2 lemma existence transport}Let $u_0 \in C^1_x$ and define $u
  (t, x) = u_0 (\psi_t (x))$, then $u \in C^{\gamma}_t C^0_x \cap C^0_t
  C^1_{\tmop{loc}}$ and it is a solution of the Young transport
  equation~{\eqref{sec5.2 transport eq}}.
\end{lemma}

\begin{proof}
  The regularity of $u \in C^{\gamma}_t C^0_x \cap C^0_t C^1_{\tmop{loc}}$
  follows from Remark~\ref{sec4.3 remark transport}, since $\psi$ satisfies
  \[ \sup_x | \psi_t (x) - \psi_s (x) | \lesssim | t - s |^{\gamma}, \quad
     \sup_{t, x} | D_x \psi (t, x) | < \infty, \]
  combined with the regularity of $u_0$. Recall that by {\eqref{sec4.3
  jacobian equation}}, for any $s < t$ it holds
  \begin{align*}
    \langle u_{s, t}, \varphi \rangle & = \int_{\mathbb{R}^d} (u_0 (\psi_t
    (x)) - u_0 (\psi_s (x))) \varphi (x) \mathd x = \int_{\mathbb{R}^d} u_0
    (x) F (s, t, x) \mathd x
  \end{align*}
  where
  \[ F (s, t, x) = \varphi (\Phi_t (x)) \exp \left( \int_0^t \tmop{div} T^w b
     (r, \Phi_r (x)) \right) - \varphi (\Phi_s (x)) \exp \left( \int_0^s
     \tmop{div} T^w b (r, \Phi_r (x)) \right) . \]
  By Young chain rule, we have the estimates
  \begin{eqnarray*}
    | \varphi (\Phi_t (x)) - \varphi (\Phi_s (x)) - \nabla \varphi (\Phi_s
    (x)) \cdot T^w_{s, t} b (\Phi_s (x)) | & \lesssim & | t - s |^{2 \alpha}\\
    \left| \exp \left( \int_s^t \tmop{div} T^w b (r, \Phi_r (x)) \right) - 1 -
    \tmop{div} T^w_{s, t} b (\Phi_s (x)) \right| & \lesssim & | t - s |^{2
    \alpha}\\
    \left| \exp \left( \int_0^s \tmop{div} T^w b (r, \Phi_r (x)) \right)
    \right| & \lesssim & 1
  \end{eqnarray*}
  which can all be taken uniform over $x$ belonging to a compact set $K$;
  combining them we deduce that
  \begin{align*}
    F (s, t, x) & \approx [\nabla \varphi (\Phi_s (x)) \cdot T^w_{s, t} b
    (\Phi_s (x)) + \tmop{div} T^w_{s, t} b (\Phi_s (x))] \exp \left( \int_0^s
    \tmop{div} T^w b (r, \Phi_r (x)) \right)\\
    & = \tmop{div} (T^w_{s, t} b \varphi) (\Phi_s (x)) \exp \left( \int_0^s
    \tmop{div} T^w b (r, \Phi_r (x)) \right)
  \end{align*}
  in the sense of the equality holding up to a term of order $| t - s |^{2
  \alpha}$. Therefore
  \begin{align*}
    \langle u_{s, t}, \varphi \rangle & \approx \int_{\mathbb{R}^d} u_0 (x)
    \tmop{div} (T^w_{s, t} b \varphi) (\Phi_s (x)) \exp \left( \int_0^s
    \tmop{div} T^w b (r, \Phi_r (x)) \right) \mathd x\\
    & = \int_{\mathbb{R}^d} u_0 (\psi_s (x)) \tmop{div} (T^w_{s, t} b
    \varphi) (x) \mathd x = \langle u_s, \tmop{div} (T^w_{s, t} b \varphi)
    \rangle
  \end{align*}
  which implies the conclusion.
\end{proof}

\begin{remark}
  \label{sec5.2 remark higher regularity}If $T^w b \in C^{\gamma}_t C^3_x$,
  then by Theorem~\ref{sec4.3 thm higher regularity} and Remark~\ref{sec4.3
  remark transport} $\psi \in C^0_t C^2_x \cap C^{\gamma}_t C^1_x$. It is then
  possible to check with standard calculations that for $u_0 \in C^2_x$, the
  solution $u$ constructed as above belongs to $C^0_t C^2_x \cap C^{\gamma}_t
  C^1_x$ as well. Therefore in this case, by Remark~\ref{sec5.2 remark defn
  transport}, $u$ also satisfies the stronger pointwise
  identity~{\eqref{sec5.2 transport eq4}}.
\end{remark}

By the method of characteristics, we are able to obtain the following
preliminary uniqueness result. It is however of limited applicability, see
Remark~\ref{sec5.2 remark limited use} below.

\begin{lemma}
  \label{sec5.2 baby lemma uniqueness}Let $T^w b \in C^{\gamma}_t C^2_x$, $u
  \in C^{\gamma}_t C^2_{\tmop{loc}}$ be a solution of~{\eqref{sec5.2 transport
  eq}}. Then $u (t, x) = u_0 (\psi_t (x))$.
\end{lemma}

\begin{proof}
  In order to conclude, it is enough to show that the function $f_t \assign u
  (t, \Phi_t (x))$ is constant; in particular, it suffices to prove that $|
  f_{s, t} | \lesssim | t - s |^{2 \gamma}$ since $\gamma > 1 / 2$. By the
  regularity assumption on $u$, it satisfies~{\eqref{sec5.2 transport eq4}}
  and therefore
  \[ | u_{s, t} (x) + \nabla u_s (x) \cdot T^w_{s, t} b (x) | \lesssim_R | t -
     s |^{2 \gamma}, \quad \forall \, y \in B_R . \]
  Choosing appropriately $R$ we have
  \begin{eqnarray*}
    f_{s, t} & = & u_{s, t} (\Phi_t (x)) + u_s (\Phi_t (x)) - u_s (\Phi_s
    (x))\\
    & = & u_{s, t} (\Phi_s (x)) + \nabla u_s (\Phi_s (x)) \cdot \Phi_{s, t}
    (x) + O (| t - s |^{2 \gamma})\\
    & = & - \nabla u_s (\Phi_s (x)) \cdot T^w_{s, t} b (y) + \nabla u_s
    (\Phi_s (x)) \cdot \Phi_{s, t} (x) + O (| t - s |^{2 \gamma})\\
    & = & O (| t - s |^{2 \gamma})
  \end{eqnarray*}
  where in the last passage we used the fact that $\Phi_{s, t} (x) = \int_s^t
  T^w b (\mathd r, \Phi_r (x)) \mathd r$.
\end{proof}

\begin{remark}
  \label{sec5.2 remark limited use}The hypothesis $u \in C^{\gamma}_t C^2_x$
  is required in order to justify the passage
  \[ u_s (\Phi_t (x)) - u_s (\Phi_s (x)) = \nabla u_s (\Phi_s (x)) \cdot
     \Phi_{s, t} (x) + O (| t - s |^{2 \gamma}) \]
  which is not true in general under the sole assumption $u \in C^{\gamma}_t
  C^1_x$. However, for general $T^w b \in C^{\gamma}_t C^2_x$, we only know
  that $\psi \in C^{\alpha}_t C^0_{\tmop{loc}} \cap C^0_t C^1_{\tmop{loc}}$
  and so the solution constructed by $u_t (x) = u_0 (\psi_t (x))$ is not a
  priori in the class $C^{\gamma}_t C^2_x$. For this reason, Lemma~\ref{sec5.2
  baby lemma uniqueness} is potentially vacuous, as it might only imply the
  non existence of $C^{\gamma}_t C^2_x$\mbox{-}solutions, while leaving open
  the problem of uniqueness in the class where $u$ constructed as in
  Lemma~\ref{sec5.2 lemma existence transport} lives. We believe the same
  issue arises in Theorems~3.6 and~3.7 from~{\cite{hu}}, which do not settle
  the problem of uniqueness.
\end{remark}

Observe that the above issue is typical of the Young regime and is completely
absent in the case $b \in C^0_{t, x}$, where uniqueness follows immediately
from standard arguments.

In order to prove uniqueness of solutions to~{\eqref{sec5.2 transport eq}} in
the class $C^{\gamma}_t C^0_{\tmop{loc}} \cap C^0_t C^1_{\tmop{loc}}$, we need
to use an appropriate commutator lemma, in the style of~{\cite{diperna}}. The
basic idea is as follows: let $\{ \rho_{\varepsilon} \}_{\varepsilon > 0}$ be
a family of standard mollifiers (assume $\rho^{}_1 = \rho$ to be supported on
$B_1$ for simplicity), denote $u^{\varepsilon} = \rho_{\varepsilon} \ast u$;
by equation~{\eqref{sec5.2 transport eq5}} we deduce that for any $R > 0$,
adopting the notation $C^0_R = C^0_{B_R}$, it holds
\[ \| u^{\varepsilon}_{s, t} + T^w_{s, t} b \cdot \nabla u^{\varepsilon}_s +
   R^{\varepsilon} (u_s, T^w_{s, t} b) \|_{C^0_R} \lesssim_{\varepsilon, R} |
   t - s |^{2 \gamma} \quad \text{uniformly in } 0 \leqslant s \leqslant t
   \leqslant T \]
where the estimate is uniform in $\varepsilon > 0$ and the commutator
$R^{\varepsilon}$ appearing is the bilinear operator
\begin{equation}
  R^{\varepsilon} (h, g) = (g \cdot \nabla h)^{\varepsilon} - g \cdot \nabla
  h^{\varepsilon} = \rho^{\varepsilon} \ast (g \cdot \nabla h) - g \cdot
  \nabla (\rho^{\varepsilon} \ast h) . \label{sec5.2 commutator}
\end{equation}
Now $u^{\varepsilon} \in C^{\gamma}_t C^2_{\tmop{loc}}$ and so we can apply
the same idea of the proof of Lemma~\ref{sec5.2 baby lemma uniqueness}, i.e.
study the function $f^{\varepsilon}_t = u^{\varepsilon}_t (\Phi_t (x))$, which
we expect to be quasi constant; in the estimates, terms of the form
$R^{\varepsilon} (u_s, T^w_{s, t} b) (\Phi_s (x))$ will then start to appear,
and so we need to control them as $\varepsilon \rightarrow 0$. For this reason
we need the following lemma.

\begin{lemma}
  \label{sec5.2 lemma baby commutator}The operator $R^{\varepsilon} :
  C_{\tmop{loc}}^0 \times C_{\tmop{loc}}^1 \rightarrow C_{\tmop{loc}}^0$
  defined by~{\eqref{sec5.2 commutator}} is such that:
  \begin{enumerateroman}
    \item There exists a constant $C$ independent of $\varepsilon$ such that
    $\| R^{\varepsilon} (h, g) \|_{C^0_R} \leqslant C \| h \|_{C_{R + 1}^0} \|
    g \|_{C_{R + 1}^1}$;
    
    \item For any fixed $h \in C^0$, $g \in C^1$ it holds $R^{\varepsilon} (h,
    g) \rightarrow 0$ uniformly on compact sets as $\varepsilon \rightarrow
    0$.
  \end{enumerateroman}
  Similar statements hold for $R^{\varepsilon} : C_{\tmop{loc}}^1 \times
  C_{\tmop{loc}}^2 \rightarrow C_{\tmop{loc}}^1$.
\end{lemma}

\begin{proof}
  The proof is analogue to the one of Lemma~II.1 from~{\cite{diperna}}. It
  holds
  \[ R^{\varepsilon} (h, g) (x) = \int_{B_1} h (x - \varepsilon z) \frac{g (x
     - \varepsilon z) - g (x)}{\varepsilon} \cdot \nabla \rho (z) \mathd z -
     (h \tmop{div} g)^{\varepsilon} (x) . \]
  Thus claim~\tmtextit{i.} follows from $\| (h \tmop{div} g)^{\varepsilon}
  \|_{C_R^0} \leqslant \| h \|_{C_{R + 1}^0} \| g \|_{C_{R + 1}^1}$ and
  \[ \left| \int_{B_1} h (x - \varepsilon z) \frac{g (x - \varepsilon z) - g
     (x)}{\varepsilon} \cdot \nabla \rho (z) \mathd z \right| \leqslant \| h
     \|_{C_{R + 1}^0} \| g \|_{C_{R + 1}^1} \| \nabla \rho \|_{L^1} \]
  where the estimate is uniform in $x \in B_R$. Now fix $R > 0$; we can assume
  that $h$, $g$, $D g$ all have modulus of continuity $\omega$ on $B_{R + 1}$.
  By known properties of convolutions, $(h \tmop{div} g)^{\varepsilon}
  \rightarrow h \tmop{div} g$ uniformly on compact sets; moreover for all $x
  \in B_R$ it holds
  \[ \left| \frac{g (x - \varepsilon z) - g (x)}{\varepsilon} - D g (x) z
     \right| = \left| \int_0^1 [D g (x - \varepsilon \theta z) - D g (x)] z
     \mathd \theta \right| \leqslant \omega (\varepsilon) ; \]
  combined with a similar estimate for $| h (x - \varepsilon z) - h (x) |$,
  this implies that, uniformly in $x \in B_R$,
  \[ \int_{B_1} h (x - \varepsilon z) \frac{g (x - \varepsilon z) - g
     (x)}{\varepsilon} \cdot \nabla \rho (z) \mathd z \rightarrow h (x) 
     \int_{B_1} \nabla \rho (z) \cdot D b (x) z \mathd z = h (x) \tmop{div} b
     (x) \]
  which implies claim~\tmtextit{ii.} . The statements for $R^{\varepsilon} :
  C_{\tmop{loc}}^1 \times C_{\tmop{loc}}^2 \rightarrow C_{\tmop{loc}}^1$
  follow immediately once we observe that $\partial_i R^{\varepsilon} (h, g) =
  R^{\varepsilon} (\partial_i h, g) + R^{\varepsilon} (h, \partial_i g)$ as we
  apply the previous results.
\end{proof}

We have now all the ingredients to show uniqueness in the class $C^{\gamma}_t
C^0_x \cap C^0_t C^1_x$.

\begin{theorem}
  \label{sec5.2 thm uniqueness transport}Let $T^w b \in C^{\gamma}_t C^2_x$
  and $u \in C^{\gamma}_t C^0_x \cap C^0_t C^1_x$ be a solution
  of~{\eqref{sec5.2 transport eq}}. Then
  \[ u (t, x) = u_0 (\psi_t (x)) \quad \forall \, (t, x) \in [0, T] \times
     \mathbb{R}^d . \]
\end{theorem}

\begin{proof}
  As before, it is enough to show that for any $x \in \mathbb{R}^d$, the
  function $f_t : = u_t (\Phi_t (x))$ satisfies $| f_{s, t} | \lesssim | t - s
  |^{2 \gamma}$, as it implies that $f$ is constant. Recall that $\Phi$
  satisfies the estimate $| x - \Phi_t (x) | \lesssim | t |^{\gamma}$
  uniformly in $x$, therefore we can fix $B_R$ such that $\Phi_t (x) \in B_{2
  R}$ for all $t \in [0, T]$ and all $x \in B_R$; from now on all the norms
  appearing will be localised on $B_{2 R}$ without writing it explicitly.
  
  Since $u$ is a solution of~{\eqref{sec5.2 transport eq}}, it
  satisfies~{\eqref{sec5.2 transport eq5}} and therefore $u^{\varepsilon}$ is
  such that
  \[ \| u^{\varepsilon}_{s, t} + \nabla u^{\varepsilon}_s \cdot T^w_{s, t} b +
     R^{\varepsilon} (u_s, T^w_{s, t} b) \|_{C^0} \lesssim | t - s |^{2
     \gamma} \quad \text{uniformly in } 0 \leqslant s \leqslant t \leqslant T.
  \]
  Define $f^{\varepsilon}_t = u^{\varepsilon}_t (\Phi_t (x))$; using the above
  property and going through similar calculations as in the proof of
  Lemma~\ref{sec5.2 baby lemma uniqueness}, we deduce that
  \begin{equation}
    | f^{\varepsilon}_{s, t} - R^{\varepsilon} (u_s, T^w_{s, t} b) (\Phi_s
    (x)) | \lesssim_{\varepsilon} | t - s |^{2 \gamma} . \label{sec5.2 thm
    uniqueness eq1}
  \end{equation}
  The estimate above a priori depends on $\varepsilon$, as it involves $\|
  u^{\varepsilon} \|_{C^{\gamma} C^2}$, but we are now going to show that
  under the assumptions on $T^w b$ and $u$ it is actually uniform in
  $\varepsilon > 0$. This is accomplished with the help of the Sewing lemma,
  see Lemma~\ref{appendixA1 sewing lemma} from Appendix~\ref{appendixA1}.
  Define
  \[ \Gamma^{\varepsilon}_{s, t} \assign R^{\varepsilon} (u_s, T^w_{s, t} b)
     (\Phi_s (x)), \]
  so that relation~{\eqref{sec5.2 thm uniqueness eq1}} can be rephrased as $|
  f^{\varepsilon}_{s, t} - \Gamma^{\varepsilon}_{s, t} | \lesssim | t - s |^{2
  \gamma}$. We can estimate $\| \delta \Gamma^{\varepsilon} \|_{2 \gamma}$ as
  follows:
  \begin{eqnarray*}
    | \delta \Gamma^{\varepsilon}_{s, u, t} | & = & | \Gamma^{\varepsilon}_{s,
    t} - \Gamma^{\varepsilon}_{s, u} - \Gamma^{\varepsilon}_{u, t} |\\
    & \leqslant & | R^{\varepsilon} (u_s, T^w_{s, t} b) (\Phi_s (x)) -
    R^{\varepsilon} (u_s, T^w_{u, t} b) (\Phi_u (x)) | + | R^{\varepsilon}
    (u_{s, u}, T^w_{u, t} b) (\Phi_u (x)) |\\
    & \leqslant & \| R^{\varepsilon} (u_s, T^w_{u, t} b) \|_{C^1} | \Phi_s
    (x) - \Phi_u (x) | + \| R^{\varepsilon} (u_{s, u}, T^w_{u, t} b)
    \|_{C^0}\\
    & \lesssim & | t - s |^{2 \gamma} (\| R^{\varepsilon} \| \| u \|_{C^0
    C^1}  \llbracket T^w b \rrbracket_{C^{\gamma} C^2} \llbracket \Phi_{\cdot}
    (x) \rrbracket_{\gamma} + \| R^{\varepsilon} \| \llbracket u
    \rrbracket_{C^{\gamma} C^0} \llbracket T^w b \rrbracket_{C^{\gamma}
    C^1})\\
    & \lesssim & | t - s |^{2 \gamma}
  \end{eqnarray*}
  where we used the fact that $\llbracket \Phi_{\cdot} (x) \rrbracket_{\gamma}
  \lesssim 1$ by Theorem~\ref{sec4.3 theorem existence flow} and the estimate
  is uniform in $\varepsilon$, since $\| R^{\varepsilon} \|_{\mathcal{L}^2
  (C^i \times C^{i + 1} ; C^i)} \leqslant C_1$ for $i = 1, 2$ by
  Lemma~\ref{sec5.2 lemma baby commutator}. It follows that $\| \delta
  \Gamma^{\varepsilon} \|_{2 \gamma} \leqslant C_2$ for some constant
  independent of $\varepsilon$ and therefore by Lemma~\ref{appendixA1 sewing
  lemma} (specifically estimate~{\eqref{appendixA1 sewing property 1}}) there
  exists $C_3$ such that
  \begin{equation}
    | u^{\varepsilon}_t (\Phi_t (x)) - u^{\varepsilon}_s (\Phi_s (x)) -
    R^{\varepsilon} (u_s, T^w_{s, t} b) (\Phi_s (x)) | = | f^{\varepsilon}_{s,
    t} - \Gamma^{\varepsilon}_{s, t} | \leqslant C_3 | t - s |^{2 \gamma}
    \quad \forall \, \varepsilon > 0, \, s < t. \label{sec5.2 thm uniqueness
    eq2}
  \end{equation}
  Since $u^{\varepsilon}_t (\Phi_t (x)) \rightarrow u_t (\Phi_t (x))$ and by
  part~\tmtextit{ii.} of Lemma~\ref{sec5.2 lemma baby commutator}
  $R^{\varepsilon} (u_s, T^w_{s, t} b) (\Phi_s (x)) \rightarrow 0$, taking the
  limit as $\varepsilon \rightarrow 0$ in~{\eqref{sec5.2 thm uniqueness eq2}}
  we deduce that $| u_t (\Phi_t (x)) - u_s (\Phi_s (x)) | \leqslant C_3 | t -
  s |^{2 \gamma}$, which gives the conclusion.
\end{proof}

We now pass to study weak solutions of the continuity equation associated to
$T^w b$. Given a distribution $v$, we say that $v \in (C^1_x)^{\ast}$ if there
exists a constant $C$ such that $| \langle v, \varphi \rangle | \leqslant C \|
\varphi \|_{C^1}$ for all smooth $\varphi$. We denote by $\| v
\|_{(C^1)^{\ast}}$ the optimal constant $C$. Note that, when $v$ is a measure,
$\| v_{s, t} \|_{(C^1)^{\ast}}$ is the $1$-Wasserstein distance between $v_t$
and $v_s$.

\begin{definition}
  \label{sec5.2 defn weak continuity eq}Let $T^w b \in C^{\gamma}_t C_x^1$ and
  let $v \in C^{\gamma}_t (C^1_x)^{\ast}$. We say that $v$ is a weak solution
  of the Young continuity equation
  \begin{equation}
    v (\mathd t) + \tmop{div} (v_t T^w b (\mathd t)) = 0, \label{sec5.2
    continuity equation}
  \end{equation}
  if there exists a constant $C$ such that for all $\varphi \in C^2
  (\mathbb{R}^d)$ the following holds:
  \begin{equation}
    | \langle v_{s, t}, \varphi \rangle - \langle v_s, T^w_{s, t} b \cdot
    \nabla \varphi \rangle | \leqslant C \| \varphi \|_{C^2} | t - s |^{2
    \gamma} . \label{sec5.2 solution continuity equation}
  \end{equation}
\end{definition}

\begin{remark}
  As before, it can be shown that for smooth $b$, any classical solution of
  \[ \partial_t v + \tmop{div} (v \tilde{b}) = 0, \]
  is also a solution in the sense of the definition above.
  Equations~{\eqref{sec5.2 continuity equation}} and~{\eqref{sec5.2 solution
  continuity equation}} can be rephrased as $v$ satisfying the functional
  Young integral equation
  \[ v_{s, t} = \tmop{div} \left( \int_s^t v_r T^w b (\mathd r) \right), \]
  where the integral inside the divergence is a well defined element of
  $(C^1)^{\ast}$ since the product between $C^1$ and $(C^1)^{\ast}$ is still
  an element of $(C^1)^{\ast}$. Formulation~{\eqref{sec5.2 solution continuity
  equation}} is however more useful for our purposes.
\end{remark}

\begin{lemma}
  \label{sec5.2 lemma existence continuity}Let $T^w b \in C^{\gamma}_t C^2_x$,
  $v_0 \in \mathcal{M}_x$ and define $v \in L^{\infty}_t \mathcal{M}_x$ by
  \[ \langle v_t, \varphi \rangle = \int_{\mathbb{R}^d} \varphi (\Phi_t (x))
     v_0 (\mathd x) \quad \forall \, \varphi \in C^{\infty}_c . \]
  Then $v$ is a weak solution of~{\eqref{sec5.2 continuity equation}} with
  initial condition $v_0$.
\end{lemma}

\begin{proof}
  Let us first show that $v$ defined as above belongs to $C^{\gamma}_t
  (C^1_x)^{\ast}$. It holds
  \begin{align*}
    | \langle v_{s, t}, \varphi \rangle | & = \, \left| \int_{\mathbb{R}^d}
    [\varphi (\Phi_t (x)) - \varphi (\Phi_s (x))] v_0 (\mathd x) \right|\\
    & \leqslant \| \varphi \|_{\tmop{Lip}} \sup_{x \in \mathbb{R}^d} |
    \Phi_{s, t} (x) |  \| v_0 \|_{\mathcal{M}} \lesssim | t - s |^{\gamma} \|
    \varphi \|_{\tmop{Lip}} \| v_0 \|_{\mathcal{M}}
  \end{align*}
  where we used estimate~{\eqref{sec4.1 thm existence a priori estimate 1}};
  it follows that $\| v \|_{C^{\gamma}_t (C^1_x)^{\ast}} \lesssim \| v_0
  \|_{\mathcal{M}}$. We now check that $v$ is a solution in the sense of
  Definition~\ref{sec5.2 defn weak continuity eq}. It holds
  \[ | \varphi (\Phi_t (x)) - \varphi (\Phi_s (x)) - \nabla \varphi (\Phi_s
     (x)) \cdot \Phi_{s, t} (x) | \lesssim \| \varphi \|_{C^2_b}  | \Phi_{s,
     t} (x) |^2 \lesssim \| \varphi \|_{C^2_b}  | t - s |^{2 \gamma} \]
  where as before we used~{\eqref{sec4.1 thm existence a priori estimate 1}}
  and the estimate is uniform in $x$; similarly
  \[ | \Phi_{s, t} (x) - T^w_{s, t} b (\Phi_s (x)) | \lesssim | t - s |^{2
     \gamma} \| T^w b \|_{C^{\gamma} C^1} \llbracket \Phi_{\cdot} (x)
     \rrbracket_{C^{\gamma}} \lesssim | t - s |^{2 \gamma} . \]
  Combining the two estimates we obtain
  \[ | \varphi (\Phi_t (x)) - \varphi (\Phi_s (x)) - \nabla \varphi (\Phi_s
     (x)) \cdot T^w_{s, t} b (\Phi_s (x)) | \lesssim \| \varphi \|_{C^2_b} | t
     - s |^{2 \gamma} \]
  which yields
  \begin{align*}
    | \langle v_{s, t}, \varphi \rangle - \langle v_s, T^w_{s, t} b \cdot
    \nabla \varphi \rangle | & \leqslant \, \left| \int_{\mathbb{R}^d}
    [\varphi (\Phi_t (x)) - \varphi (\Phi_s (x)) - \nabla \varphi (\Phi_s (x))
    \cdot T^w_{s, t} b (\Phi_s (x))] v_0 (\mathd x) \right|\\
    & \lesssim \, \| \varphi \|_{C^2_b}  \| v_0 \|_{\mathcal{M}}  | t - s
    |^{2 \gamma}
  \end{align*}
  and thus the conclusion.
\end{proof}

\begin{theorem}
  \label{sec5.2 thm uniqueness continuity}For any given $v_0 \in
  \mathcal{M}_x$, there exists a unique weak solution of~{\eqref{sec5.2
  continuity equation}} in the class $v \in L^{\infty}_t M_x \cap C^{\gamma}_t
  (C^1_x)^{\ast}$, which is given by the one from Lemma~\ref{sec5.2 lemma
  existence continuity}.
\end{theorem}

\begin{proof}
  As before, by linearity it is enough to show that there exists a unique
  solution for the initial condition $v_0 \equiv 0$. The basic strategy is the
  usual one: given any $u_0 \in C^{\infty}_c$, setting $u_t (x) = u_0 (\psi_t
  (x))$, it is enough to show that the function $f_t \assign \langle v_t, u_t
  \rangle$ is constant, as it implies
  \[ \langle v_t, u_t \rangle = \int u_0 (\psi_t (x)) v_t (\mathd x) = 0, \]
  and thus reasoning as in the proof of Proposition~\ref{sec5.1 thm
  continuity} that $v_t \equiv 0$. Observe that the function $u$ has compact
  space-time support, so we don't need to introduce localisations here.
  
  Now we reason following the same lines as in Theorem~\ref{sec5.2 thm
  uniqueness transport}, namely we spatially mollify $u$ so that now
  $u^{\varepsilon}$ solves
  \begin{equation}
    u^{\varepsilon}_{s, t} (x) + \nabla u^{\varepsilon}_s (x) \cdot T^w_{s, t}
    b (x) = R^{\varepsilon} (u_s, T^w_{s, t} b) (x) + O_{\varepsilon} (| t - s
    |^{2 \gamma}) \label{sec5.2 thm uniqueness continuity eq1}
  \end{equation}
  and all the terms are in $C^1_x$ due to the mollification. Define
  $f^{\varepsilon}_t = \langle v_t, u^{\varepsilon}_t \rangle$, then
  \[ f^{\varepsilon}_{s, t} = \langle v_{s, t}, u^{\varepsilon}_s \rangle +
     \langle v_s, u^{\varepsilon}_{s, t} \rangle + \langle v_{s, t},
     u^{\varepsilon}_{s, t} \rangle . \]
  The last term trivially satisfies $| \langle v_{s, t}, u^{\varepsilon}_{s,
  t} \rangle | \lesssim_{\varepsilon} | t - s |^{2 \gamma}$. Combining the
  estimates
  \[ | \langle v_{s, t}, u^{\varepsilon}_s \rangle - \langle v_s, T^w_{s, t} b
     \cdot \nabla u^{\varepsilon}_s \rangle | \lesssim \| u^{\varepsilon}_s
     \|_{C^2} | t - s |^{2 \gamma} \lesssim_{\varepsilon} | t - s |^{2 \gamma}
  \]
  \[ | \langle v_s, u^{\varepsilon}_{s, t} \rangle + \langle v_{s,} \nabla
     u^{\varepsilon}_s \cdot T^w_{s, t} b \rangle - \langle v_s,
     R^{\varepsilon} (u_s, T^w_{s, t} b) \rangle | \lesssim_{\varepsilon} \|
     v_s \|_{(C^1)^{\ast}} | t - s |^{2 \gamma} \]
  which come respectively from $v$ being a solution of~{\eqref{sec5.2 solution
  continuity equation}} and~{\eqref{sec5.2 thm uniqueness continuity eq1}}
  above, we overall obtain
  \[ | f^{\varepsilon}_{s, t} - \langle v_s, R^{\varepsilon} (u_s, T^w_{s, t}
     b) \rangle | \lesssim_{\varepsilon} | t - s |^{2 \gamma} . \]
  As before, the estimate a priori depends on $\varepsilon$, but we can apply
  the Sewing lemma for the choice $\Gamma_{s, t} = \langle v_s,
  R^{\varepsilon} (u_s, T^w_{s, t} b) \rangle$ for which, by analogue
  computations to the ones of Theorem~\ref{sec5.2 thm uniqueness transport},
  it holds
  \[ \| \delta \Gamma \|_{2 \gamma} \leqslant \| R^{\varepsilon} \| 
     (\llbracket v \rrbracket_{C^{\gamma} (C^1)^{\ast}}  \| u \|_{C^0 C^1} 
     \llbracket T^w b \rrbracket_{C^{\gamma} C^2} + \| v \|_{L^{\infty} M} 
     \llbracket u \rrbracket_{C^{\gamma} C^0}  \llbracket T^w b
     \rrbracket_{C^{\gamma} C^1}) \lesssim 1 \]
  uniformly in $\varepsilon > 0$. Therefore there exists a constant $C$
  independent of $\varepsilon$ such that
  \[ | \langle v_t, u^{\varepsilon}_t \rangle - \langle v_s, u^{\varepsilon}_s
     \rangle - \langle v_s, R^{\varepsilon} (u_s, b_{s, t}) \rangle |
     \leqslant C | t - s |^{2 \gamma} . \]
  By the properties of $R^{\varepsilon}$, taking $\varepsilon \rightarrow 0$
  we deduce $| \langle v_t, u_t \rangle - \langle v_s, u_s \rangle | \lesssim
  | t - s |^{2 \gamma}$ which implies the conclusion.
\end{proof}

\appendix\section{Some tools}\label{appendix}

This appendix collect some technical estimates and some reminder of various
standard results, from certain functional spaces to stochastic integration in
Banach setting.

\subsection{Some useful lemmas}\label{appendixA1}

The following chaining lemma is a slight variation on the one
from~{\cite{catelliergubinelli}}, Lemma~3.1.

\begin{lemma}
  \label{appendixA1 chaining lemma}Let $E$ be a Banach space and let $X : [0,
  T] \rightarrow E$ be a continuous stochastic process such that, for some
  $\lambda > 0$,
  \begin{equation}
    \mathbb{E} \left[ \exp \left( \lambda \frac{\| X_t - X_s \|^2_E}{| t - s
    |^{2 \alpha}} \right) \right] \leqslant C \quad \forall \, s \neq t \in
    [0, T] .
  \end{equation}
  Then $\mathbb{P}$-a.s. $X \in C^{\omega} E$ for the modulus of continuity
  $\omega (| t - s |) = | t - s |^{\alpha} \sqrt{- \log | t - s |}$ and there
  exists $\beta > 0$ such that
  \[ \mathbb{E} [\exp (\beta \llbracket X \rrbracket^2_{C^{\omega} E})] <
     \infty . \]
  In particular, if $X_0 \equiv 0$, then for any $\gamma < \alpha$ there
  exists $\beta > 0$ such that
  \[ \mathbb{E} [\exp (\beta \| X \|^2_{C^{\gamma} E})] < \infty . \]
\end{lemma}

\begin{proof}
  Without loss of generality we can assume $T = 1$. Also, we will only show
  that proof in the case $\alpha = 1 / 2$, the other cases being entirely
  analogue. Let us define the random variable
  \[ R (\lambda) = \sum_{n \in \mathbb{N}} \sum_{k = 0}^{2^n - 1} 2^{- 2 n} \,
     \exp \left( \mu \frac{\| X_{(k + 1) 2^{- n}} - X_{k 2^{- n}} \|^2_E}{2^{-
     n}} \right) . \]
  Then it follows from the assumption that $\mathbb{E} [R (\lambda)] \leqslant
  C$. We can then apply Lemma~3.1 from~{\cite{catelliergubinelli}} to deduce
  that there exist deterministic positive constants $K, \beta$ such that
  \[ \exp \left( \beta \frac{\| X_t - X_s \|^2_E}{| t - s |} \right) \lesssim
     | t - s |^{- K} \, R (\lambda) \quad \forall \, s \neq t \]
  which implies by taking the logarithm and dividing by $- \log | t - s |$
  that
  \[ \exp \left( \beta \left( \sup_{s \neq t} \frac{\| X_t - X_s \|_E}{| t - s
     | \sqrt{- \log | t - s |}} \right)^2 \right) = \sup_{s \neq t} \, \exp
     \left( \beta \frac{\| X_t - X_s \|^2_E}{| t - s | (- \log | t - s |)}
     \right) \lesssim R (\lambda) \]
  which yields the conclusion. Alternatively, it follows from the assumption
  that
  \[ \mathbb{E} [B] \assign \mathbb{E} \left[ \int_{[0, T]^2} \exp \left(
     \lambda \frac{\| X_t - X_s \|^2_E}{| t - s |^{2 \alpha}} \right) \,
     \mathd t \, \mathd s \right] < \infty \]
  which implies that we can apply Garsia-Rodemich-Rumsey Theorem
  (see~{\cite{garsia}}) for the choice $\psi (x) = e^{\lambda x^2}$, $p (x) =
  x^{\alpha}$, which gives
  \begin{align*}
    \| X_t - X_s \|_E & \lesssim \int_0^{| t - s |} \sqrt{B - \log \, u} \,
    u^{\alpha - 1} \, \mathd u \lesssim \left( \sqrt{B} + \sqrt{- \log | t - s
    |} \right) | t - s |^{\alpha}
  \end{align*}
  and from which we can again deduce that
  \[ \sup_{s \neq t} \frac{\| X_t - X_s \|_E}{| t - s | \sqrt{- \log | t - s
     |}} \lesssim 1 + \sqrt{B} \]
  and the exponential integrability bound. The final claim follows
  immediately.
\end{proof}

We also provide here a simple lemma on a priori bounds on solutions to linear
Young differential equations, in the style of Section~6.2 from~{\cite{lejay}}.

\begin{proposition}
  \label{appendixA1 lemma bound linear YDE}Let $A \in C^{\gamma}_t (0, T ;
  \mathcal{L} (\mathbb{R}^d ; \mathbb{R}^d))$, $h \in C^{\gamma} ([0, T] ;
  \mathbb{R}^d)$ and $\gamma > 1 / 2$. Then there exists a unique solution to
  the YDE
  \begin{equation}
    x_t = x_0 + \int_0^t A_{\mathd s} x_s + h_t \label{appendixA1 linear YDE}
  \end{equation}
  and there exist suitable positive constants which only depend on $\gamma$
  such that
  \begin{equation}
    \llbracket x \rrbracket_{C^{\gamma}} \lesssim \llbracket A
    \rrbracket_{C^{\gamma}} \| x \|_{C^0} + \llbracket h
    \rrbracket_{C^{\gamma}} ; \label{appendixA1 bound linear YDE}
  \end{equation}
  \begin{equation}
    \| x \|_{C^0} \lesssim e^{C \llbracket A \rrbracket_{C^{\gamma}}^{1 /
    \gamma} T}  (| x_0 + h_0 | + \llbracket h \rrbracket_{C^{\gamma}}) .
    \label{appendixA1 bound2 linear YDE}
  \end{equation}
\end{proposition}

\begin{proof}
  Since $A \in C^{\gamma}_t C^{\infty}_x$, uniqueness of solutions is well
  known (see for instance~{\cite{lejay}}), so we are only interested in
  proving the bounds~{\eqref{appendixA1 bound linear YDE}}
  and~{\eqref{appendixA1 bound2 linear YDE}}. Up to renaming $x_0$, we can
  assume $h_0 = 0$; we can also assume up to rescaling everything that $T =
  1$.
  
  We adopt the following notation: for $\Delta \leqslant 1$, we consider
  \[ \llbracket x \rrbracket_{\gamma, \Delta} \assign
     \sup_{\tmscript{\begin{array}{l}
       0 \leqslant s < t \leqslant T\\
       | s - t | \leqslant \Delta
     \end{array}}}  \frac{| x_{s, t} |}{| t - s |^{\gamma}} . \]
  Let $\Delta > 0$ to be chosen later, $s < t$ such that $| t - s | \leqslant
  \Delta$, by~{\eqref{appendixA1 linear YDE}} it holds
  \begin{align*}
    | x_{s, t} | & \leqslant \, \left| \int_s^t A_{\mathd r} x_r \right| + |
    h_{s, t} |\\
    & \leqslant \, | A_{s, t} x_s | + C | t - s |^{2 \gamma} \llbracket A
    \rrbracket_{C^{\gamma}}  \llbracket x \rrbracket_{\gamma, \Delta} + | t -
    s |^{\gamma} \| h \|_{C^{\gamma}}\\
    & \leqslant | t - s |^{\gamma} (\llbracket A \rrbracket_{C^{\gamma}} \| x
    \|_{C^0} + \llbracket h \rrbracket_{C^{\gamma}}) + C \Delta^{\gamma}
    \llbracket A \rrbracket_{C^{\gamma}} \llbracket x \rrbracket_{\gamma,
    \Delta}
  \end{align*}
  and so dividing both sides by $| t - s |^{\gamma}$, taking the supremum over
  $s, t$ and choosing $\Delta$ such that $C \Delta^{\gamma} \llbracket A
  \rrbracket_{C^{\gamma}} \leqslant 1 / 2$ we obtain
  \begin{equation}
    \llbracket x \rrbracket_{\gamma, \Delta} \leqslant 2 (\llbracket A
    \rrbracket_{C^{\gamma}} \| x \|_{C^0} + \llbracket h
    \rrbracket_{C^{\gamma}}) . \label{appendixA1 linear YDE eq1}
  \end{equation}
  We now distinguish two cases. If $\llbracket A \rrbracket_{C^{\gamma}}$ is
  such that $(2 + C) \llbracket A \rrbracket_{C^{\gamma}} \leqslant 1 / 2$,
  then it follows from~{\eqref{appendixA1 linear YDE eq1}} with the choice
  $\Delta = 1$ and the trivial estimate $\| x \|_{C^0} \leqslant | x_0 | +
  \llbracket x \rrbracket_{\gamma}$ that
  \[ \llbracket x \rrbracket_{\gamma} \lesssim \llbracket A
     \rrbracket_{C^{\gamma}} | x_0 | + \llbracket h \rrbracket_{C^{\gamma}}
     \lesssim | x_0 | + \llbracket h \rrbracket_{C^{\gamma}} \]
  which immediately implies the conclusion. Suppose instead the opposite and
  choose $\Delta$ such that $1 / 4 \leqslant (2 + C) \Delta^{\gamma}
  \llbracket A \rrbracket_{C^{\gamma}} \leqslant 1 / 2$; define $I_n = [(n -
  1) \Delta, n \Delta]$, $J_n = \sup_{t \in I_n} | x_t |$, then estimates
  similar to the one done above show that
  \begin{align*}
    J_{n + 1} & \leqslant \, | x_{n \Delta} | + \Delta^{\gamma} \, \llbracket
    x \rrbracket_{C^{\gamma} (I_n)}\\
    & \leqslant \, | x_{n \Delta} | (1 + 2 \Delta^{\gamma} \llbracket A
    \rrbracket_{C^{\gamma}}) + 2 \llbracket h \rrbracket_{C^{\gamma}}\\
    & \lesssim \, J_n + \llbracket h \rrbracket_{C^{\gamma}}
  \end{align*}
  which implies recursively that for a suitable constant $C$ it holds $J_n
  \leqslant C^n (| x_0 | + \llbracket h \rrbracket_{C^{\gamma}})$. Since $n
  \sim \Delta^{- 1} \sim \llbracket A \rrbracket_{C^{\gamma}}^{1 / \gamma}$ we
  deduce that
  \[ \| x \|_{C^0} = \sup_n J_n \lesssim C^{\llbracket A \rrbracket^{1 /
     \gamma}_{C^{\gamma}}} (| x_0 | + \llbracket h \rrbracket_{C^{\gamma}}) \]
  which gives~{\eqref{appendixA1 bound2 linear YDE}}; this combined with
  $\Delta^{- \gamma} \sim \llbracket A \rrbracket_{C^{\gamma}}$,
  estimate~{\eqref{appendixA1 linear YDE eq1}} and the basic inequality
  \[ \llbracket x \rrbracket_{C^{\gamma}} \lesssim \Delta^{- \gamma} \| x
     \|_{C^0} + \llbracket x \rrbracket_{\gamma, \Delta} \]
  yields estimate~{\eqref{appendixA1 bound linear YDE}}.
\end{proof}

Similarly to the above lemma, we also have the following result.

\begin{lemma}
  \label{appendixA1 lemma bound nonlinear YDE}Let $A \in C^{\gamma}_t
  \tmop{Lip}_x$ such that $A (t, 0) = 0$ for all $t \geqslant 0$, $h \in
  C^{\gamma}_t$ and let $x$ be a solution of the nonlinear YDE
  \[ x_t = x_0 + \int_0^t A (\mathd s, x_s) + h_t . \]
  Then there exist suitable positive constants which only depend on $\gamma$
  such that
  \begin{equation}
    \| x \|_{C^{\gamma}} \lesssim e^{C \llbracket A \rrbracket_{C^{\gamma}}^{1
    / \gamma} T} (1 + \llbracket A \rrbracket_{C^{\gamma} \tmop{Lip}}) (| x_0
    + h_0 | + \llbracket h \rrbracket_{C^{\gamma}}) . \label{appendixA1 bound
    nonlinear YDE}
  \end{equation}
\end{lemma}

\begin{proof}
  Since $x$ is a solution to the nonlinear YDE, for any $s < t$ it holds
  \begin{align*}
    | x_{s, t} | & \leqslant \, | A_{s, t} (x_s) | + C | t - s |^{2 \gamma}
    \llbracket A \rrbracket_{C^{\gamma} \tmop{Lip}} \llbracket x
    \rrbracket_{C^{\gamma} ([s, t])} + | h_{s, t} |\\
    & \leqslant | t - s |^{\gamma} \llbracket A \rrbracket_{C^{\gamma}
    \tmop{Lip}} | x_s | + C | t - s |^{2 \gamma} \llbracket A
    \rrbracket_{C^{\gamma} \tmop{Lip}} \llbracket x \rrbracket_{C^{\gamma}
    ([s, t])} + | t - s |^{\gamma} \llbracket h \rrbracket_{C^{\gamma}}
  \end{align*}
  where in the second line we used the fact that $A (s, 0) = 0$ by hypothesis.
  The rest of the proof from here on is identical to the one of
  Lemma~\ref{appendixA1 lemma bound linear YDE} and we omit it. The
  inequality~{\eqref{appendixA1 bound nonlinear YDE}} is a combination of
  inequalities~{\eqref{appendixA1 bound linear YDE}} and~{\eqref{appendixA1
  bound2 linear YDE}}.
\end{proof}

We conclude this section by recalling the Sewing lemma, which is a fundamental
tool in the theory of rough paths. Consider an interval $[0, T]$ and a Banach
space $E$; let $\Delta_n$ denote the $n$\mbox{-}simplex on $[0, T]$, so that
$\Delta_n = \{ (t_1, \ldots, t_n) : 0 \leqslant t_1 \leqslant \ldots \leqslant
t_n \leqslant T \}$. Given a map $\Gamma : \Delta_2 \rightarrow E$, we define
$\delta \Gamma : \Delta_3 \rightarrow E$ by
\[ \delta \Gamma_{s, u, t} : = \, \Gamma_{s, t} - \Gamma_{s, u} - \Gamma_{u,
   t} . \]
We say that $\Gamma \in C^{\alpha, \beta}_2 ([0, T] ; E)$ if $\Gamma_{t, t} =
0$ for all $t \in [0, T]$ and $\| \Gamma \|_{\alpha, \beta} < \infty$, where
\[ \| \Gamma \|_{\alpha} \assign \sup_{s < t} \frac{\| \Gamma_{s, t} \|_E}{| t
   - s |^{\alpha}}, \quad \left\| \delta \, \Gamma \right\|_{\beta} \assign
   \sup_{s < u < t} \frac{\left\| \delta \, \Gamma_{s, u, t} \right\|_E}{| t -
   s |^{\beta}}, \quad \| \Gamma \|_{\alpha, \beta} \assign \| \Gamma
   \|_{\alpha} + \left\| \delta \, \Gamma \right\|_{\beta} . \]
Let us remark that for a map $f : [0, T] \rightarrow E$, we still denote by
$f_{s, t}$ the increment $f_t - f_s$.

\begin{lemma}[Sewing lemma]
  \label{appendixA1 sewing lemma}Let $\alpha$, $\beta$ be such that $0 <
  \alpha \leqslant 1 < \beta$. For any $\Gamma \in C^{\alpha, \beta}_2 ([0, T]
  ; E)$ there exists a unique map $\mathcal{I} \, \Gamma \in C^{\alpha} ([0,
  T] ; E)$ such that $\left( \mathcal{I} \, \Gamma \right)_0 = 0$ and
  \begin{equation}
    \left\| \left( \mathcal{I} \, \Gamma \right)_{s, t} - \Gamma_{s, t}
    \right\|_E \leqslant C \, \| \delta \Gamma \|_{\beta}  | t - s |^{\beta}
    \label{appendixA1 sewing property 1}
  \end{equation}
  where the constant $C$ only depends on $\beta$. In particular, the map
  $\mathcal{I} : C^{\alpha, \beta}_2 \rightarrow C^{\alpha}$ is linear and
  bounded and there exists a constant $C'$ which only depends on $\beta$ and
  $T$ such that
  \begin{equation}
    \left\| \mathcal{I} \, \Gamma \right\|_{C^{\alpha}} \leqslant C'  \|
    \Gamma \|_{\alpha, \beta} . \label{appendixA1 sewing property 2}
  \end{equation}
  For given $\Gamma$, the map $\mathcal{I} \, \Gamma$ is characterised as the
  unique limit of Riemann\mbox{-}Stjeltes sums: for any $t > 0$
  \[ \left( \mathcal{I} \, \Gamma \right)_t = \lim_{| \Pi | \rightarrow 0}
     \sum_i \Gamma_{t_i, t_{i + 1}} . \]
  The notation above means that for any sequence of partitions $\Pi_n = \{ 0 =
  t_0 < t_1 < \ldots < t_{k_n} = t \}$ with mesh $| \Pi_n | = \sup_{i = 1,
  \ldots, k_n} | t_i - t_{i - 1} | \rightarrow 0$ as $n \rightarrow \infty$,
  it holds
  \[ \left( \mathcal{I} \, \Gamma \right)_t = \lim_{n \rightarrow \infty}
     \sum_{i = 0}^{k_n - 1} \Gamma_{t_i, t_{i + 1}} . \]
\end{lemma}

For a proof, see Lemma~4.2 from~{\cite{frizhairer}}. Let us point out that
estimate~{\eqref{appendixA1 sewing property 1}} is extremely useful even in
cases even when $\mathcal{I} \, \Gamma$ is already known, as it asserts that
in order to control $\left\| \left( \mathcal{I} \, \Gamma \right)_{s, t} -
\Gamma_{s, t} \right\|$ it is enough to have an estimate for $\| \delta \Gamma
\|_{\beta}$.

\subsection{Function spaces}\label{appendixA2}

We recall here the definition and basic properties of the function spaces we
consider, which are Bessel potential spaces $L^{s, p} (\mathbb{R}^d)$ and
Besov spaces $B^s_{p, q} (\mathbb{R}^d)$. In particular, in view of
application to regularity estimates from Section~\ref{sec3.3}, we need
interpolation estimates and heat kernel estimates for such spaces. Bessel
potential spaces are a subclass of Triebel--Lizorkin spaces, which will be
also introduced. Most of the material is classical and covered in the
monographs~{\cite{bahouri}} and~{\cite{triebel2}}.

\begin{definition}
  \label{appendixA2 defn bessel}Let $s \geqslant 0$, we call Bessel potential
  and we denote it by $G^s$ the linear operator with Fourier symbol given by
  $(1 + | \xi |^2)^{- s / 2}$, with the convention that $G^0 = I$. For any $p
  \in (1, \infty)$, $G^s$ is a continuous embedding of $L^p$ into itself and
  it satisfies the semigroup property $G^t G^s = G^{t + s}$. For $p \in (1,
  \infty)$ and $s \geqslant 0$ we define the Bessel potential space $L^{s, p}$
  as $G^s (L^p)$ (with the convention $L^{0, p} = L^p$), endowed with the norm
  \[ \| f \|_{L^{s, p}} : = \| g \|_{L^p}  \text{ if } f = G^s g. \]
\end{definition}

It follows immediately from the definition and the semigroup property that
$G^t$ provides an isomorphism of $L^{s, p}$ and $L^{s + t, p}$ in the sense
that $\| G^t f \|_{L^{s + t, p}} = \| f \|_{_{L^{s, p}}}$. This allows also to
define $L^{s, p}$ for negative values of $s$, as the set of distributions $f$
such that $G^s f \in L^p$. Whenever $s = m$ integer, the space $L^{s, p}$
coincides with the classical Sobolev space $W^{m, p}$, with equivalent norm.
Similarly to Sobolev spaces, Bessel embeddings are available; in particular if
$s p > d$, we have the continuous embedding \ $L^{s, p} \hookrightarrow
C^{\gamma}$ with $\gamma = s - d / p$, whenever $\gamma$ is not an integer.

\begin{definition}
  \label{appendixA2 dyadic pair}Let $\mathcal{A}$ be the annulus $\bar{B}_{8 /
  3} \setminus B_{3 / 4}$. A dyadic pair is a couple of functions $(\chi,
  \varphi)$ such that $\chi \in C^{\infty}_c (B_{4 / 3})$, $\varphi \in
  C^{\infty}_c (\mathcal{A})$ and such that
  \[ \chi (\xi) + \sum_{j = 0}^{\infty} \varphi (2^{- j} \xi) = 1 \quad
     \forall \, \xi \in \mathbb{R}^d \]
  as well as
  \[ | j - j' | \geqslant 2 \, \Rightarrow \, \text{supp} \varphi (2^{- j}
     \cdot) \cap \text{supp} \varphi (2^{- j'} \cdot) = \emptyset . \]
  Given such a dyadic pair, we define the operator $\Delta_{- 1}$ by
  $\Delta_{- 1} f = \mathcal{F}^{- 1} (\chi \mathcal{F} f)$ and similarly
  $\Delta_j$ for $j \geqslant 0$ by $\Delta_j f = F^{- 1} (\varphi (2^{- j}
  \cdot) \mathcal{F} f)$.
\end{definition}

\begin{definition}
  \label{appendixA2 defn besov}For $s \in \mathbb{R}$, $(p, q) \in [1,
  \infty]^2$ we define the Besov space $B^s_{p, q}$ as the set of all tempered
  distributions $f$ such that
  \[ \| f \|_{B^s_{p, q}}^q \assign \sum_{j = 0}^{\infty} 2^{s \, j \, q} \|
     \Delta_j f \|_{L^p}^q < \infty . \]
\end{definition}

The spaces $B^s_{2, 2}$ coincide with the (fractional) Sobolev spaces $H^s$
which also coincide with $L^{s, 2}$; however, for $p \neq 2$ Bessel and Besov
spaces do not coincide. $B^s_{\infty, \infty}$ coincide with $C^s$ whenever
$s$ is not an integer. Also in the case of Besov spaces, embedding theorems
are available; in particular, $B^s_{p, q} \hookrightarrow B^{s - d /
p}_{\infty, \infty}$, which coincides with $C^{\gamma}$ whenever $\gamma = s -
d / p$ is not an integer. Let us also point out that the exponent $q$ is most
of the time not particularly relevant, as for any $\tilde{q} < q$ and any
$\varepsilon > 0$ we have the embeddings $B^s_{p, \tilde{q}} \hookrightarrow
B^s_{p, q} \hookrightarrow B^{s - \varepsilon}_{p, \tilde{q}}$.

\begin{definition}
  \label{appendixA2 defn triebel}For $s \in \mathbb{R}$, $(p, q) \in [1,
  \infty]^2$ we define the Triebel--Lizorkin space $F^s_{p, q}$ as the set of
  all tempered distributions $f$ such that
  \[ \| f \|_{F^s_{p, q}} : = \left\| \left( \sum_{j = 0}^{\infty} 2^{s j q} |
     \Delta_j f (\cdot) |^q \right)^{1 / q} \right\|_{L^p (\mathbb{R}^d)} <
     \infty . \]
\end{definition}

Both definitions of Besov and Triebel--Lizorkin spaces are independent of the
dyadic pair $(\chi, \varphi)$, in the sense that different pairs yields the
same space of distributions with equivalent norms. Bessel spaces $L^{s, p}$
correspond to $F^s_{p, 2}$; the spaces $F^s_{p, q}$ and $B^s_{p, q}$ coincide
if and only if $p = q$, in which case $F^s_{p, p} = B^s_{p, p} = W^{s, p}$ are
sometimes referred to as fractional Sobolev spaces, see~{\cite{dinezza}}. In
the case $p \neq q$, suitable embeddings between $F^s_{p, q}$ and $B^s_{p, q}$
follow immediately from Minkowski's inequality, since their norms can be
regarded respectively as $L^p (\mathbb{R}^d, \lambda ; \ell^q (\mathbb{N},
\mu))$\mbox{-} and $\ell^q (\mathbb{N}, \mu ; L^p (\mathbb{R}^d,
\lambda))$\mbox{-}norms, where $\lambda$ is the Lebesgue measure in
$\mathbb{R}^d$ and $\mu$ is the counting measure on $\mathbb{N}$. In
particular, for $p > q$ it holds $B^s_{p, q} \hookrightarrow F^s_{p, q}$ while
for $p < q$ we have the reversed embedding.

We now state a simple interpolation-like inequality for Bessel and Besov
spaces. Since we don't have a direct reference for this result, we also
provide a quick proof.

\begin{lemma}
  \label{appendixA2 interpolation inequality}Let $s > 0$, $p \in [2, \infty)$,
  then for any $\varepsilon > 0$ there exists a constant $c_{\varepsilon}$
  such that
  \[ \| f \|_{L^{s, p}} \leqslant c_{\varepsilon} \| f \|_{L^p}^{1 - \theta}
     \| f \|_{L^{s + \varepsilon, p}}^{\theta}, \enspace \tmop{where} \enspace
     \theta = \frac{s}{s + \varepsilon} . \]
  The same statement holds with the $L^{s, p}$ norm replaced by $B^s_{p, q}$.
\end{lemma}

\begin{proof}
  We use here the equivalent norm for $L^{s, p}$ given by $\| \cdot
  \|_{F^s_{p, 2}}$ as defined above. For any $N \in \mathbb{N}$ it holds
  \begin{eqnarray*}
    \| f \|_{F^s_{p, 2}} & = & \left\| \left( \sum_j 2^{2 s j} | \Delta_j f |
    \right)^{1 / 2} \right\|_{L^p}\\
    & \leqslant & \left\| \left( \sum_{j \leqslant N} 2^{2 s j} | \Delta_j f
    | \right)^{1 / 2} \right\|_{L^p} + \left\| \left( \sum_{j > N} 2^{2 s j} |
    \Delta_j f | \right)^{1 / 2} \right\|_{L^p}\\
    & \leqslant & 2^{s N} \left\| \left( \sum_{j \leqslant N} | \Delta_j f |
    \right)^{1 / 2} \right\|_{L^p} + 2^{- \varepsilon N} \left\| \left(
    \sum_{j > N} 2^{2 (s + \varepsilon) j} | \Delta_j f | \right)^{1 / 2}
    \right\|_{L^p}\\
    & \leqslant & 2^{s N} \| f \|_{L^p} + 2^{- \varepsilon N} \| f \|_{F^{s +
    \varepsilon}_{p, 2}} .
  \end{eqnarray*}
  Choosing suitable $N$ such that $2^{s N} \| f \|_{L^p} \sim 2^{- \varepsilon
  N} \| f \|_{F^{s + \varepsilon}_{p, 2}}$ we obtain the conclusion for $L^{s,
  p}$.
  
  An similar proof can be carried out for $B^s_{p, q}$; alternatively in this
  case one can use H{\"o}lder inequality as follows:
  \[ \| f \|^q_{B^s_{p, q}} = \sum_j 2^{s q j} \| \Delta_j f \|^q_{L^p}
     \leqslant \left( \sum_j 2^{s q j / \theta} \| \Delta_j f \|^q_{L^p}
     \right)^{\theta} \left( \sum_j \| \Delta_j f \|^q_{L^p} \right)^{1 -
     \theta} = \| f \|^{\theta q}_{B^{s / \theta}_{p, q}} \| f \|_{B^0_{p,
     q}}^{(1 - \theta) q} \]
  which gives the conclusion for the choice $\theta = s / (s + \varepsilon)$.
\end{proof}

We also need to recall the action of the heat flow $P_t$ on such spaces; with
a slight abuse of notation we will denote by $P_t$ both the convolution
operator and the Gaussian density itself.

\begin{lemma}
  \label{appendixA2 heat kernel estimates}For any $s \in \mathbb{R}$, $\rho >
  0$, $p \in (1, \infty)$ and for any $f \in L^{s, p}$, $t > 0$ it holds
  \[ \| P_t f \|_{L^{s + \rho, p}} \lesssim t^{- \rho / 2} \| f \|_{L^{s, p}}
     . \]
  Similarly, for any $s \in \mathbb{R}$, $\rho > 0$, $p, q \in [1, \infty]$
  and for any $f \in B^s_{p, q}$, $t > 0$ it holds
  \[ \| P_t f \|_{B^{s + \rho}_{p, q}} \lesssim t^{- \rho / 2} \| f
     \|_{B^s_{p, q}} . \]
\end{lemma}

Both statements are classical, the first one following immediately from the
fact that, due to the scaling $P_t = t^{- d / 2} P_1 (t^{- 1 / 2} \cdot)$
which implies $\| P_t \|_{L^{\rho, 1}} = t^{- \rho / 2} \| P_1 \|_{L^{\rho,
1}}$; see Proposition~5 at page~2414 of~{\cite{mourratweber}}, for a proof in
a more general context of the second statement.

\subsection{A primer on stochastic integration in UMD Banach
spaces}\label{appendixA3}

In this appendix we recall several results on abstract stochastic integration
which are needed in order to complete the proof of Theorem~\ref{sec3.2
functional ito tanaka}; we believe they are also of independent interest and
therefore provide a general presentation. In view of application to
Section~\ref{sec3.3} we only need results for martingale type~2 spaces, which
however yield the restriction to work on $L^p$-based spaces with $p \geqslant
2$; weakening this condition to the case $p = 1$ would highly enhance the
results, as discussed in Remark~\ref{sec3.3 remark limitation}, which is why
in this appendix we also discuss UMD Banach spaces. Even with this more
general theory we are currently not able to overcome the obstacle, we believe
it might be of help for future developments and improvements.

All the material presented here is taken
from~{\cite{vanneerven2007}},~{\cite{vanneerven2015}}. Also, we restrict for
simplicity to the case $W$ is a real valued Brownian motion (the extension to
the vector valued case $W \in \mathbb{R}^d$ being straightforward) but the
theory is far more general as it considers the case of $H$\mbox{-}cylindrical
Brownian motion, $H$ being an abstract Hilbert space. This gives rise to
$\gamma$\mbox{-}Radonifying norms $\gamma (H, E)$; in our simple setting, $H
=\mathbb{R}$, for any Banach space $E$ it holds $\| \cdot \|_{\gamma (H, E)} =
\| \cdot \|_E$.

\begin{definition}
  \label{appendixA3 defn martingale type}Let $p \in [1, 2]$. A Banach space
  $E$ has martingale type~p if there exists a constant $C \geqslant 0$ such
  that for all finite $E$\mbox{-}valued martingale difference sequences
  $(d_n)_{n = 1}^N$ it holds
  \[ \mathbb{E} \left[ \left\| \sum^N_{n = 1} d_n \right\|_E^p \right]
     \leqslant C^p \sum_{n = 1}^N \mathbb{E} [\| d_n \|_E^p] . \]
  The least admissible constant is denoted by $C_{p, E}$.
\end{definition}

Examples of martingale type spaces are the following:
\begin{itemize}
  \item Every Banach space has martingale type~1.
  
  \item Every Hilbert space has martingale type~2.
  
  \item A closed subspace of a Banach space of martingale type~$p$ has still
  martingale type~$p$.
  
  \item If $E$ has martingale type~$p$ and $(S, \mathcal{A}, \mu)$ is a
  measure space, then $L^r (S ; E)$ with $r \in [1, \infty)$ has martingale
  type~$p \wedge r$; in particular Lebesgue spaces $L^p (\mathbb{R}^d)$ have
  martingale type~$p \wedge 2$.
  
  \item Let $(E_0, E_1)$ be an interpolation couple such that $E_i$ has
  martingale type~$p_i \in [1, 2]$, let $\theta \in (0, 1)$ and consider $p
  \in [1, 2]$ such that $1 / p = (1 - \theta) / p_0 + \theta / p_1$. Then both
  the complex and real interpolation spaces $E_{\theta}$ and
  $\tilde{E}_{\theta}$ have martingale type~$p$.
\end{itemize}
For the last two examples see Propositions~7.1.3 and~7.1.4
from~{\cite{hytonen}}. It follows from the previous list of examples that
Sobolev spaces $W^{k, p} (\mathbb{R}^d)$ with $p \in [2, \infty)$ have
martingale type as they can be identified with closed subspaces of $L^p
(\mathbb{R}^d)^{\otimes n}$ for suitable $n$; Bessel potential spaces $L^{s,
p} (\mathbb{R}^d)$ with general $s$ are isomorphic to $L^p (\mathbb{R}^d)$,
with isomorphism given by $G^s = (1 - \Delta)^{s / 2}$, therefore for $p \in
[2, \infty)$ they have martingale type~2. In the case of Besov spaces $B^s_{p,
q}$ with $p, q \in [2, \infty)$, again it can be shown that they have
martingale type~2, either by constructing them as interpolation spaces (see
for instance Section~17.3 from~{\cite{leoni}}) or reasoning as follows: by
definition, any $\varphi \in B^s_{p, q}$ can be identified with a sequence $\{
\Delta_j \varphi \}_j \subset L^p (\mathbb{R}^d)$ with suitable summability,
namely such that it belongs to $\ell^q (\mathbb{N}, \mu ; L^p
(\mathbb{R}^d))$, where $\mu (\{ j \}) = 2^{- s q j}$; in the case $p, q \in
[2, \infty)$ by the previous examples it has martingale type~2.

\

Now let $W$ be a real valued $\mathcal{F}_t$\mbox{-}Brownian motion on a
filtered probability space $(\Omega, \mathcal{F}, \{ \mathcal{F}_t \}_{t
\geqslant 0}, \mathbb{P})$, $\{ \mathcal{F}_t \}_{t \geqslant 0}$ being a
filtration satisfying the usual conditions. For martingale type~2 spaces it is
possible to define stochastic integrals analogously to the standard case: for
an adapted elementary process $\phi : \mathbb{R}_+ \times \Omega \rightarrow
E$, namely of the form
\[ \phi (t, \omega) = \sum_{i = 1}^{n - 1} x_i  \mathbbm{1}_{(t_i, t_{i + 1}]
   \times F_i} (t, \omega) \]
where $0 \leqslant t_1 < t_2 < \cdots < t_n$, $x_i \in E$, $F_i \in
\mathcal{F}_{t_i}$, we set
\[ \int_0^{\cdot} \phi \mathd W : = \sum_{i = 1}^{n - 1} x_i 
   \mathbbm{1}_{F_i} (W_{\cdot \wedge t_{i + 1}} - W_{\cdot \wedge t_i}) . \]
Using the martingale type~2 property it is then possible to show that the
$L^2$~norm of the process defined in this way is controlled by $\| \phi
\|_{L^2 (\mathbb{R}_+ \times \Omega, E)}$, see Theorem~4.6
from~{\cite{vanneerven2015}}. By standard approximation procedures, together
with Doob's maximal inequality, the following analogue of standard It{\^o}
integration can then be proven.

\begin{theorem}
  \label{apendixA3 thm stoch integral}Let $\phi : \mathbb{R}_+ \times \Omega
  \rightarrow E$ be a progressively measurable process satisfying
  \[ \| \phi \|_{L^2 (\mathbb{R}_+ \times \Omega, E)}^2 =\mathbb{E} \left[
     \int_0^{+ \infty} \| \phi_t \|_E^2 \mathd t \right] < \infty . \]
  Then $\int \phi \tmop{dW}$ is well defined as an $E$\mbox{-}valued
  martingale with paths in $C_b (\mathbb{R}_+ ; E)$ and satisfies
  \begin{equation}
    \mathbb{E} \left[ \sup_{t \geqslant 0} \left\| \int_0^t \phi_s \mathd W_s
    \right\|_E^2 \right] \leqslant 4 C_{2, E}^2 \mathbb{E} \left[ \int_0^{+
    \infty} \| \phi_t \|_E^2 \mathd t \right] . \label{appendixA3 ito}
  \end{equation}
\end{theorem}

Let us also remark that it follows immediately from the definition for simple
processes and the usual approximation procedure that, for any $\phi$ as above
and any deterministic $\varphi^{\ast} \in E^{\ast}$, the following identity
holds
\begin{equation}
  \langle \varphi^{\ast}, \int_0^{\cdot} \phi_s \mathd W_s \rangle =
  \int_0^{\cdot} \langle \varphi^{\ast}, \phi_s \rangle \mathd W_s
  \label{appendixA3 identity dual}
\end{equation}
where the integral on the r.h.s. is a standard real valued stochastic
integral.

We are now ready to complete the proof of Theorem~\ref{sec3.2 functional ito
tanaka}.

\begin{proof}[of Theorem~\ref{sec3.2 functional ito tanaka}]
  Let us first show the following general fact: given a separable Banach space
  $E$ and two $E$\mbox{-}valued random variables $X$ and $Y$ such that for any
  $\varphi^{\ast}$ in a linearly dense subspace of $E^{\ast}$ it holds
  \begin{equation}
    \langle \varphi^{\ast}, X \rangle = \langle \varphi^{\ast}, Y \rangle
    \quad \mathbb{P} \text{-a.s.} \label{appendixA3 eq1}
  \end{equation}
  then necessarily $X = Y$ $\mathbb{P}$-a.s. Indeed, it follows from the
  linear density assumption that relation~{\eqref{appendixA3 eq1}} holds for
  any $\varphi^{\ast} \in E^{\ast}$; by separability of $E$ and Hahn-Banach
  Theorem, it is possible to find a countable collection $\{ \varphi_n^{\ast}
  \}_n \subset E^{\ast}$ such that $\| \varphi^{\ast}_n \|_{E^{\ast}} = 1$ for
  all $n$ and
  \[ \| x \| = \sup_n | \langle \varphi^{\ast}_n, x \rangle | \quad \forall \,
     x \in E. \]
  By~{\eqref{appendixA3 eq1}} and the fact that the supremum is over a
  countable set, we can find a set $\Gamma$ of full probability such that
  \[ \| X - Y \|_E = \sup_n | \langle \varphi_n^{\ast}, X - Y \rangle | = 0
     \quad \forall \, \omega \in \Gamma \]
  which proves the claim. Now let $b \in C^{\infty}_c ([0, T] \times
  \mathbb{R}^d ; \mathbb{R})$, so that it can be identified with an element of
  $L^2 (0, T ; H^{\alpha} (\mathbb{R}^d))$ for any $\alpha > 0$; choose
  $\alpha$ big enough so that $H^{\alpha}$ embeds into continuous functions
  vanishing at infinity. Then thanks to relation~{\eqref{appendixA3 identity
  dual}}, equation~{\eqref{sec3.2 preliminary ito tanaka}} can be written as:
  for a given $x \in \mathbb{R}^d$, $\mathbb{P}$-a.s. it holds
  
  \begin{align*}
    \langle \delta_x, \int_s^t b \left( r, \cdot \, + W^H_r \right) \, \mathd
    r \rangle & = \tmcolor{black}{\langle \delta_x, \int_s^t P_{\tilde{c}_H |
    r - s |^{2 H}} \, b \left( r, \cdot \, + W^{2, H}_{s, r} \right) \, \mathd
    r \rangle}\\
    & \enspace + \langle \delta_x, \int_s^t \int_u^t P_{\tilde{c}_H | r - u
    |^{2 H}} \nabla b \left( r, \cdot \, + W^{2, H}_{u, r} \right) c_H | r - u
    |^{H - 1 / 2} \, \mathd r \, \cdot \mathd B_u \rangle
  \end{align*}
  
  where the first two integrals are interpreted as (random) Bochner integrals
  while the last one as a stochastic integral in $H^{\alpha}$ (with the inner
  integral being a random Bochner integral); integrability and predictability
  are straightforward due to the regularity of $b$ and the properties of
  $W^{2, H}_{u, r}$. Finally, as the collection $\{ \delta_x \}_{x \in
  \mathbb{R}^d}$ is linearly dense in $H^{- \alpha}$ and $H^{\alpha}$ is
  separable, we can apply the general fact above to deduce that, for $s < t$
  fixed, the random variables above coincide on a set of full probability,
  without the need of testing against $\delta_x$. This is exactly
  formula~{\eqref{sec3.2 ito tanaka formula}}.
\end{proof}

In the setting of martingale type~2 spaces a one-sided Burkholder's inequality
is available; we state it with the optimal asymptotic behaviour of the
constants, which is needed in the estimates in Section~\ref{sec3.3}. It was
first shown by Seidler in~{\cite{seidler}}.

\begin{theorem}[Theorem 4.7 from {\cite{vanneerven2015}}]
  \label{appendixA3 thm burkholder}Let $E$ be martingale type~2. Then for any
  progressively measurable process $\phi : \mathbb{R}_+ \times \Omega
  \rightarrow E$ and $p \in (0, \infty)$ there exists a constant
  $\tilde{C}_{p, E}$ such that
  \begin{equation}
    \mathbb{E} \left[ \sup_{t \geqslant 0} \left\| \int_0^t \phi_s \mathd W_s
    \right\|_E^2 \right] \leqslant \tilde{C}_{p, E}^p \mathbb{E} \left[ \left(
    \int_0^{\infty} \| \phi_s \|^2_E \mathd s \right)^{p / 2} \right] .
    \label{appendixA3 burkholder}
  \end{equation}
  In particular, it is possible to choose $\tilde{C}_{p, E}$ such that
  $\tilde{C}_{p, E} \leqslant C_E \sqrt{p}$ for any $p \geqslant 2$, where
  $C_E$ is a universal constant that only depends on the space $E$.
\end{theorem}

This concludes the exposition of results needed in the proofs of this work. In
the rest of this appendix, we present a brief account on stochastic
integration in UMD Banach spaces.

\

Some of the major drawbacks of martingale type~2 spaces are the fact that
they do not include $L^p$ spaces with $p < 2$, Burkholder's inequality is in
general only one-sided and it is not sharp, which is troublesome in
applications to maximal regularity of mild solutions of SPDEs. This motivates
the introduction of a larger class of spaces. As before, we only consider the
case of a real valued $W$, but the theory extends to $W$ being a cylindrical
$H$\mbox{-}Brownian motion for an Hilbert space $H$.

\begin{definition}
  \label{appendixA3 defn UMD} A Banach space $E$ is called a UMD space (i.e.
  it has unconditional martingale differences) for some $p \in (1, \infty)$ if
  there exists a constant $\beta \geqslant 0$ such that for all
  $E$\mbox{-}valued $L^p$\mbox{-}martingale differences $(d_n)_{n \geqslant
  1}$ and signs $(\varepsilon_n)_{n \geqslant 1}$ one has
  \[ \mathbb{E} \left[ \left\| \sum_{n = 1}^d \varepsilon_n d_n \right\|_E^p
     \right] \leqslant \beta^p \mathbb{E} \left[ \left\| \sum_{n = 1}^N d_n
     \right\|_E^p \right] \quad \forall N \geqslant 1. \]
  The least admissible constant is denoted by $\beta_{p, E}$.
\end{definition}

It can be shown that if $E$ is UMD for some $p \in (1, \infty)$, then it is
actually UMD for all $p \in (1, \infty)$. Examples are the following (here $p'
\in (1, \infty)$ denotes the conjugate of $p$):
\begin{itemize}
  \item Every Hilbert space $H$ is UMD with $\beta_{p, H} = \max \{ p, p' \}$.
  
  \item If $E$ is a UMD Banach space and $(S, \mathcal{A}, \mu)$ is measure
  space, then $L^p (\mu ; E)$ is a UMD space with $\beta_{p, L^p (\mu ; E)} =
  \beta_{p, E}$.
  
  \item $E$ is UMD if and only if $E^{\ast}$ is UMD and it holds $\beta_{p, E}
  = \beta_{p', E^{\ast}}$.
\end{itemize}
In the case of UMD spaces, it is possible again to construct stochastic
integrals in a suitable class of predictable processes and to obtain two-sided
Burkholder inequalities.

\begin{theorem}[Theorem 5.5 from {\cite{vanneerven2015}}]
  \label{appendixA3 burkholder UMD}Let $E$ be a UMD Banach space and let $p
  \in (1, \infty)$. For all progressively measurable processes $\phi :
  \mathbb{R}_+ \times \Omega \rightarrow E$ we have
  \[ \frac{1}{\beta_{p, E}} \| \phi \|_{L^p (\Omega ; \gamma^p (L^2
     (\mathbb{R}_+), X))} \leqslant \mathbb{E} \left[ \left\| \int_0^{\infty}
     \phi \mathd W \right\|^p \right]^{1 / p} \leqslant \beta_{p, E} \| \phi
     \|_{L^p (\Omega ; \gamma^p (L^2 (\mathbb{R}_+), X))} . \]
\end{theorem}

In the above statement, $\gamma^p (L^2 (\mathbb{R}_+), X)$ stands for the
$p$\mbox{-}th $\gamma$\mbox{-}Radonifying norm; we omit the precise
definition, which can be found
in~{\cite{vanneerven2007}},~{\cite{vanneerven2015}}. There are special cases
in which the $\gamma$-Radonifying norm is equivalent to other norms with a
simpler expression, in particular when $E = L^q (\mu)$, in which case there is
an isomorphism of Banach spaces
\[ \gamma^p (L^2 (\mathbb{R}_+), L^p (\mu)) = L^p (\mu ; L^2 (\mathbb{R}_+))
\]
and so the previous inequality can be reformulated as
\begin{eqnarray*}
  \mathbb{E} \left[ \left\| \int_0^{\infty} \phi_s \mathd W_s \right\|_{L^q
  (\mu)}^p \right] & \sim_{p, q} & \mathbb{E} \left[ \left\| \left(
  \int_0^{\infty} \phi_s^2 (\cdot) \mathd s \right)^{1 / 2} \right\|_{L^q
  (\mu)}^p \right]\\
  & = & \mathbb{E} \left[ \left( \int_S \left( \int_0^{\infty} \phi^2 (s, x)
  \mathd s \right)^{q / 2} \mathd \mu (x) \right)^{p / q} \right] .
\end{eqnarray*}
In the case $q \geqslant 2$, an application of Minkowski's inequality then
yields
\[ \mathbb{E} \left[ \left\| \int_0^{\infty} \phi_s \mathd W_s \right\|_{L^q
   (\mu)}^p \right] \lesssim_{p, q} \mathbb{E} \left[ \left( \int_0^{\infty}
   \| \phi (s, \cdot) \|_{L^q (\mu)}^2 \mathd s \right)^{p / 2} \right] \]
which is consistent with the aforementioned results for martingale type~2
spaces. In the general case instead, assuming we want to estimate the $L^q
(\mathbb{R}^d)$~norm of an averaged operator by means of the It\^{o}--Tanaka
formula~{\eqref{sec3.2 ito tanaka formula}}, we would then need to estimate a
term of the form (we omit the constants for simplicity)
\[ \mathbb{E} \left[ \left( \int_{\mathbb{R}^d} \left( \int_s^t \left(
   \int_u^t P_{| r - u |^{2 H}} \nabla b (r, x + W^{2, H}_{u, r})  | r - u
   |^{H - 1 / 2} \mathd r \right)^2 \mathd u \right)^{q / 2} \mathd x
   \right)^{p / q} \right] \]
which we are currently not able to do. The techniques employed in
Section~\ref{sec3.3} rely quite crucially on the simplifications given by a
formula of the form~{\eqref{appendixA3 burkholder}}.

\

\end{document}